\documentclass[12pt]{amsart}
\usepackage{amsmath, mathabx}
\usepackage{amsxtra}
\usepackage{amscd}
\usepackage{amsthm}
\usepackage{amsfonts}
\usepackage{amssymb}
\usepackage{mathrsfs}
\usepackage{mathbbol}

\usepackage {pstricks}
\usepackage {tikz}
\usepackage{pstricks,pst-node}
\usepackage[all]{xy}
\usepackage{mathbbol}

\usepackage{verbatim}

\usepackage{dashundergaps}
\usepackage{tikz-cd}
\usepackage{enumerate}
\usepackage{graphicx}

\usepackage{epigraph}
\usepackage{physics}

\usepackage{ytableau}

\DeclareRobustCommand{\loongrightarrow}{%
  \DOTSB\relbar\joinrel\relbar\joinrel\relbar\joinrel\rightarrow
}

\DeclareRobustCommand{\looongrightarrow}{%
  \DOTSB\relbar\joinrel\relbar\joinrel\relbar\joinrel\relbar\joinrel\joinrel\relbar\relbar\joinrel\relbar\joinrel\rightarrow
}

\newtheorem{theorem}{Theorem}[section]
\newtheorem{assertion}[theorem]{Assertion}
\newtheorem{lemma}[theorem]{Lemma}

\newtheorem{corollary}[theorem]{Corollary}

\newtheorem{scholium}[theorem]{Scholium}

\theoremstyle{definition}
\newtheorem{definition}[theorem]{Definition}
\newtheorem{example}[theorem]{Example}

\newtheorem{assumption}[theorem]{Antipode Conditions}

\newtheorem{remark}[theorem]{Remark}

\newtheorem{warning}[theorem]{Warning}

\numberwithin{equation}{subsection}
\numberwithin{figure}{section}

\newcommand{\mc}{\mathcal}

\newcommand{\be}{\begin{equation}}
\newcommand{\ee}{\end{equation}}

\newcommand{\A}{{\mathbb A}}

\newcommand{\C}{{\mathbb C}}
\newcommand{\R}{{\mathbb R}}
\newcommand{\Z}{{\mathbb Z}}

\newcommand{\I}{{\mathbb I}}
\newcommand{\J}{{\mathbb J}}

\newcommand{\BE}{{\mathbb E}}

\newcommand{\BD}{{\mathbb D}}

\newcommand{\CA}{{\mathcal A}}
\newcommand{\CB}{{\mathcal B}}
\newcommand{\CC}{{\mathcal C}}

\newcommand{\CF}{{\mathcal F}}

\newcommand{\CO}{{\mathcal O}}

\newcommand{\CU}{\mc U}

\newcommand{\CW}{{\mathcal W}}

\newcommand{\SF}{{\mathscr F}}

\newcommand{\SE}{{\mathscr E}}

\newcommand{\SP}{{\mathscr P}}

\newcommand{\ST}{{\mathscr T}}

\newcommand{\mf}{\mathfrak}
\newcommand{\g}{{\mathfrak g}}
\newcommand{\fg}{{\mf g}}
\newcommand{\fh}{{\mf h}}
\newcommand{\fb}{{\mf b}}
\newcommand{\fn}{{\mf n}}
\newcommand{\fk}{{\mathfrak k}}
\newcommand{\fl}{{\mathfrak l}}

\newcommand{\fp}{{\mathfrak p}}
\newcommand{\fq}{{\mathfrak q}}

\newcommand{\fu}{{\mathfrak u}}
\newcommand{\fv}{{\mathfrak v}}

\newcommand{\fsl}{{\mathfrak {sl}}}

\newcommand{\gl}{{\mathfrak {gl}}}
\newcommand{\osp}{{\mathfrak {osp}}}

\newcommand{\id}{{\rm{id}}}

\newcommand{\U}{{\rm{U}}}
\newcommand{\End}{{\rm{End}}}
\newcommand{\Hom}{{\rm{Hom}}}

\newcommand{\GL}{{\rm{GL}}}

\newcommand{\Sym}{{\rm{Sym}}}

\newcommand{\im}{{\rm{Im}}}

\newcommand{\wt}{\widetilde}

\newcommand{\ol}{\overline}

\newcommand{\ot}{\otimes}

\newcommand{\beq}{\begin{eqnarray}}
\newcommand{\eeq}{\end{eqnarray}}

\newcommand{\baln}{\begin{aligned}}
\newcommand{\ealn}{\end{aligned}}

\newcommand{\lra}{\longrightarrow}

\newcommand{\CS}{\mathcal{S}}

\newcommand{\Px}{\C_{\omega}[{\bf x}]}
\newcommand{\Pd}{\C_{\omega}[{\boldsymbol{\partial}}]}

\newcommand{\Vect}{\text{${\bf Vect}(\Gamma, \omega)$}}
\newcommand{\Set}{\text{\bf Set}}
\newcommand{\Alg}{\text{${\bf Alg}(\Gamma, \omega)$}}

\begin{document}

\normalfont


\title[Lie $(\Gamma, \omega)$-algebras]
{Invariants and representations of \\ the $\Gamma$-graded general linear Lie $\omega$-algebras}

\author{R.B. Zhang}
\address{School of Mathematics and Statistics,
The University of Sydney, Sydney, N.S.W. 2006, Australia}
\email{ruibin.zhang@sydney.edu.au}
\date{\today}

\begin{abstract}
There is considerable current interest in applications of
generalised Lie algebras graded by an abelian group $\Gamma$ with a commutative factor $\omega$. 
This calls for a systematic development of the theory of such algebraic structures. 
We treat the representation theory and  invariant theory of 
%
%
the $\Gamma$-graded general linear Lie $\omega$-algebra 
$\gl(V(\Gamma, \omega))$, where $V(\Gamma, \omega)$ is 
any finite dimensional $\Gamma$-graded vector space.  
Generalised Howe dualities over symmetric $(\Gamma, \omega)$-algebras are established, 
from which we derive the first and second fundamental theorems of invariant theory, 
and a generalised Schur-Weyl duality. 
The unitarisable $\gl(V(\Gamma, \omega))$-modules for 
two ``compact'' $\ast$-structures are classified, and it is shown that  
the tensor powers of $V(\Gamma, \omega)$ and their duals
are unitarisable for the two compact $\ast$-structures  respectively.   
A  Hopf $(\Gamma, \omega)$-algebra is constructed, which gives rise to a group functor 
corresponding to the general linear group in the $\Gamma$-graded setting. 
Using this Hopf $(\Gamma, \omega)$-algebra, 
we realise simple tensor modules and their dual modules 
by mimicking the classic Borel-Weil theorem. 
We also analyse in some detail the case with $\Gamma=\Z^{\dim{V(\Gamma, \omega)}}$ 
and $\omega$ depending on a complex parameter $q\ne 0$,
where $\gl(V(\Gamma, \omega))$ shares common features with the quantum general linear (super)group, 
but is better behaved especially when $q$ is a root of unity. 

\end{abstract}

\maketitle

\tableofcontents

\section{Introduction}
\noindent{1.1.}
Given any additive abelian group $\Gamma$ with a commutative factor $\omega$ \cite{B}, 
Rittenberg and Wyler \cite{RW, RW-2} introduced $\Gamma$-graded Lie $\omega$-algebras,   
or Lie $(\Gamma, \omega)$-algebras for short, 
in the late 70s
following the advent of supersymmetry (see \cite{HLS, SS, WZu} and references therein)
and Lie superalgebras (see \cite{K, Sch} and references therein).    
Initial investigations on aspects of their structure theory \cite{Sch79, Sch83} and their applications to 
generalised supersymmetries \cite{JYW, LR, V} were carried out immediately after. 
It was also known \cite{RW-2} at the time already that these algebraic structures 
were relevant to Green's theory of parastatistics \cite{H}.  
There was further research on applications  in these areas \cite{JYW, T, V, YJ} in the ensuing years. 

Lie $(\Gamma, \omega)$-algebras are also loosely referred to as Lie colour (super)algebras 
(see Section \ref{sect:Lie} for further explanation of terminology). 
In the late 90s,  they were brought into play in various areas of algebra, e.g, 
the study of ring theoretical aspects of
graded associative and Hopf algebras \cite{BFM, BFM-2, Mo}, 
and group gradings of Lie algebras \cite{BK}.
There was also work studying Lie colour (super)algebras themselves, 
for example, 
a double centraliser theorem between the symmetric group 
and general linear Lie colour (super)algebra was proved in \cite{FM}, 
and a cohomology theory for Lie colour (super)algebras was 
introduced in \cite{SchZ} which was much studied later by many researchers. 
Since then there have been considerable research activities in mathematics 
studying applications of Lie colour (super)algebras in other algebraic contexts. 

There was also some research in different directions.  We mention in particular that 
the paper \cite{CSO} 
gave a  classification and construction of the finite dimensional simple modules for 
the Lie colour algebra $\fsl_2^c$.    
It revealed unexpected complexity of the representation theory of 
Lie colour algebras even in this simplest case 
(also see \cite{Rm} for further developments). 
Also
the paper \cite{F} explored general aspects of representations 
of generic Lie colour (super)algebras from a ring theoretical perspective. 

Starting about 10 years ago, several research groups in mathematical physics 
have worked intensively on the subject. 
A Lie colour (super)algebraic framework was established
for parastatistics \cite{BFRT, SV18, SV23, T, Tf1, Tf2, YJ} (also see \cite{Z24}), 
and measurable effects of $\Z_2\times\Z_2$-graded paraparticles \cite{Tf1, Tf2},  
and more generally of $\Z_2^n$-graded paraparticles \cite{BFRT}, 
were discovered theoretically (see \cite{Tf3} for a brief review and further references).
Colour supersymmetries were developed in quantum mechanics \cite{AD, AIKT, AKT1, AKT2, BFRT, BD}  
and quantum field theory \cite{Bruce, DA}, and an 
important result obtained was the construction of $\Z_2\times\Z_2$-graded extensions 
of the Liouville and Sinh–Gordon theories \cite{AIKT}, which retain integrability. 
Another interesting discovery 
\cite{AKTT} was the $\Z_2\times \Z_2$-graded Lie algebra symmetries of
the non-relativistic analogues of the Dirac equation known as L\'evy-Leblond equations (see  
\cite{Rm} and references therein for further developments).

There have also been works developing 
aspects of Lie colour (super)algebras
for particular grading groups like $\Gamma=\Z_2\times\Z_2$  
with commutative factor $\omega(\alpha, \beta)= (-1)^{\alpha_1\beta_2\pm \alpha_2 \beta_1}$ 
(which is the same for both $\pm$), 
which are directly relevant to physical applications, see \cite{AI, AA, ISV, Rm25, SV23, SV25b} and references therein.
Affine Lie colour (super)algebras have also been studied  \cite{AIS, AS}.

These new results have aroused considerable interest 
in  Lie colour (super)algebras and related algebraic structures in mathematical physics,   
and call for a systematic development of their theory, especially those aspects 
relevant to physical applications.
Such a theory will provide necessary tools for further exploration of  physical applications of Lie colour (super)algebras, 
which continues to be a very active research area today, 
see, e.g., \cite{AIT, AIKT, BFRT, SV25a, SV25b} and references therein.
 
\medskip
\noindent{1.2.} 
 A well known theorem of Scheunert \cite[Theorem 2]{Sch79} from the early days
states that cocycle twistings can turn Lie colour (super)algebras into ordinary Lie (super)algebras, 
and this carries over to representations as well (also see \cite{MB}). 
A vast generalisation of Scheunert's cocycle twistings, referred to as bosonisation,  
was developed by Majid  \cite{Ma}
and others for algebras over braided tensor categories. 
A particularly interesting example is the quantum correspondences among 
(affine) Lie superalgebras developed in \cite{XZ, Z92, Z97}, 
which proved to be very useful for studying primitive ideals 
of universal enveloping algebras of Lie superalgebras \cite{Le}, 
and constructing vertex operator representations for 
quantum affine superalgebras \cite{XZ26}. 

This might mislead some people to ask whether there is anything new in 
Lie colour (super)algebras, and even Lie superalgebras and 
Hopf superalgebras like quantum supergroups \cite{BGZ, Y91, Y94, Z93, Z98}. 
The answer is of course affirmative. 
In particular, recent research discussed above 
clearly demonstrates the usefulness of Lie colour (super)algebras
for applications in mathematical physics.

As is well known, in a quantum system with symmetry, 
the states are classified, and  their properties (e.g., statistics)  
are determined, by representations of the symmetry algebra. 
If the symmetry algebra is a  Lie colour superalgebra, cocycle twistings of the states 
and the symmetry algebra will alter properties of the states.

A mathematical example due to Montgomery  \cite[\S 6]{Mo} is also noteworthy in this regard. 
The  $n$-th Weyl algebra $\A_n$ was graded by $\Z_2^n$, leading to a 
Lie colour superalgebra $L=\A_n$ with the graded commutator as the Lie bracket.  
By applying Scheunert's cocycle twisting,  this $\Z_2^n$-graded $\A_n$ 
was turned into an associative superalgebra graded by even and odd degrees, 
which yielded an ordinary Lie superalgebra. However, 
this Lie superalgebra could not be obtained from $L$ by cocycle twisting. 
We note that the Weyl algebra setting of Montgomery's example  loc. cit.
is quite similar to that of a quantum system; 
it suggests that simultaneous cocycle twisting of the states 
and the  symmetry colour algebra in a compatible manner 
may not always be achievable.

\medskip
\noindent{1.3.}
From the viewpoints of both Lie theory and mathematical physics,    
we are most interested in those Lie colour (super)algebras 
which are akin to simple Lie (super)algebras 
and the corresponding affine Kac-Moody Lie (super)algebras.  
The present paper studies the colour analogue of the type $A$ case. 
Given any finite dimensional $\Gamma$-graded vector space $V(\Gamma, \omega)$,   
we investigate the general linear 
Lie $(\Gamma, \omega)$-algebra $\gl(V(\Gamma, \omega))$, 
to develop a representation theory and classical invariant theory  
analogous to those  treated in \cite{GW, Ho, W} for the classical Lie algebras.  

It appears that there exists no coherent treatment 
of the fundamentals of the structure and representation theory
of $\gl(V(\Gamma, \omega))$ in the literature. Thus
this paper will start by developing that,  then progresses to addressing more advanced topics
in the classical invariant theory and representation theory. 
These include colour Howe duality, generalised Schur-Weyl duality, 
fundamental theorems of invariant theory, unitarisable modules for important $\ast$-structures, 
and a non-commutative geometric construction of representations 
using the algebra of functions of the ``general linear group'' 
in the $\Gamma$-graded setting (see Section \ref{sect:group}).

It is pleasing that 
much of the machinery of usual Lie theory can be generalised to the graded setting 
for arbitrary grading groups and commutative factors, as we will see 
in the main body of the paper.

\smallskip

This paper consists of four parts. The following main results are obtained in each of them.

\smallskip
\noindent
I. Basic structure and representation theory. 
\begin{enumerate}
\item \label{aim-1}
Develop the structure of $\gl(V(\Gamma, \omega))$, in particular, its root system and Weyl group;

\item \label{aim-2}
investigate highest weight modules, and classify the finite dimensional simple modules; 

\item 
treat the tensor modules and their duals in detail from different view points, 
gaining a complete understanding of them.
\end{enumerate}

These results generalise those for the general linear Lie superalgebra to the $\Gamma$-graded setting.
They are completely elementary, but foundational for the entire theory. 
It appears that the material was not treated coherently before.  

We mention in particular that by using parabolic induction with respect to an appropriate parabolic subalgebra, 
we achieve a classification of the finite dimensional simple modules. Also, there is a 
natural generalisation of the notion of typical and atypical modules originally
introduced by Kac \cite{K} in the representation theory of  Lie superalgebras.  

\smallskip
\noindent 
II. Invariant theory. 
\begin{enumerate}
\setcounter{enumi}{3}

\item \label{aim-3}
Establish colour versions of Howe dualities over symmetric $(\Gamma, \omega)$-algebras;

\item \label{aim-5}
determine $\gl(V(\Gamma, \omega))$-invariants in symmetric $(\Gamma, \omega)$-algebras, and prove first and second fundamental theorems (FFT \& SFT) of invariant theory   in this setting;

\item \label{aim-4}
 deduce colour Schur-Weyl duality from the colour Howe duality, strengthening a double commutant theorem obtained by Fischman and Montgomery \cite[Theorem 2.15]{FM};

\item \label{aim-6}
treat in detail a special case with $\Gamma=\Z^{\dim V(\Gamma, \omega)}$ and $\omega$ depending on a complex parameter $q\ne 0$. In this case,  the symmetric $(\Gamma, \omega)$-algebra $S_\omega(V(\Gamma, \omega))$ 
coincides with the quantised coordinate superalgebra $\C_q\big[\C^{M_+|M_-}\big]$ of  the superspace $\C^{M_+|M_-}$ \cite{Z98, Zy} for appropriate $M_+$ and $M_-$. 
We obtain a multiplicity free decomposition of the quantised coordinate algebra of $N$ copies of $\C^{M_+|M_-}$ with respect to the action of 
$\gl(V(\Gamma, \omega))\times \gl_N(\C)$ at arbitrary $q\ne 0$ including roots of unity.
\end{enumerate} 

This part of the paper amounts to a comprehensive treatment of the classical invariant theory 
of $\gl(V(\Gamma, \omega))$. The approach adopted here is a generalisation of that  
developed by Howe \cite{Ho1}  (see \cite{Ho} for an introduction) in the context of automorphic forms. 

Scheunert  \cite{Sch83} investigated invariant theory for Lie $(\Gamma, \omega)$-algebras by developing 
a ``graded tensor calculus'' to show that the familiar technique of ``tensor contraction'' 
 in mathematical physics can be generalised to the $\Gamma$-graded setting to produce invariants. 
The first fundamental theorem of invariant theory for $\gl(V(\Gamma, \omega))$ developed here 
implies Scheunert's result in this case,  and more importantly,  enables one to construct all invariants 
(a problem which Scheunert did not address in loc. cit.).  
 
Invariant theory is at the foundation of symmetries in physics; the very concept of symmetry 
is an invariant theoretical notion. The elemental situation is that the Hamiltonian 
of a quantum system is invariant under the action of the symmetry algebra of the system, 
e.g., a Lie (super)group or Lie colour (super)algebra.  
A major part of the study of the quantum system is to understand implications of the symmetry.

We mention that in recent years, there have been extensive research developing 
the invariant theories of Lie superalgebras \cite{DLZ, LZ17, LZ21, Zy1} 
and quantum (super)groups \cite{CW20, CW23, LZZ11, LZZ20, Zy}. 
In particular, FFTs and SFTs in various formulations were established.

\smallskip
\noindent 
 III. Unitarisable modules
\begin{enumerate}
\setcounter{enumi}{7}
\item \label{aim-10} 
Develop a basic theory of unitarisable modules for Hopf $(\Gamma, \omega)$-algebras; 
\item  \label{aim-11}
classify the unitarisable $\gl(V(\Gamma, \omega))$-modules with respect 
to the two compact $\ast$-structures; and 
prove that tensor modules and their dual modules for $\gl(V(\Gamma, \omega))$ are unitarisable
with respect to the two compact $\ast$-structures respectively. 
\end{enumerate} 

Similar treatment for quantum supergroups was given in \cite{WZ, Z98}. 

The study of unitarisable modules 
is indispensable for applications of Lie $(\Gamma, \omega)$-algebras in quantum physics, 
especially in view of Wigner's classification of elementary particles 
in terms of unitary representations of the Poincare group, which is 
fundamental for understanding sub-nuclear structures. 

We point out that an associative  $(\Gamma, \omega)$-algebra admits $\ast$-structures
only if the commutative factor $\omega$ satisfies the Unit Modulus Property \eqref{eq:inv-ast}. 
Also the antipode of any 
Hopf $\ast$-$(\Gamma, \omega)$-algebra must satisfy the Antipode Conditions \ref{assum}. 

\smallskip
\noindent 
IV. The coordinate algebra.
\begin{enumerate}
\setcounter{enumi}{9}
\item \label{aim-7}
Introduce a Hopf $(\Gamma, \omega)$-subalgebra $\ST(V(\Gamma, \omega))$ of the finite dual of the universal enveloping algebra of $\gl(V(\Gamma, \omega))$, which contains the subalgebras of matrix coefficients of tensor representations and their duals; Peter-Weyl type theorems are proved for these subalgebras, and 
$\ST(V(\Gamma, \omega))$ is shown to be isomorphic to a ring of fractions of a symmetric $(\Gamma, \omega)$-algebra;

\item \label{aim-8}
construct simple tensor representations and their duals by generalising the classic Borel-Weil theorem 
(recall that the Borel-Weil theorem  realises representations on sections of line bundles on flag manifolds);

\item \label{aim-9}
construct a  group functor from the Hopf $(\Gamma, \omega)$-algebra $\ST(V(\Gamma, \omega))$, 
which may be considered as the general linear colour (super)group. 
\end{enumerate}

This part of the paper aims to explore some generalised notion of a ``group'' associated with the general linear Lie $(\Gamma, \omega)$-algebra. 
The philosophy of this investigation is similar to that in the contexts of algebraic groups and quantum (super)groups, 
namely, to study the coordinate algebras instead of the groups themselves.  
In the present context, the coordinate algebra has
the structure of a $\Gamma$-graded $\omega$-commutative Hopf algebra, 
and gives rise to a group functor.

Using the coordinate algebra, we are able to construct some algebraic analogue of a flag variety and line bundles on it.
Simple tensor modules can be realised as certain subspaces of sections of the line bundles.
In the special case $\Gamma=\Z_2^n$, there should be a connection  with $\Z_2^n$-graded supermanifolds 
studied in recent years (see \cite{BIP, BP-2, CGP-1} and references therein) and 
the non-commutative para-manifolds introduced in \cite{Z24}. 
It will be very interesting mathematically to further develop  this part of the theory from an algebraic geometric viewpoint, 
which, however, is well beyond the scope of the present paper. 

We also mention that a class of new algebras referred to as Schur $(\Gamma, \omega)$-algebras (see Definition \ref{def:Schur-alg}) naturally emerged from the investigation on coordinate algebras, which are of interest in their own right. 

\medskip
\noindent{1.4.}
We note that the study of $\gl(V(\Gamma, \omega))$-actions on 
symmetric $(\Gamma, \omega)$-algebras and the coordinate algebra $\ST(V(\Gamma, \omega))$ 
(which  are all non-commutative) 
falls within the realm of non-commutative geometry in the general spirit  of \cite{C, La} 
and of \cite{M}.
In particular, the appropriate conceptual framework for 
IV.(\ref{aim-8}) is the geometry of non-commutative analogues 
of flag varieties and lines bundles on them.  
Part  II.(\ref{aim-6}) indicates that such non-commutative geometries
are closely related to the quantum geometry  \cite{M}
of quantum groups \cite{D, J} and quantum supergroups \cite{BGZ, Y91, Y94, ZGB, Z93, Z98}, 
but have much better properties.  
For example, results of II.(\ref{aim-6}) remain valid even when $q$ is a root of $1$. 
This is in sharp contrast to quantum (super)groups, 
which become exceedingly difficult at roots of unity. 

\medskip
\noindent{1.5.}
To make the paper accessible to a broad audience interested in Lie colour (super)algebras 
and/or their physical applications, we have provided considerable 
background information on the subject. In particular,  
we have included an introductory section, Section \ref{sect:grad-alg},  
to explain basic notions of various types of algebras over the category of $\Gamma$-graded vector spaces.  

We prove results in this paper by elementary means whenever possible  
(sometimes sacrificing elegance of proofs), and present the pertinent steps of all proofs. 
Most of the methods and techniques required have their origins in the
theories of semi-simple Lie algebras \cite{GW, Ho, H} and Lie superalgebras \cite{K, Sch}, 
thus can be readily mastered by readers who have some familiarity with these subjects. 

\medskip

\noindent{\bf Acknowledgements}. 
We thank Naruhiko Aizawa, Alan Carey, Masud Chaichan, Phillip Isaac, Peter Jarvis, 
Francesco Toppan, Anca Tureanu and Joris Van der Jeugt for comments and suggestions.

\section{Algebras over the category of $\Gamma$-graded vector spaces}\label{sect:grad-alg}

We recall basic notions of 
$\Gamma$-graded $\omega$-algebras of various types, 
which will be needed for studying representations of 
$\Gamma$-graded Lie $\omega$-algebras.
In particular, we lay down the rules on how the commutative factors 
enter various contexts. Most of the material in this section is either extracted from 
various sources (e.g., \cite{FM, MB, Sch79, Sch83, SchZ}) or  
quite straightforward generalisations 
from superalgebras to the present setting.

\subsection{The category of $\Gamma$-graded vector spaces}

We fix an additive abelian group $\Gamma$ throughout the paper, 
and assume that there exists a commutative factor \cite[\S 4, part 7]{B}  
$
\omega: \Gamma\times \Gamma \lra \C^*, 
$
which has the usual properties:
\beq
&&\omega(\alpha, \beta) = \omega(\beta, \alpha)^{-1}, \label{eq:cycle-1} \\
&&\omega(\alpha, \beta+\gamma)= \omega(\alpha, \beta) \omega(\alpha, \gamma), \label{eq:cycle-2}\\
&&\omega(\alpha +\beta, \gamma) = \omega(\alpha, \gamma) \omega(\beta, \gamma), \quad \forall
\alpha, \beta, \gamma.  \label{eq:cycle-3}
\eeq
These relations in particular imply that $\omega(0, 0)=1$, and $\omega(\alpha, \alpha)=\pm 1$ for all $\alpha$.
Thus $\Gamma$ is the disjoint union of 
$\Gamma^\pm=\{\alpha\in\Gamma\mid \omega(\alpha, \alpha)=\pm 1\}$. 
Since
$
\omega(\alpha+\beta,  \alpha+\beta) 
=\omega(\alpha,  \alpha) \omega(\beta,  \beta),
$
we have $\alpha + \beta\in \Gamma^+$ if $\alpha, \beta\in \Gamma^+$. 
Hence $\Gamma^+$ is a subgroup of $\Gamma$, 
and $\omega$ restricts to a factor on $\Gamma^+$. 
We refer to \cite[\S2]{MB} and \cite[\S3]{Sch79} (also Section \ref{sect:q-alg}) 
for explicit examples of commutative factors relevant to the investigation in this paper.

\subsubsection{The category of $\Gamma$-graded vector spaces}
A $\Gamma$-graded vector space $V$ is the direct sum $V=\sum_{\alpha\in\Gamma} V_\alpha$ 
of homogeneous subspaces $V_\alpha$. We will call $\alpha$ the degree of $V_\alpha$.  
For any homogeneous element $v\in V$, we denote by $d(v)\in\Gamma$ its degree.
The space of morphisms $\Hom_\C(V, W)$ between $\Gamma$-graded vector spaces $V$ and $W$
is naturally $\Gamma$-graded $\Hom_\C(V, W)=\sum_\alpha \Hom_\C(V, W)_\alpha$, 
with 
\[
\Hom_\C(V, W)_\alpha = \sum_{\beta\in\Gamma} \Hom_\C(V_\beta, W_{\alpha+\beta}).
\]
We denote $\End_\C(V)=\Hom_\C(V, V)$, the space of endomorphisms of $V$. 

The tensor product $V\ot_\C W$ of $\Gamma$-graded vector spaces $V$ and $W$ is naturally $\Gamma$-graded, with 
$(V\ot_\C W)_\gamma = \sum_{\alpha+\beta=\gamma} V_\alpha\ot_\C W_\beta$. 

The category $\Vect$ of $\Gamma$-graded vector spaces is a strict braided monoidal category, 
where the monoidal structure is given by the tensor product $\ot_\C$, and the braiding is given by a functorial map 
$\tau_{V, W}: V\ot_\C W\lra W\ot_\C V$ arising from the commutative factor $\omega$. It is defined by
\beq\label{eq:tau}
\tau_{V, W}(v\ot w)= \omega(d(v), d(w)) w\ot v, 
\eeq
for homogeneous $v$ and $w$, and by linearly extending this 
\beq
\tau_{V, W}(v\ot w)= \sum_{\alpha, \beta} \omega(\alpha, \beta) w_\beta\ot v_\alpha, 
\eeq
for inhomogeneous elements
$v=\sum_{\alpha\in \Gamma} v_\alpha$  and  $w=\sum_{\alpha\in \Gamma} w_\alpha$, 
where $v_\alpha\in V_\alpha$ and $w_\alpha\in W_\alpha$. 
The braiding will be further discussed in Section \ref{sect:braid}.

\begin{remark}[{\bf Convention}]\label{rmk:conven}
We shall adopt the convention that any explicit formula involving the commutative factor will be written in a form analogous to \eqref{eq:tau}, but is tacitly understood to be extended linearly for inhomogeneous elements. 
\end{remark}

The tensor product of linear homomorphisms also involves the braiding. Consider
the natural injection 
$
\Hom_\C(V, W)\ot  \Hom_\C(V', W')\lra \Hom_\C(V\ot V', W\ot W').
$
Write $H=\Hom_\C(V, W)$ and $H'= \Hom_\C(V', W')$. 
The action of $H\ot H'$ on $V\ot W$ is defined by  
\[
H\ot H'\ot V\ot W\stackrel{\id_H\ot\tau_{H', V}\ot \id_W}\looongrightarrow H\ot V\ot  H'\ot W\lra V'\ot W'.
\]
Explicitly, for any $\varphi\in H$ and $\varphi'\in H'$,   
\beq\label{eq:sign}
(\varphi\ot \varphi')(v\ot v')=\omega(d(\varphi'), d(v)) \varphi(v)\ot \varphi'(v'),  \quad v\in V, v'\in V', 
\eeq
in the convention described in Remark \ref{rmk:conven}. 

In the special case with $V'=V$ and $W'=W$, thus $H=\End_\C(V)$ and $H'=\End_\C(W)$, 
this is an inclusion of algebras as we will see from the discussions preceding  
\eqref{eq:multip-tensor}.

The $\Gamma$-graded 
dual vector space of $V$ is $V^*=\sum_{\alpha} (V^*)_{-\alpha}$ with  the homogeneous subspaces
$(V^*)_{-\alpha}:=\Hom_\C(V_\alpha, \C)$ for all $\alpha$. 
It will be convenient to denote dual space pairing by $\langle \ , \  \rangle: V^*\ot V\lra \C$.

  There is a canonical injection $V^*\ot W^*\hookrightarrow (V\ot_\C W)^*$. Now the dual space pairing between $(V\ot_\C W)^*$ and $V\ot_\C W$ restricts to 
\beq\label{eq:ot-paired}
\langle\ , \ \rangle: (V^*\ot W^*)\ot (V\ot W)\lra \C, 
\eeq
which is the composition of the following maps
\[
V^*\ot W^*\ot V\ot W\stackrel{id_{V^*}\ot \tau_{W^*, V}\ot \id_W}{\looongrightarrow}  V^*\ot V\ot  W^* \ot W \stackrel{\langle \ , \  \rangle \ot \langle \ ,  \  \rangle}{\loongrightarrow} \C.
\]
[Here we have surpressed the injective map from the formulae, and we shall continue to do so in the future.] 
Explicitly, given any $\ol{v}\ot\ol{w}\in V^*\ot W^*$, we have 
\[
\baln
&\langle \ol{v}\ot\ol{w}, v'\ot w'\rangle =\omega(d(\ol{w}), d(v')) \langle\ol{v}, v'\rangle \langle\ol{w}, w'\rangle, 
\quad \forall w'\in W, v' \in V.
\ealn
\]

Note that the full subcategory $\Vect_{\rm fin}$ of finite dimensional $\Gamma$-graded vector spaces 
is a strict braided tensor category. 

\subsubsection{The braiding and symmetric group representations}\label{sect:braid}
The material in this section, in particular, 
the representations of the symmetric group to be constructed, 
will play a crucial role in this paper. 

The functorial map $\tau$ indeed defines a braiding for $\Vect$, as 
part (2) of the following lemma shows. 
\begin{lemma}\label{lem:tau}
Let $U, V, W$ be $\Gamma$-graded vector spaces.  Then
 \begin{enumerate}
 \item $\tau_{W, V} \circ \tau_{V, W}=\id_{V\ot W}$; and 
 \item
the following relation holds in $\Hom_\C(U\ot V\ot W, W\ot V\ot U)$.  
\beq\label{eq:braid}
 &&(\tau_{V,W}\ot\id_U)      \circ    (id_V\ot \tau_{U, W})\circ (\tau_{U, V}\ot\id_W) \\
 &&= (id_W\ot \tau_{U, W}) \circ (\tau_{U, W}\ot\id_V)\circ  (id_U\ot \tau_{V, W}).\nonumber
\eeq
\end{enumerate}
\end{lemma}
\begin{proof}
It follows equation \eqref{eq:tau} that for any $v\in V, w\in W$, 
\[
\baln
\tau_{W, V} \circ \tau_{V, W}(v\ot w)
&= \omega(d(v), d(w)) \omega(d(w), d(v)) v\ot w \\
&= v\ot w,  \qquad \qquad (\text{by \eqref{eq:cycle-1}}).
\ealn
\]
This proves part (1). 

Denote by $P_{U, V, W}$ and $P'_{U, V, W}$ the maps on the left and right sides of \eqref{eq:braid}
respectively. 
Then for any $u\in U, v\in V, w\in W$, we have
\[
\baln
P_{U, V, W}(u\ot v\ot  w) = \omega(d(v), d(w)) \omega(d(u), d(w)) \omega(d(u), d(v)) w\ot v\ot u, \\
P'_{U, V, W}(u\ot v\ot  w) =\omega(d(u), d(v)) \omega(d(u), d(w)) \omega(d(v), d(w))   w\ot v\ot u, 
\ealn
\]
in the convention of Remark \ref{rmk:conven}.
Hence $P_{U, V, W}(u\ot v\ot  w) =P'_{U, V, W}(u\ot v\ot  w)$, proving part (2). 
\end{proof}

\begin{remark}
We will suppress  $\circ$ from compositions of morphisms whenever this  will not cause confusion, 
e.g., $(\tau_{U, W}\ot\id_V)\circ  (id_U\ot \tau_{V, W})$ will simply be written 
as $(\tau_{U, W}\ot\id_V) (id_U\ot \tau_{V, W})$. 
\end{remark}

We will call the braiding $\tau$ the symmetry of $\Vect$ because of its involution property given by Lemma \ref{lem:tau}.(1).   
It is a standard fact that such a symmetry gives rise to representations of the symmetric group.

Denote by $\Sym_r$ the symmetric group of degree $r$, which we present in the standard way with generators $s_i$, $i=1,  2, \dots, r-1$, and defining relations
\[
\baln
&s_i^2=1, \quad s_i s_j = s_j s_i, \quad |i-j|\ge 2, \\
& s_i s_{i+1} s_i = s_{i+1} s_i s_{i+1}, \quad \text{for all valid $i$}.
\ealn
\]
Given any $V\in\Vect$, we construct a representation of 
$\Sym_r$ on $V^{\ot r}$. Let $P=\tau_{V, V}$. 
 Lemma \ref{lem:tau} in the special case with $U=W=V$ reduces to
\beq
&&P^2=\id_V\ot \id_V,  \label{eq:PP}\\
&&(P\ot \id_V) (\id_V\ot P) (P\ot \id_V)= (\id_V\ot P) (P\ot \id_V) (\id_V\ot P). \label{eq:PPP}
\eeq
Define the following elements of $\End_\C(V^{\ot r})$:
\[
\sigma_i:=\underbrace{\id_V\ot\dots\id_V}_{i-1}\ot P\ot \underbrace{\id_V\ot\dots\id_V}_{r-i-1}, \quad   i=1, 2, \dots, r-1.
\]
It easily follows from equations \eqref{eq:PP} and \eqref{eq:PPP} that $\sigma_i$ satisfy the defining relations of $\Sym_r$.  Hence we have the 
following result.
\begin{corollary}\label{cor:sym}
The map
$s_i\mapsto  \sigma_i,$  for all $i=1, 2, \dots, r-1,$
extends  uniquely to a representation 
$\nu_r: \C\Sym_r\lra \End_\C(V^{\ot r})$ of the group algebra of $\Sym_r$.
\end{corollary}

\begin{remark} 
This representation of the symmetric group was considered in \cite[\S 5]{Sch83}. 
\end{remark}

\subsection{Associative $(\Gamma, \omega)$-algebras}\label{sect:assoc}

We briefly describe algebras of various types over the category of $\Gamma$-graded vector spaces. 
An algebra of a given type over $\Vect$ will be called a $\Gamma$-graded $\omega$-algebra
of that type (following the convention of \cite{Sch79}), or a $(\Gamma, \omega)$-algebra for short. 
The advantage of this terminology is that it conveys the dependence of the algebras 
on the grading group $\Gamma$ and commutative factor $\omega$.

A $(\Gamma, \omega)$-algebra $A$ 
will also be loosely referred to as a colour algebra if 
$A^- := \sum_{\alpha\in\Gamma^-} A_\alpha=0$, 
and a  colour superalgebra if $A^-\ne 0$.

\medskip

An associative $(\Gamma, \omega)$-algebra $A$ is a $\Gamma$-graded vector space with an associative multiplication
$A\ot_\C A\lra A$ such that $A_\alpha A_\beta\subseteq A_{\alpha+\beta}$ for all $\alpha, \beta\in\Gamma$.  
%
We say that $A$ is a graded commutative associative $(\Gamma, \omega)$-algebra if
\[
\text{$x y = \omega(d(x), d(y))y x$,  for all homogeneous elements $x, y\in A$.}
\]
In this case,  $A^-$ is nilpotent, that is, for any $x\in A^-$, 
there is a positive integer $n$ such that $x^n=0$.  If $x\in A^-$ is
homogeneous, then  $x^2=0$. 

\begin{warning}
However, inhomogeneous elements of a graded commutative associative 
$(\Gamma, \omega)$-algebra  do not (graded) commute in general. 
\end{warning}

A homomorphism $\varphi: A\lra B$ between $(\Gamma, \omega)$-algebras $A$ and $B$ is an algebra homomorphism satisfying the condition $\varphi(A_\alpha)\subset B_\alpha$ for all $\alpha\in\Gamma$.  
Note that $\varphi$ is homogeneous of degree $0$.
A $\Gamma$-graded module $M=\sum_\alpha M_\alpha$ for $A$ is an $A$-module 
satisfying the condition $A_\alpha\cdot M_\beta \subset M_{\alpha+\beta}$. 
If  $M$ and $M'$ are $\Gamma$-graded $A$-modules, the space of $A$-module 
homomorphisms from $M$ to $M'$ is defined by
\[
\Hom_A(M, M') :=\{\varphi\in \Hom_\C(M, M')_0\mid x\varphi(v)=\varphi(x v), \ \forall x\in A, \ v\in M\}. 
\]
Observe that $A$-module homomorphisms are homogeneous of degree $0$.  We denote $\End_A(M) = \Hom_A(M, M)$. 

\begin{remark}
All modules for $(\Gamma, \omega)$-algebras to be considered in this paper are assumed to be $\Gamma$-graded. 
\end{remark}

The commutative factor $\omega$ comes into play when further structures are considered. 
An anti-homomorphism
$S: A\lra B$ of associative $(\Gamma, \omega)$-algebras is a homogeneous linear map of degree $0$ which renders the following diagram commutative, 
\[
\begin{tikzcd}
A\ot A \arrow[r, "S\ot S"] \arrow[d, swap, "\mu_A"] & B\ot B \arrow[r, "\tau_{B, B}"]  &B\ot B \arrow[d, "\mu_B"]&\\
A \arrow[rr, swap, "S"]{}		&& B,
\end{tikzcd}
\]
where $\mu_A$ and $\mu_B$ are the multiplications of $A$ and $B$ respectively. 
Thus 
\beq\label{eq:anti-hom}
S(x y)= \omega(d(x), d(y)) S(y) S(x), \quad x, y \in  A.  
\eeq

The category of associative $(\Gamma, \omega)$-algebras has a braided monoidal structure 
arising from the tensor product of $\Gamma$-graded vector spaces. 
Let $(A, \mu_A)$ and $(B, \mu_B)$ be associative $(\Gamma, \omega)$-algebras.
The multiplication $\mu_{A\ot B}$  of their tensor product $A\ot B$ is defined by 
\beq\label{eq:multip-tensor}
\mu_{A\ot B}=(\mu_A\ot\mu_B)\circ(\id_A\ot \tau_{A, B}\ot \id_B).
\eeq
Explicitly, for any $a\ot b, a'\ot b'\in A\ot B$, we have 
$
(a\ot b)(a'\ot b')= \omega(d(b), d(a'))  a a' \ot b b'.
$

Associativity of $\mu_{A\ot B}$ follows from the braid relation \eqref{eq:braid} obeyed by the commutative factor. Let us verify it explicitly. 
We have 
\[
\baln
((a\ot b)(a'\ot b'))(a''\ot b'')
&=\omega(d(b), d(a')) \omega(d(b b'), d(a''))aa'a''\ot b b' b'';\\
(a\ot b)((a'\ot b')(a''\ot b''))
&= \omega(d(b'), d(a'')) \omega(d(b), d(a' a'')) aa'a''\ot b b'b''.
\ealn
\]
Note that $\omega(d(b'), d(a'')) \omega(d(b), d(a' a''))=\omega(d(b), d(a')) \omega(d(bb'), d(a''))$ 
by \eqref{eq:cycle-2} and \eqref{eq:cycle-3}. 
Hence  $((a\ot b)(a'\ot b'))(a''\ot b'')=(a\ot b)((a'\ot b')(a''\ot b''))$, 
proving the associativity of $\mu_{A\ot B}$.

Let $V$ and $W$ be $\Gamma$-graded modules for $A$ and $B$ with respective structure maps 
$\pi_A: A\ot V\lra V$ and $\pi_B: B\ot W\lra W$. Then $V\ot_\C W$ is a $\Gamma$-graded $A\ot B$-module with the action defined by the composition
\beq \label{eq:act-tensor}
A\ot B\ot V\ot W \stackrel{\id_A\ot\tau_{B, V}\ot\id_W}\looongrightarrow A\ot V\ot B\ot W \stackrel{\pi_A\ot \pi_B}\loongrightarrow V\ot W.
\eeq
Explicitly, for any $a\in A,  b\in B,  v\in V,  w\in W$, 
\beq\label{eq:act-tensor-1}
(a\ot b)(v\ot w)=\omega(d(b), d(v)) a v \ot b w.
\eeq
To show that this is indeed a well-defined $A\ot B$-action on $V\ot W$, 
consider another pair of elements $a'\in A,  b'\in B$. We have 
\[
\baln
(a'\ot b')((a\ot b)(v\ot w)) 
&= \omega(d(b), d(v)) \omega(d(b'), d(a v)) (a' a v \ot b' b w),\\
((a'\ot b')(a\ot b))(v\ot w)
&= \omega(d(b b'), d(v)) \omega(d(b'), d(a)) (a' a v \ot b' b w).
\ealn
\]
It is easy to verify that $\omega(d(b), d(v)) \omega(d(b'), d(a v))
= \omega(d(b b'), d(v)) \omega(d(b'), d(a)).$
Hence $(a'\ot b')((a\ot b)(v\ot w)) =((a'\ot b')(a\ot b))(v\ot w)$.

We call an element $\partial\in \End_\C(A)_\gamma$ a $\Gamma$-graded $\omega$-derivation, or
$(\Gamma, \omega)$-derivation for short,  of degree $\gamma$,  
if it  satisfies the following condition for all $a, b\in A$.
\beq\label{eq:deriv}
\partial(a b) = \partial(a)b + \omega(\gamma, d(a)) a \partial(b). 
\eeq
We will simply call it a graded derivation when $\Gamma$ and $\omega$ are clear from the context. 

The following example of graded derivations motivates the definition of Lie $(\Gamma, \omega)$-algebras.
The $\Gamma$-graded $\omega$-commutator, or $(\Gamma, \omega)$-commutator, on $A$ is defined by 
\beq\label{eq:bracket}
\phantom{XXX}
[\ , \ ]: A\times A \lra A, \quad {[X, Y]} = X Y - \omega(d(X), d(Y)) Y X, \quad X, Y\in A,  
\eeq
which clearly satisfies 
\beq\label{eq:skew}
{[X, Y]} =  - \omega(d(X), d(Y)) [Y,  X]. 
\eeq
 It produces $(\Gamma, \omega)$-derivations as follows. 
Given any $X \in A$, define
$\partial_{X}\in\End_\C(A)$ by 
\[
\partial_{X}(Y)= [X, Y], \quad \forall Y\in A. 
\]
Then the following facts are well known and easy to verify by direct calculations

\begin{lemma} \label{lem:deriv}
Assume that $X\in A_\alpha$. Then
\begin{enumerate}
\item $\partial_{X}$ is a $(\Gamma, \omega)$-derivation of degree $\alpha$; and 
\item for all $Y, Z\in A$, 
\[
\partial_X([Y, Z]) = [ \partial_X(Y), Z] + \omega(\alpha, d(Y)) [Y, \partial_X(Z)].
\]
\end{enumerate}
\end{lemma}

Note that
part (2) of the lemma is equivalent to 
\beq \label{eq:Jocob}
[X, [Y, Z]] = [ [X, Y], Z] + \omega(\alpha, d(Y)) [Y, [X, Z]].
\eeq

\subsection{Lie $(\Gamma, \omega)$-algebras} \label{sect:Lie}
The  discussions above on graded derivations naturally suggest the following 
notion of generalised Lie algebras.  

\begin{definition} \cite{RW, Sch79} 
A $\Gamma$-graded Lie $\omega$-algebra, or Lie $(\Gamma, \omega)$-algebra, 
is a $\Gamma$-graded vector space 
$\fg=\sum_{\alpha\in\Gamma} \fg_\alpha$ endowed with a bilinear map
$
[\ , \ ]: \fg\times \fg \lra \fg, 
$
called the Lie $(\Gamma, \omega)$-bracket, which is homogeneous of degree $0$, and satisfies the following conditions 
\beq
&&[X, Y] =- \omega(d(X), d(Y))[Y, X], \label{eq:skew-2}\\
&&[X, [Y, Z]] = [ [X, Y], Z] + \omega(d(X), d(Y)) [Y, [X, Z]], \label{eq:Jocob-2}
\eeq
for any $X, Y, Z\in \fg$ in the convention of Remark \ref{rmk:conven}.
\end{definition} 

We will refer to these relations respectively as the skew $\omega$-symmetry, and $\omega$-Jacobian identity, 
of the Lie $(\Gamma, \omega)$-bracket. 


A homomorphism  $\varphi: \fg \lra \fg'$ of Lie $(\Gamma, \omega)$-algebras is a  homogeneous linear map of degree $0$  such that $\varphi([X, Y])= [\varphi(X) ,  \varphi(Y)]$ for all $X, Y\in \fg$

Let  $\fg^\pm=\sum_{\alpha\in \Gamma^\pm} \fg_\alpha$.  Then $\fg^+$ is a Lie $(\Gamma, \omega)$-subalgebra 
(which is a Lie $(\Gamma^+, \omega)$-algebra), $\fg^-$ is a $\Gamma$-graded $\fg^+$-module (see definition below)
and $[\fg^-, \fg^-]\subset\fg^+$.  
Consider $X, Y\in \fg_\alpha$. If $\alpha\in\Gamma^+$, then
$[X, Y]= -[Y, X]$, thus they behave like even elements of a Lie superalgebra. On the other hand, if 
$\alpha\in\Gamma^-$, then $[X, Y]= [Y, X]$, thus $X, Y$ behave like odd elements of a Lie superalgebra.
Thus a Lie colour algebra $\fg$ contains no ``odd'' elements as $\fg^-= 0$, but 
a Lie colour superalgebra $\fg$ has $\fg^-\ne 0$. 

\begin{example}[General linear Lie $(\Gamma, \omega)$-algebras]
\label{eg:gl}
For any $\Gamma$-graded
vector space $V(\Gamma, \omega)=\sum_{\alpha\in\Gamma} V_\alpha$,  the space of endomorphisms  $\End_\C(V(\Gamma, \omega))$ is an associative $(\Gamma, \omega)$-algebra with the composition of endomorphisms as the multiplication.  
 
The general linear Lie $(\Gamma, \omega)$-algebra $\gl(V(\Gamma, \omega))$ of $V(\Gamma, \omega)$ is $\End_\C(V(\Gamma, \omega))$ with the $(\Gamma, \omega)$-commutator defined by \eqref{eq:bracket} as the Lie $(\Gamma, \omega)$-bracket. 
Now the degree $\alpha$ subspace  of  $\gl(V(\Gamma, \omega))$ is $\gl(V)_\alpha = \sum_{\beta} \Hom_\C(V_\beta, V_{\alpha+\beta})$. In particular, $\gl(V(\Gamma, \omega))_0=\sum_{\beta} \End_\C(V_\beta)=\sum_{\beta} \gl_\C(V_\beta)$ forms a Lie subalgebra (an ordinary Lie algebra).

Observe that for any $\phi_{\beta \alpha} \in \Hom_\C(V_\alpha, V_\beta)$ 
and $\psi_{\delta \gamma} \in \Hom_\C(V_\gamma, V_\delta)$, their composition 
$\phi_{\beta \alpha} \psi_{\delta \gamma}$ is non-vanishing only if $\alpha= \delta$, and  belongs to $\End_\C(V(\Gamma, \omega))_{\beta-\gamma}$ in this case. This enables us to introduce decompositions for $\gl(V(\Gamma, \omega))$
as follows.

Let $\Gamma_R:=\{\alpha\in \Gamma\mid \dim V_\alpha>0\}$.  
We assume that there is a linear order on the set $\Gamma_R$ (making it into a chain).
Let 
\[
\fu:=\sum_{\alpha<\beta}   \Hom_\C(V_\beta, V_\alpha), \quad 
\ol\fu:=\sum_{\alpha<\beta}   \Hom_\C(V_\alpha, V_\beta).
\]
If $\psi(\alpha, \beta)\in \Hom_\C(V_\beta, V_\alpha),  \psi(\gamma, \eta)\in \Hom_\C(V_\eta, V_\gamma)$ with 
$\alpha<\beta$ and $\gamma<\eta$, which are elements of $\fu$, then
\[
{[\psi(\alpha, \beta), \psi(\gamma, \eta)]} = \delta_{\beta \gamma} \psi(\alpha, \beta)\psi(\gamma, \eta)
- \delta_{\eta\alpha}\omega(\alpha-\beta,  \gamma-\eta)   \psi(\gamma, \eta)\psi(\alpha, \beta)
\]
belongs to $\fu$. Hence $\fu$ is a Lie $(\Gamma, \omega)$-subalgebra, and we can similarly show that $\ol{\fu}$ is 
also a Lie $(\Gamma, \omega)$-subalgebra. 
Thus $\gl(V(\Gamma, \omega))$ has the decomposition
\beq\label{eq:parabolic}
\gl(V(\Gamma, \omega))=\ol\fu +\gl(V(\Gamma, \omega))_0 + \fu.
\eeq
Denote $\fl=\gl(V(\Gamma, \omega))_0$, and let $\fp=\fl+\fu$. 
Then $\fp$ is a parabolic subalgebra of $\gl(V(\Gamma, \omega))$, 
and $\fu$ is a Lie $(\Gamma, \omega)$-ideal of $\fp$.  
\end{example}

Return to an arbitrary Lie $(\Gamma, \omega)$-algebra $\fg$. 
\begin{definition}
A $\fg$-module $M$ is a vector space endowed with a bilinear map $\fg\times M\lra M$ such that for any $X\in \fg_\alpha$ and $Y\in\fg_\beta$, 
\[
X\cdot (Y\cdot v) - \omega(\alpha, \beta) Y\cdot (X\cdot v) = [X, Y]\cdot v, \quad \forall v\in M.  
\]
If $M$ is $\Gamma$-graded as a vector space and $\fg_\alpha\cdot V_\beta\subset V_{\alpha+\beta}$ for all $\alpha, \beta\in\Gamma$, then $M$ is a $\Gamma$-graded $\fg$-module. 
\end{definition}
Thus a $\fg$-module is a vector space $M$ together with a linear map $\pi: \fg\lra \End_\C(M)$ such that 
$\pi(X)\pi(Y)- \omega(\alpha, \beta) \pi(Y)\pi(X)=\pi([X, Y])$ for all $X, Y\in \fg$. The $\fg$-module $M$ is $\Gamma$-graded 
if it is a $\Gamma$-graded vector space and $\pi: \fg\lra \gl(M)$ is a Lie  $(\Gamma, \omega)$-algebra homomorphism.

The universal enveloping algebra of $\fg$ is defined as follows. 
Consider the tensor algebra $T(\fg)$ over $\fg$, which has the usual $\Z_+$-grading with  
\[
T(\fg)=\sum_{r=0}^\infty T^r(\fg), \quad T^r(\fg)= \underbrace{\fg\ot\fg\ot\dots\ot\fg}_r.
\]
It is also naturally $\Gamma$-graded, with $T(\fg)= \sum_{\alpha\in \Gamma}T(\fg)_\alpha$, where
\[
\baln
T(\fg)_\alpha=\sum_{r=0}^\infty T^r(\fg)_\alpha, \quad
T^r(\fg)_\alpha=\sum_{\alpha_1+\dots+\alpha_r=\alpha}\fg_{\alpha_1}\ot \fg_{\alpha_2}\ot \dots\ot \fg_{\alpha_r}.
\ealn
\]
Let $J$ be the $2$-sided ideal of $T(\fg)$ generated by the following set of elements 
\[
\bigcup_{\alpha, \beta}\left\{
X Y - \omega(\alpha, \beta) Y X - [X, Y] \mid X\in\fg_\alpha, \ Y\in \g_\beta
\right\}, 
\]
which is a homogeneous ideal. 
The universal enveloping algebra of $\fg$ is defined by
\beq
\U(\fg)=T(\fg)/J, 
\eeq
which is an associative $(\Gamma, \omega)$-algebra. 

It is evident that any $\fg$-module is automatically a $\U(\fg)$-module and vice versa. 

A generalised Poincar\'e-Birkhoff-Witt theorem  for $\U(\fg)$ was established in \cite[\S C]{Sch79}. 
Let $\{X_i\mid i\in I\}$ be a homogeneous basis of a Lie $(\Gamma, \omega)$-algebra $\fg$, 
where $I$ is a totally ordered set. For any $i\in I$, let $\Z_{(i)}=\left\{
\begin{array}{l l}
\Z_+, &\text{if $d(X_i)\in\Gamma^+$}, \\
\{0, 1\}, &\text{if $d(X_i)\in\Gamma^-$}.
\end{array}
\right.$ 
Given any finite chain ${\bf i}=(i_1<_2<\dots <i_r)$ in $I$, let 
$\Z_{({\bf i})} = \Z_{(i_1)}\times\Z_{(i_2)}\times\dots \times \Z_{(i_r)}$, and define  
\[
X_{\bf i}^{\bf k}=X_{i_1}^{k_1} X_{i_2}^{k_2}\dots X_{i_r}^{k_r}, \quad {\bf k}=(k_1, k_2, \dots, k_r)\in \Z_{({\bf i})}.
\]
\begin{theorem}[Poincar\'e-Birkhoff-Witt theorem] \label{thm:PBW} Retain notation above. 
The elements $X_{\bf i}^{\bf k}$, for all ${\bf k}\in \Z_{({\bf i})}$ and finite chains ${\bf i}$ in $I$, form a basis of the universal enveloping algebra $\U(\fg)$. 
\end{theorem}
This fundamental result is due to Scheunert, see \cite[Corollary 1]{Sch79}.

\subsection{Hopf $(\Gamma, \omega)$-algebras}

We want to explain the generalised Hopf algebra structure of the universal enveloping algebra $\U(\fg)$ 
of any Lie $(\Gamma, \omega)$-algebra. 
Let us begin by recalling the definition of Hopf $(\Gamma, \omega)$-algebras (see, e.g., \cite{MB}). 

\subsubsection{Hopf $(\Gamma, \omega)$-algebras}\label{sect:Hopf-dual}

Let $A$ be a unital associative $(\Gamma, \omega)$-algebra with multiplication $\mu: A\ot A\lra A$.
It is a 
$(\Gamma, \omega)$-bi-algebra if there exist $(\Gamma, \omega)$-algebra 
homomorphisms $\varepsilon: A\lra\C$ and $\Delta: A\lra A\ot A$ such that  
\[
\baln
&(\Delta\ot\id)\Delta= (\id\ot \Delta)\Delta: A\lra A\ot A\ot A, \\
&(\varepsilon\ot\id)\Delta =\id= (\id\ot \varepsilon)\Delta: A\lra A, 
\ealn
\]
where the second condition involves the canonical identifications  $\C\ot A\simeq A\simeq A\ot \C$.
Call $\Delta$ the co-multiplication and $\varepsilon$ the co-unit.  
We will adopt Sweedler's notation for the co-multiplication: for any $x\in A$, 
\[
\Delta(x)=\sum_{(x)}x_{(1)}\ot x_{(2)}, \quad 
(\Delta\ot\id)\Delta(x)= \sum_{(x)}x_{(1)}\ot x_{(2)}\ot x_{(3)}, \ etc..
\]
 
A Hopf $(\Gamma, \omega)$-algebra is a $(\Gamma, \omega)$-bi-algebra  $(A, \mu; \Delta, \varepsilon)$ with a 
$(\Gamma, \omega)$-algebra anti-homomorphism 
$
S: A\lra A, 
$
called the antipode, such that 
\beq\label{eq:Del-S}
\mu(S\ot \id)\Delta=\mu(\id\ot S)\Delta= \varepsilon.
\eeq
In Sweedler's notation, 
$\sum_{(x)}S(x_{(1)}) x_{(2)} = \sum_{(x)}x_{(1)}S(x_{(2)})= \varepsilon(x)$ for all $x\in S$. 

A $\Gamma$-graded module for a Hopf $(\Gamma, \omega)$-algebra $A$ is a $\Gamma$-graded module  for $A$ as an associative $(\Gamma, \omega)$-algebra. 
Let $V$ and $W$ be $\Gamma$-graded $A$-modules, then  
$V\ot W$ is a $\Gamma$-graded module for $A\ot A$ with the action defined by \eqref{eq:act-tensor}.  
As the co-multiplication $\Delta: A\lra A\ot A$ is a $(\Gamma, \omega)$-algebra homomorphism, 
it induces an $A$-action on $V\ot W$.  

\begin{remark}
Even though the definition of a Hopf $(\Gamma, \omega)$-algebra formally looks the same 
as that of an ordinary Hopf algebra,  its co-multiplication  (see \eqref{eq:multip-tensor}) and 
antipode (see \eqref{eq:anti-hom}) depend on the grading group $\Gamma$ and  
commutative factor $\omega$ in a highly non-trivial way. 
\end{remark}

\subsubsection{The Hopf  $(\Gamma, \omega)$-algebra $\U(\fg)$}
Let $\U(\fg)$ be the universal enveloping algebra of a Lie $(\Gamma, \omega)$-algebra $\fg$. 
Note that $\U(\fg)$ has an augmentation 
$\varepsilon: \U(\fg)\lra \C$ such that $\ker\varepsilon=\fg\U(\fg)$. Let us define a linear map
$
\Delta: \U(\fg)\lra \U(\fg)\ot \U(\fg)
$
by 
\beq
&\Delta(1)=1\ot 1, \quad \Delta(x y) = \Delta(x)\Delta(y),  \quad x, y\in  \U(\fg), \\
&\Delta(X) = X\ot 1 + 1\ot X, \quad X\in\fg\hookrightarrow \U(\fg).
\eeq
Then this gives rise to the co-multiplication for $\U(\fg)$. To see this, 
we note that for any $X\in\fg_\alpha$ and $Y\in\fg_\beta$, we have 
\[
\baln
\Delta([X, Y])&=\Delta(X) \Delta(Y)- \omega(\alpha, \beta) \Delta(Y) \Delta(X ) 
=[X, Y] \ot 1 + 1\ot [X, Y].
\ealn
\]
Hence $\Delta$ is indeed a $(\Gamma, \omega)$-algebra homomorphism. Clearly $(\varepsilon\ot\Delta)(X)= 1\ot X$ and $(\Delta\ot\varepsilon)(X)=  X\ot 1$. 
Hence $\U(\fg)$ is a $(\Gamma, \omega)$-bi-algebra with co-multiplication $\Delta$ and co-unit $\varepsilon$. 

It is easy to see that we have the following algebra anti-automorphism 
\[
S: \U(\fg)\lra\U(\fg), \quad S(X)=-X, \quad X\in\fg. 
\]
Indeed, for any $X\in\fg_\alpha$ and $Y\in\fg_\beta$, we have 
\[
\baln
S([X, Y])&=S(X Y- \omega(\alpha, \beta) Y X ) 
&= \omega(\alpha, \beta) Y X - X Y  = -[X, Y]. 
\ealn
\]
Clearly $(S\ot \id)\Delta(X)= (\id\ot S)\Delta(X)=0$ for all $X\in \fg$, and it follows that  $(S\ot \id)\Delta(u)= (\id\ot S)\Delta(u)=\varepsilon(u)$ for all $u\in \U(\fg)$. Thus $S$ defines an antipode for $\U(\fg)$. 
Note that $S^2=\id_{\U(\fg)}$.

\subsubsection{The finite dual of a Hopf $(\Gamma, \omega)$-algebra}

Let $(A, \mu; \Delta, \varepsilon, S)$ be a Hopf $(\Gamma, \omega)$-algebra.
The dual space $A^*$ is a $(\Gamma, \omega)$-algebra  with the associative multiplication defined, for any $f, g\in A^*$,  by 
\[
(fg)(x)=\langle f\ot g,  \Delta(x)\rangle, \quad \forall x\in A. 
\]
Note that associativity of this multiplication follows from the co-associativity of $\Delta$.

Consider  the subspace $A^0$ of $ A^*$ consisting of elements with finite dimensional cokernels, 
i.e., every $f\in A^0$ satisfies the condition $\dim_\C\frac{ A}{\ker(f)}<\infty$.  
Then the usual arguments in non-graded case 
can be applied verbatim here to show that $A^0$ has  
a Hopf $(\Gamma, \omega)$-algebraic structure dualising that of $A$. 
Explicitly, the structure maps are defined as follows:

multiplication: 
$
\mu_0: A^0\ot A^0\lra  A^0$, $\langle \mu_0(f\ot g), x \rangle = \langle f\ot g, \Delta(x) \rangle, \ \forall x\in A, 
$

unit:  $u_0: \C\lra A^0$, $u_0(a)= a 1_{A^0}, \ \forall a\in\C$, where $1_{A^0}=\varepsilon$, 

co-multiplication:
$\Delta_0: A^0\lra A^0\ot A^0$,  $\langle\Delta_0(f), x\ot y \rangle = \langle f, x y \rangle, \ \forall x, y\in A$,

co-unit $\varepsilon_0: A^0\lra \C$,  $\varepsilon_0(f)= f(1_A), \ \forall f\in A^0$,

antipode $S_0: A^0\lra A^0$, $\langle S_0(f), x \rangle = \langle f, S(x) \rangle, \ \forall x\in A$. 

\noindent 
Note that  the definitions of $\mu_0$ and $\Delta_0$ involve the dual space pairing of tensor product spaces
given by \eqref{eq:ot-paired}.

For any $A$-module $V$, its dual space $V^*$ is an $A$-module defined as follows. 
For any $x\in A$ and $\ol{w}\in V^*$, 
\beq\label{eq:dual-mod}
\langle x\cdot \ol{w}, v\rangle = \omega(d(x), d(\ol{w})) \langle \ol{w}, S(x)\cdot v \rangle, \quad \forall v\in V. 
\eeq
Note in particular that  
\[
\langle  x\cdot (\ol{w} \ot v) \rangle = \varepsilon(x) \langle  \ol{w},  v \rangle, \quad \forall 
x\in A, \ol{w}\in V^*, v\in V. 
\]

\begin{remark}
If the antipode $S$ is an anti-automorphism, then there is another $A$-module structure on 
$V^*$ defined by
$
\langle x\cdot' \ol{w}, v \rangle= \omega(d(x), d(\ol{w})) \langle \ol{w}, S^{-1}(v) \rangle$ for all $v\in V. 
$
However, we will always take the dual module structure as defined by  \eqref{eq:dual-mod} 
unless explicitly stated otherwise. 
\end{remark}

Assume that $V$ is a finite dimensional $\Gamma$-graded $A$-module. Then we have the canonical $\Gamma$-graded vector space isomorphism 
\beq 
\iota: V\ot V^*\lra \End_\C(V),
\eeq
defined for any $u\in V, \ol{v}\in V^*$ by $\iota(u\ot \ol{v})(w)=  \ol{v}(w) u$ for all $w\in V$.  

Now  $V\ot V^*$ is an $A$-module in the natural way.  Also $A$ acts on $\End_\C(V)$ by the adjoint action
\beq
ad: A\ot \End_\C(V)\lra \End_\C(V), \quad x\ot\varphi\mapsto ad_x(\varphi), 
\eeq
which is defined as follows. Let $E=\End_\C(V)$, and denote by $\mu_E$ its multiplication. Write the representation of $A$ 
associated with $V$ by $\pi: A\lra  E$. Then the adjoint action is defined by the composition
\[
A\ot E \stackrel{\Delta\ot\id}{\longrightarrow}A\ot A\ot E\stackrel{\id\ot \tau}{\longrightarrow}
A\ot E\ot A \stackrel{\pi\ot \id\ot \pi\circ S}{\looongrightarrow} E^{\ot 3} \stackrel{\mu_E(\mu_E\ot \id)}{\looongrightarrow} E. 
\]
Thus for any $x\in A$ and $\varphi\in E$,  
\beq\label{eq:adx}
ad_x(\varphi) =\sum_{(x)} \omega(d(x_{(2)}), d(\varphi)) \pi(x_{(1)}) \varphi \pi(S(x_{(2)})). 
\eeq
\begin{lemma}
Assume that $V$ is a finite dimensional $\Gamma$-graded module for the Hopf $(\Gamma, \omega)$-algebra $A$.
Then the map $\iota: V\ot V^*\lra \End_\C(V)$ is an $A$-module isomorphism. 
\end{lemma}
\begin{proof} Since $\iota$ is clearly a vector space isomorphism, we only need to show that it is an $A$-map.   
For any $u\in V, \ol{v}\in V^*$, and $x\in A$, 
\[
x\cdot(u\ot\ol{v})= \sum \omega(d(x_{(2)}), d(u)) x_{(1)}\cdot u \ot x_{(2)}\cdot \ol{v}. 
\]
Hence for any $w\in V$, 
\[
\baln
\iota(x\cdot(u\ot\ol{v}))(w) &= \sum \omega(d(x_{(2)}), d(u)) x_{(1)}\cdot u   \langle x_{(2)}\cdot \ol{v}, w\rangle\\
&= \sum \omega(d(x_{(2)}), d(u))  \omega(d(x_{(2)}), d(\ol{v})) 
	x_{(1)}\cdot u   \langle \ol{v}, S(x_{(2)})\cdot w\rangle\\
&= \sum \omega(d(x_{(2)}), d(u\ot\ol{v})) 
	\iota(x_{(1)}\cdot u\ot \ol{v})(S(x_{(2)})\cdot w)\\
	&= \sum \omega(d(x_{(2)}), d(u\ot\ol{v})) 
	\pi(x_{(1)}) \iota(u\ot \ol{v}) \pi(S(x_{(2)}))(w). 
\ealn
\]
By \eqref{eq:adx}, the last expression is equal to $ad_x(\iota(u\ot \ol{v}))(w)$. 
Hence $\iota(x\cdot(u\ot\ol{v}))= ad_x(\iota(u\ot \ol{v}))$, 
showing that $\iota$ is an $A$-morphism. 
\end{proof}

Hence $\iota^{-1}(\id_V)$ is an $A$-invariant in $V\ot V^*$, i.e., 
$ad_x(\iota^{-1}(\id_V))=\varepsilon(x) \iota^{-1}(\id_V)$ for all $x\in A$.
Call $\iota^{-1}(\id_V)$ the canonical $A$-invariant in $V\ot V^*$. 

\begin{remark}
The adjoint action generalises to $\Hom_\C(V, W)$ in the obvious way for any $A$-modules $V$ and $W$, and the canonical injection $W\ot V^*\lra \Hom_\C(V, W)$ is an $A$-homomorphism.
\end{remark}

The notion of co-modules generalises to $(\Gamma, \omega)$-bi-algebras in the obvious way.  
Let us consider co-modules for $A^0$. 
A right $A^0$-comodule is a vector space $V$ equipped with a map 
$
\delta: V\lra V\ot A^0, 
$
such that 
\beq\label{eq:co-act}
(\delta\ot\id_{A^0})\delta = (\id_V\ot\Delta_0)\delta, \quad (\id_V\ot\epsilon_0)\delta =\id_V.
\eeq

Elements of $V^{A^0}=\{v\in V\mid  \delta(v)=v\ot 1_{A^0}\}$ are called $A^0$-invariants of $V$. 

Let $V$ and $W$ be $A^0$-comodules with respective structure maps 
$\delta_V: V\lra V\ot A^0$ and $\delta_W: W\lra W\ot A^0$. [{\bf When more than one comodules are considered, we will use the notation $\delta_M$ to denote the structure map of a comodule $M$.}]
Then $V\ot W$ is also a right $A^0$-comodule with the structure map
\beq\label{eq:comod-VW}
\delta_{V\ot W}: V\ot W\lra V\ot W\ot A^0
\eeq
defined by the following composition of  maps
\[
{V{\ot}W}{\stackrel{\delta_V\ot\delta_W}\loongrightarrow}{V{\ot}{A^0}{\ot}W{\ot}{A^0}}{\stackrel{\id_V\ot\tau_{A^0, W}\ot \id_{A^0}}\looongrightarrow}{V}{\ot}{W}{\ot}{A^0}{\ot}{A^0}{\stackrel{\id_{V\ot W}\ot\mu_0}\loongrightarrow}{V{\ot}W{\ot}{A^0}}. 
\]

There is a bijection between locally finite left $A$-modules and right $A^0$-comodules. 
If $V$ is a right $A^0$-comodule  with structure map $\delta_V: V\lra V\ot A^0$, the left $A$-action $M_A: A\ot V\lra V$ on $V$ is defined by the following commutative diagram, 
\beq\label{eq:mod-comod}
\begin{tikzcd}
A\ot V \arrow[d, swap, "M_A"] \arrow[r, "\tau_{V, A}"]  &  V\ot A \arrow[d, "\delta\ot\id_A"]  \\
V  &  \arrow[l, "\id_V\ot{\langle \ , \ \rangle}"]V\ot A^0\ot A. 
\end{tikzcd}
\eeq
The first relation of \eqref{eq:co-act} implies that $x\cdot(y\cdot v) = (x y)\cdot v$ for all $x, y\in A$ and $v\in V$, and the second relation implies that $1\cdot v= v$ for all $v\in V$. 
Given a locally finite $A$-module, the commutative diagram also defines the corresponding $A^0$-co-action. 
 
\begin{example}\label{eg:coeffs}
 Let $V$ be a right $A^0$-co-module with structure map $\delta: V\lra V\ot A^0$. 
 For any $v\in V$, we can always write $\delta(v)$ as a finite sum $\delta(v)=\sum_i v_i\ot f_{v, i}$
 for some {\bf linearly independent} elements $v_i\in V$ and $f_{v, i}\in A^0$. 
Call the elements $f_{v, i}$ for all $i$ and $v$ the matrix coefficients of the right co-module $V$, 
and denote their $\C$-span by $T^{(V)}$. Note that $T^{(V)}$ satisfies 
$\Delta_0(T^{(V)})\subset T^{(V)}\ot T^{(V)}$, 
thus is a two-sided co-ideal of $A^0$. 
\end{example} 

\subsubsection{Left actions of Hopf  $(\Gamma, \omega)$-algebra on its finite dual}\label{sect:L-R-acts}

There are two natural left actions 
of a Hopf $(\Gamma, \omega)$-algebra $A$ on its finite dual $A^0$ (see e.g., \cite{M}), 
which will be important for a Borel-Weil type construction of simple tensor modules 
of the general linear Lie $(\Gamma, \omega)$-algebra in Section \ref{sect:BW}. 
Thus we briefly discuss them here.  
 
Define the maps $R, L:    A\ot A^0\lra A^0$ respectively by the following compositions.
  \beq
  &&R:   A\ot A^0\stackrel{\tau_{A, A^0}}\longrightarrow A^0\ot A \stackrel{\Delta_0\ot \id}\loongrightarrow
 A^0\ot A^0\ot A \stackrel{\id\ot {\langle \ , \ \rangle}}\loongrightarrow A^0, \label{eq:R-act}\\
 &&L:  A\ot A^0\stackrel{S\ot \Delta_0}\loongrightarrow A\ot A^0\ot A^0 \stackrel{\tau_{A, A^0}\ot \id}\loongrightarrow
 A^0\ot A\ot A^0 \stackrel{{\langle \ , \ \rangle}\ot \id}\loongrightarrow A^0.  \label{eq:L-act}
 \eeq
 For any $x\in A$ and $f\in A^0$, we have 
 \beq
 R_x(f)= \omega(d(x), d(f))  \sum  f_{(1)} \langle f_{(2)}, x \rangle, \label{eq:R}\\
 L_x(f)
 = \omega(d(x), d(x)) \sum   \langle f_{(1)}, S(x) \rangle f_{(2)}. \label{eq:L}
 \eeq
 The first formula is easy to see, and the second follows from the following calculation
 \[
 \baln
 L_x(f)
 =\sum \omega(d(x), d(f_{(1)}))  \langle f_{(1)}, S(x) \rangle f_{(2)}
 = \omega(d(x), d(x)) \sum   \langle f_{(1)}, S(x) \rangle f_{(2)}. 
 \ealn
 \] 

 
 \begin{lemma} \label{lem:R-L}
 The maps $R$ and $L$ give rise to left actions of $A$ on $A^0$, 
which graded commute in the sense that    
\[
R_y   L_x = \omega(d(y), d(x)) L_x  R_y, \quad \forall x, y\in A. 
\]
\end{lemma}

 \begin{lemma} \label{lem:mod-alg}
 Retain notation above. 
 \begin{enumerate}
\item 
The $A$-action $R$ on $A^0$ satisfies
 \[ 
 R_x\circ \mu_0=\mu_0\circ (R\ot R)_{\Delta(x)}, \quad \forall x\in A. 
 \] 
Thus $A^0$  is a left $R_A$-module algebra with respect to the co-multiplication $\Delta$. 
 \item 
 The $A$-action $L$ on $A^0$ satisfies
 \[ 
 L_x\circ \mu_0=\mu_0\circ (L\ot L)_{\Delta'(x)}, \quad \forall x\in A, 
 \]
where $\Delta'$ is the opposite co-multiplication. 
 Thus $A^0$ is a left $L_A$-module algebra with respect to $\Delta'$.
\end{enumerate}
\end{lemma}
 
\begin{remark}
The usual notation for the two left actions in the Hopf algebra literature is 
$\rightharpoonup$ and $\leftharpoonup$.  We choose our notation because of
the analogy of the actions with the left and right translations in the Lie group context. 
\end{remark}

The lemmas are known, thus we could  take them for granted in principle. 
However,  we explicitly prove them here for peace of the mind. 
As we will see, some of the manipulations required in the proof are quite delicate. 
 
 \begin{proof}[Proof of Lemma \ref{lem:R-L}]
  We consider $R$ first. For any $x, y\in A$ and $f\in A^0$, we have 
 \[
 \baln
 R_y(R_x(f))&= \omega(d(x), d(f))  \sum  R_y(f_{(1)}) \langle f_{(2)}, x \rangle \\
&= \omega(d(x), d(f))  \sum \omega(d(y), d(f_{(1)})+d(f_{2)}))   f_{(1)} \langle f_{(2)}, y \rangle  \langle f_{(3)}, x \rangle.
 \ealn
 \]
On the right hand side, $d(f_{(1)})+d(f_{(2)})= d(f)-d(f_{(3)})= d(f)+d(x)$. Thus 
\[
 \baln
 R_y(R_x(f))
&=\omega(d(y x), d(f))   \omega(d(y), d(x)) \sum  f_{(1)} \langle f_{(2)}, y \rangle  \langle f_{(3)}, x \rangle. 
 \ealn
 \]
Note that 
\[
 \baln
\sum  f_{(1)} \langle f_{(2)}, y x\rangle  
  &= \sum  f_{(1)} \langle f_{(2)}\ot f_{(3)}, y \ot x\rangle \\
&= \sum  \omega(d(f_{(3)}), d(y)) f_{(1)} \langle f_{(2)}, y \rangle  \langle f_{(3)}, x \rangle  \\
&=\omega( d(y), d(x))  \sum  f_{(1)} \langle f_{(2)}, y \rangle  \langle f_{(3)}, x \rangle.
 \ealn
 \]
Hence 
$
 R_y(R_x(f))
=\omega(d(y x), d(f))   \sum  f_{(1)} \langle f_{(2)}, y x\rangle
= R_{y x}(f),  
$
proving that $R$ is a left action.

 Using \eqref{eq:L}, we obtain
 \[
 \baln
 L_y (L_x(f))&= \omega(d(x), d(x)) \omega(d(y), d(y)) \sum   \langle f_{(1)}, S(x) \rangle \langle f_{(2)}, S(y) \rangle f_{(3)}.
 \ealn
 \] 
 Note that $\omega(d(x), d(x)) \omega(d(y), d(y)) =\omega(d(y x), d(y x))$. Also, 
 \[
 \baln
 &\sum   \langle f_{(1)}, S(x) \rangle \langle f_{(2)}, S(y) \rangle f_{(3)}\\
 &=  \omega(d(y), d(x)) \sum   \langle f_{(1)}\ot f_{(2)}, S(x)\ot S(y) \rangle f_{(3)}\\
 &=\omega(d(y), d(x)) \sum   \langle \Delta_0(f_{(1)}), S(x)\ot S(y) \rangle f_{(2)}\\
 &=\sum   \langle f_{(1)}, S(y x) \rangle f_{(2)}. 
 \ealn
 \]
 Thus
 $
 L_y (L_x(f)) = \omega(d(y x), d(y x)) \sum   \langle f_{(1)}, S(y x) \rangle f_{(2)}$, 
 which is equal to $L_{y x}(f) 
 $ by \eqref{eq:L}. This shows that $L$ is indeed a left action of $A$ on $A^0$. 
 
 Now we prove that the two left actions $\omega$-commute.  Note that
 \[
 \baln
 R_y (L_x(f))&= \omega(d(x), d(x)) \sum   \langle f_{(1)}, S(x) \rangle  R_y(f_{(2)})\\
 &= \omega(d(x), d(x)) \sum   \omega(d(y), d(f_{(2)})+d(f_{(3)})) \langle f_{(1)}, S(x) \rangle  f_{(2)} \langle f_{(3)}, y \rangle\\
 &= \omega(d(x), d(x)) \omega(d(y), d(f)+d(x))  \sum   \langle f_{(1)}, S(x) \rangle  f_{(2)} \langle f_{(3)}, y \rangle, \\
 L_x (R_y(f))&= \omega(d(y), d(f))  \sum  L_x( f_{(1)}) \langle f_{(2)}, y \rangle \\
 &= \omega(d(y), d(f)) \omega(d(x), d(x))  \sum  \langle  f_{(1)}, S(x) \rangle  f_{(2)}  \langle f_{(3)}, y \rangle \\
 \ealn
 \] 
 Hence $R_y (L_x(f)) = \omega(d(y), d(x)) L_x (R_y(f))$. 
 
 This completes the proof of Lemma \ref{lem:R-L}.
 \end{proof}
 
 \begin{proof}[Proof of Lemma \ref{lem:mod-alg}]
  (1). 
 For $x\in A$, $f, g\in A^0$, we have 
  \beq\label{eq:Rxfg}
 R_x(f g)= \omega(d(x), d(f g))  \sum (f g)_{(1)} \langle (f g)_{(2)}, x \rangle.
 \eeq
 Note that $\Delta(f g) = \sum \omega(d(f_{(2)}), d(g_{(1)}) f_{(1)} g_{(1)} \ot  f_{(2)} g_{(2)}$. Hence 
  \[
 \baln
 \sum (f g)_{(1)} \langle (f g)_{(2)}, x \rangle 
 &= \sum \omega(d(f_{(2)}), d( g_{(1)}))  f_{(1)} g_{(1)} \langle f_{(2)} g_{(2)}, x \rangle, 
 \ealn
 \]
 where the right hand side can be expressed as 
 \[
 \baln
 &\sum \omega(d(f_{(2)}), d( g_{(1)})) \omega(d(g_{(2)}), d( x_{(1)}))  f_{(1)}  \langle f_{(2)},  x_{(1)}\rangle 
  g_{(1)} \langle g_{(2)}, x_{(2)} \rangle \\
  &=\sum \omega(d(g), d( x_{(1)}))  f_{(1)}  \langle f_{(2)},  x_{(1)}\rangle 
  g_{(1)} \langle g_{(2)}, x_{(2)} \rangle. 
  \ealn
  \]
 Using \eqref{eq:R}, we can re-write the right hand side  as  
  \[
  \baln
  &\sum \omega(d(g), d( x_{(1)}))   \omega(d(f), d( x_{(1)}))  \omega(d(g), d( x_{(2)})) 
 R_{x_{(1)}}(f)  R_{x_{(2)}}(g) \\
  &=\omega(d(g), d(x))   \sum  \omega(d(f), d( x_{(1)}))   \omega(d(f), d( x_{(2)}))
   \mu_0\circ(R_{x_{(1)}} \ot R_{x_{(2)}})(f\ot g). 
 \ealn
 \] 
Some easy manipulations of the commutative factors cast this expression into 
  \[
  \baln 
   &\omega(d(f g), d(x))   \sum \mu_0\circ(R_{x_{(1)}} \ot R_{x_{(2)}})(f\ot g)\\
   &= \omega(d(f g), d(x))    \mu_0\circ(R\ot R)_{\Delta(x)}(f\ot g). 
 \ealn
 \]
 Using this in \eqref{eq:Rxfg}, we obtain $R_x(f g)=\mu_0\circ(R\ot R)_{\Delta(x)}(f\ot g)$, proving the statement. 
 
 (2).  We have 
 \[
 \baln
 L_x(f g) &= \omega(d(x), d(x)) \sum  \langle (f g)_{(1)}, S(x) \rangle  (f g)_{(2)}\qquad\text{(by \eqref{eq:L})}\\
 &= \omega(d(x), d(x)) \sum  \omega(d(f_{(2)}), d(g_{(1)})) \langle f_{(1)} g_{(1)}, S(x) \rangle  f_{(2)} g_{(2)}\\
 &= \omega(d(x), d(x)) \sum  \omega(d(f_{(2)}), d(g_{(1)})) \langle f_{(1)}\ot  g_{(1)}, \Delta(S(x)) \rangle  f_{(2)} g_{(2)}. 
 \ealn
 \]
 The right hand side is equal to 
\[
 \baln
  &\omega(d(x), d(x)) \sum  \omega(d(f_{(2)}), d(g_{(1)}))  \omega(d(x_{(1)}), d(x_{(2)}))\\
 &\times  \omega(d(g_{(1)}), d(x_{(2)})) \langle f_{(1)}, S(x_{(2)})\rangle   \langle g_{(1)}, S(x_{(1)}) \rangle  f_{(2)} g_{(2)}. 
 \ealn
\]
 By manipulating the commutative factors, we can show that this is equal to 
%
\[
 \baln
 &\omega(d(x), d(x)) \sum  \omega(d(x_{(1)}), d(f)+d(x_{(2)}))  \\
 &\times   \langle f_{(1)}, S(x_{(2)})\rangle  f_{(2)} \langle g_{(1)}, S(x_{(1)}) \rangle  g_{(2)}. 
 \ealn
 \]
 Now we use \eqref{eq:L} to express this in terms of $L_{x_{(2)}}(f)$ and $L_{x_{(1)}}(g)$, obtaining 
 \[
 \baln
 & \omega(d(x), d(x)) \sum  \omega(d(x_{(1)}), d(f)+d(x_{(2)}))   \\
 &\times   \omega(d(x_{(1)}), d(x_{(1)})) \omega(d(x_{(2)}), d(x_{(2)})) L_{x_{(2)}}(f)L_{x_{(1)}}(g)\\
 &=  \sum  \omega(d(x_{(1)}), d(f)+d(x_{(2)}))   \omega(d(f), d(x_{(1)})) 
  \mu_0 \circ(L_{x_{(2)}}\ot L_{x_{(1)}})(f \ot g)\\
  &=  \sum  \omega(d(x_{(1)}), d(x_{(2)}))   \mu_0\circ (L_{x_{(2)}}\ot L_{x_{(1)}})(f \ot g)\\
   &=  \mu_0\circ (L\ot L)_{\Delta'(x)}(f \ot g).
 \ealn
 \]
Therefore,  $L_x\circ \mu_0= \mu_0\circ (L\ot L)_{\Delta'(x)}$ for all $x\in A$. 
 
 This completes the proof of Lemma \ref{lem:mod-alg}.
  \end{proof}
  
\begin{example} \label{eg:L-R-mod}
If $V$ is a simple right $A^0$-co-module,  
the space $T^{(V)}$ of matrix coefficients of $V$ is a simple module for $L_A\ot R_A$, 
which is isomorphic to $V^*\ot V$. 
\end{example}

\section{Representations of general linear Lie $(\Gamma, \omega)$-algebras}

We develop the basics of the structure and representations of 
the general linear Lie $(\Gamma, \omega)$-algebra $\gl(V(\Gamma, \omega))$. 
The material in this section is elementary; 
nevertheless there is a need to treat it methodically. 

We  point out that in the case $\Gamma=\Z_2\times Z_2$ and $\omega(\alpha, \beta)=(-1)^{\alpha_1\beta_2 + \alpha_2\beta_1}$,  irreducible  tensor representations of $\gl(V(\Gamma, \omega))$ were constructed in \cite{ISV} 
using Gelfand-Zetlin patterns. This was achieved by  
using Scheunert's cocycle twisting theorem  \cite[Theorem 2]{Sch79}  to relate them 
to Gelfand-Zetlin patterns for irreducible tensor representations
of the general linear Lie superalgebra $\gl_{m|n}(\C)$. 
\subsection{Root system and Weyl group}\label{sect:struct}
Let $V(\Gamma, \omega)$ be a finite dimensional $\Gamma$-graded vector space, and let $\gl(V(\Gamma, \omega))$
be the  general linear Lie $(\Gamma, \omega)$-algebra of $V(\Gamma, \omega)$. 
To simplify notation, we write $V=V(\Gamma, \omega)$, and $\gl(V)=\gl(V(\Gamma, \omega))$.

\subsubsection{Root system}  \label{sect:roots}

Let us briefly describe the structure of $\gl(V)$, 
which is very similar to that of the usual general linear Lie superalgebra.

Denote $m_\alpha = \dim V_\alpha$. Recall from Example \ref{eg:gl} that 
 $\Gamma_R=\{\alpha\in \Gamma\mid m_\alpha>0\}$.  
Let
$\Gamma_R^\pm=\Gamma_R\cap\Gamma^\pm$, then $\Gamma_R=\Gamma_R^+\cup \Gamma_R^-$.
Since $V$ is finite dimensional, $\Gamma_R$ is a finite set. 
There are $|\Gamma_R|!$ total orders for $\Gamma_R$. 
Call a total order distinguished if  
$\alpha<\beta$ for all $\alpha\in\Gamma_R^+$ and 
$\beta \in\Gamma_R^-$.

Fix a total order for $\Gamma_R$ (which may or may not be distinguished).
Choose an ordered basis $B_\alpha=\{e(\alpha)_i\mid 1\le i\le m_\alpha\}$ for each $V_\alpha$, 
and order the $B_\alpha$'s according to the given order for $\Gamma_R$.
Thus we have the ordered basis $B=\bigcup_{\alpha\in\Gamma_R} B_\alpha$ for $V$.  

Denote by $E(\alpha, \beta)_{i j}\in \Hom_\C(V_\beta, V_\alpha)$,  for $\alpha, \beta\in\Gamma_R$,  $1\le i\le m_\alpha$ and $1\le j \le m_\beta$, the matrix units relative the basis $B$. They obey the relations
 \beq
 &&E(\alpha, \beta)_{i j} e(\gamma)_k = \delta_{\beta \gamma} \delta_{j k} e(\alpha)_i,  \\
 &&
 E(\alpha, \beta)_{i j} E(\gamma, \delta)_{k\ell} = \delta_{\beta \gamma} \delta_{j k} E(\alpha, \delta)_{i \ell}.
 \eeq
The matrix units form a basis of $\gl(V)$. Using them, we can express the defining relations of $\gl(V)$ as
\beq \label{eq:CR}
\phantom{XXXX}
[E(\alpha, \beta)_{i j},  E(\gamma, \delta)_{k\ell} ]= \delta_{\beta \gamma} \delta_{j k} E(\alpha, \delta)_{i \ell} 
-\omega(\alpha-\beta, \gamma - \delta) \delta_{\alpha \delta} \delta_{\ell i} E(\gamma, \beta)_{k j}. 
\eeq

Let $\alpha,  \beta$ be given elements of $\Gamma_R$, and fix $i, j$ such that
with $1\le i\le m_\alpha$ and $1\le j\le m_\beta$. If  $\alpha=\beta$, we require $i\ne j$. 
Let $V_{\alpha, i; \beta, j}=\C e(\alpha)_i\oplus \C e(\beta)_j$. 
Then the set $\{E(\alpha, \beta)_{i j}, E(\beta, \alpha)_{ji},  E(\alpha, \alpha)_{i i}, E(\beta, \beta)_{j j}\}$ 
spans a Lie $(\Gamma, \omega)$-subalgebra, which we denote by $\gl(V_{\alpha, i; \beta, j})$. 
It is important to observe the following facts.
\begin{lemma}\label{lem:subalgs}
Retain notation above. 

$\bullet$
If both $\alpha,\beta$ belong to $\Gamma_R^+$ or $\Gamma_R^-$, then $\gl(V_{\alpha, i; \beta, j})\simeq\gl_2(\C)$.  

$\bullet$ 
If one of $\alpha, \beta$  belongs to $\Gamma_R^+$ and the other to $\Gamma_R^-$,  
then  $\gl(V_{\alpha, i; \beta, j})\simeq \gl_{1|1}(\C)$.
\end{lemma}
\begin{proof} 
As $\chi:=\omega(\alpha-\beta, \beta - \alpha)=\omega(\alpha, \alpha)\omega(\beta, \beta)$, we obtain from \eqref{eq:CR} 
\[
[E(\alpha, \beta)_{i j},  E(\beta, \alpha)_{j i} ]=  E(\alpha, \alpha)_{i i} 
- \chi E(\beta, \beta)_{j j}. 
\]
In the first case $\chi=1$, thus 
$[E(\alpha, \beta)_{i j},  E(\beta, \alpha)_{j i} ]= - [E(\beta, \alpha)_{j i}, E(\alpha, \beta)_{i j}]$. 
The subalgebra  $\gl(V_{\alpha, i; \beta, j})$ is isomorphic to $\gl_2(\C)$. 
In the second case $\chi=-1$, and the subalgebra is isomorphic to  $\gl_{1|1}(\C)$. 
\end{proof} 

Return to equation \eqref{eq:CR}.
In the special case $ \beta=\alpha$ and $j=i$,  it reduces to
\beq \label{eq:Cartan}
[E(\alpha, \alpha)_{i i},  E(\gamma, \delta)_{k\ell} ]= (\delta_{\alpha \gamma} \delta_{i k}  
-\delta_{\alpha \delta} \delta_{i \ell}) E(\gamma, \delta)_{k \ell}. 
\eeq

Let $\fh(\alpha)=\sum_{i=1}^{m_\alpha} \C E(\alpha, \alpha)_{i i}$, which  
is the standard Cartan subalgebra of $\gl(V_\alpha)$ for each $\alpha\in\Gamma_R$. 
Let $\fh=\sum_{\alpha\in\Gamma_R}\fh(\alpha)$. 
Equation \eqref{eq:Cartan} shows that $\fh$ is a Cartan subalgebra of $\gl(V)$, 
which enables one to describe the structure 
of $\gl(V)$ in the standard way by using root spaces. More specifically, we introduce a basis $\{\varepsilon(\alpha)_i\mid 1\le i \le m_\alpha, \ \alpha \in \Gamma_R\}$ for the dual space $\fh^*$ of $\fh$ such that 
\[
\varepsilon(\beta)_k(E(\alpha, \alpha)_{i i})=\delta_{\alpha \beta}\delta_{i k}, \quad \alpha, \beta\in \Gamma_R, 1\le i\le m_\alpha, \ 1\le k\le m_\beta. 
\]
Then \eqref{eq:Cartan} can be re-written as 
\beq \label{eq:roots}
[E(\alpha, \alpha)_{i i},  E(\gamma, \zeta)_{k\ell} ]= (\varepsilon(\gamma)_k -  \varepsilon(\zeta)_\ell)(E(\alpha, \alpha)_{i i}) E(\gamma, \zeta)_{k \ell}. 
\eeq
Write $ \Upsilon(\gamma, \zeta)_{k \ell}= \varepsilon(\gamma)_k -  \varepsilon(\zeta)_\ell$ if the element is non-zero. Let 
\[
\baln
\Phi(\gamma, \gamma)&=\{\Upsilon(\gamma, \gamma)_{k \ell}\mid   k, \ell=1, 2, \dots,  m_\gamma, k\ne \ell\}, \ \gamma\in\Gamma_R,  \\
\Phi(\gamma, \zeta)&=\{\Upsilon(\gamma, \zeta)_{k \ell}\mid 1\le k\le m_\gamma,  1\le \ell\le m_\zeta\}, \ \gamma, \zeta\in \Gamma_R, \gamma\ne \zeta.
\ealn
\]
Then $\Phi= \cup_{\gamma, \zeta\in\Gamma_R}\Phi(\gamma,\zeta)$ is the set of roots of $\gl(V)$. 
We take the set $\Phi^+$ of  positive roots to consist of the following elements.
\[
\baln
\Upsilon(\gamma, \zeta)_{k \ell},  &\text{ for }  \gamma, \zeta\in \Gamma_R, \gamma<\zeta,  
1\le k\le m_\gamma,  1\le \ell\le m_\zeta, \\
\Upsilon(\gamma, \gamma)_{k \ell}, & \text{ for } \gamma\in\Gamma_R,  1\le k<\ell\le m_\gamma. 
\ealn
\]
Then $\Phi=\Phi^+\cup(-\Phi^+)$. Define a map $\chi: \Phi\lra \{1, -1\}$ by
\beq
\chi(\Upsilon)={\omega(\alpha, \alpha)}{\omega(\beta, \beta)}, \   \Upsilon\in \Phi(\alpha, \beta),  \alpha, \beta \in \Gamma.
\eeq
Let 
$
\Phi_0=\{\Upsilon\in \Phi\mid \chi(\Upsilon)=1\}$ and $\Phi_1=\{\Upsilon\in \Phi\mid \chi(\Upsilon)=-1\}.
$
Then we have $\Phi^+=\Phi^+_0\cup \Phi^+_1$ with $\Phi^+_0= \Phi_0\cap\Phi^+$
and $\Phi^+_1= \Phi_1\cap\Phi^+$.
For later use, we define
 \beq\label{eq:rho}
 \rho=\frac{1}{2}\sum_{\Upsilon\in \Phi^+}\chi(\Upsilon) \Upsilon, \quad \rho_\iota=\frac{1}{2}\sum_{\Upsilon\in \Phi^+_\iota} \Upsilon, \ \iota=0, 1.
 \eeq
Thus $\rho=\rho_0-\rho_1$.

For any root $\Upsilon\in \Phi$, the corresponding root space is 
\[
\fg_\Upsilon =\{X\in \gl(V)\mid [h, X]=\Upsilon(h)X, \ \forall h\in\fh\},
\]
which is $1$-dimensional. 
The subspaces $\fn=\sum_{\Upsilon\in \Phi^+}\fg_\Upsilon$ and $\ol{\fn}=\sum_{\Upsilon\in \Phi^+}\fg_{-\Upsilon}$ are Lie $(\Gamma, \omega)$-subalgebras of $\gl(V)$, and we have the following triangular decomposition
\[
\gl(V)=\fn+\fh+\ol{\fn}.  
\]
It is clear but very important to note that $\fb:=\fn+\fh$ is a Borel subalgebra
containing the Cartan subalgebra $\fh$.

Let $V_\pm=\sum_{\alpha\in\Gamma_R^\pm}V_\alpha$, and thus $V=V_+\oplus V_-$.  
Denote $M_\pm=\dim V_\pm$, and  let $d_{max}=M_+M_-$. 
We have the Lie colour subalgebras $\gl(V_\pm)$.  If $\Upsilon\in\Phi_0$, then $\fg_\Upsilon$ is either contained in
$\gl(V_+)$ or in $\gl(V_-)$. Denote 
$\Phi^+_{0, \pm}=\{\Upsilon\in \Phi^+_0\mid \fg_\Upsilon\subset \gl(V_\pm)\}$. 
Then $\Phi^+_0=\Phi^+_{0, +}\cup \Phi^+_{0, -}$.

\subsubsection{Invariant bilinear forms} \label{sect:forms}
There is a generalised trace  ${\rm tr}_{(\Gamma, \omega)}: \End_\C(V)\lra\C$,  
the $\omega$-trace,  which is defined as follows \cite[\S8]{Sch83}. 
For any $\varphi =\sum_{\alpha, \beta\in\Gamma}  \varphi(\alpha, \beta)$ 
with $\varphi(\alpha, \beta)\in \Hom_\C(V_\beta, V_\alpha)$,
\[
{\rm tr}_{(\Gamma, \omega)}(\varphi) =\sum_{\alpha\in \Gamma}  \omega(\alpha, \alpha) {\rm tr}(\varphi(\alpha, \alpha)), 
\]
where ${\rm tr}$ denotes the usual trace.  Clearly this is  homogeneous of degree $0$. 
Write  $\varphi_0= \sum_{\alpha\in \Gamma}  \varphi(\alpha, \alpha)$, then
$
{\rm tr}_{(\Gamma, \omega)}(\varphi) =  {\rm tr}_{(\Gamma, \omega)}(\varphi_0).
$
[See Remark \ref{rmk:trace} for a more conceptual explanation of the generalised trace.]

\begin{lemma}
The generalised trace ${\rm tr}_{(\Gamma, \omega)}: \End_\C(V)\lra\C$ has the following property. 
For any $\varphi\in\End_\C(V)_\alpha$ and $\psi\in\End_\C(V)_\beta$, where $\alpha, \beta\in\Gamma$, 
\beq\label{eq:om-sym}
{\rm tr}_{(\Gamma, \omega)}(\varphi \psi)&=&\omega(\alpha, \beta) {\rm tr}_{(\Gamma, \omega)}( \psi \varphi).
\eeq
This implies that $ {\rm tr}_{(\Gamma, \omega)}([\eta, \zeta])=0$ for all $\eta, \zeta\in\End_\C(V)$. 
\end{lemma}
\begin{proof}
Recall that for any $\varphi(\alpha, \beta)\in \Hom_\C(V_\beta, V_\alpha)$, and $\psi
(\gamma, \zeta)\in \Hom_\C(V_\zeta, V_\gamma)$, we have $\varphi(\alpha, \beta) \psi(\gamma, \zeta)=\delta_{\beta\gamma} \varphi(\alpha, \beta) \psi(\gamma, \zeta)$, and the degree $0$ component is given by 
$(\varphi(\alpha, \beta) \psi(\gamma, \zeta))_0 = \delta_{\alpha\zeta}\delta_{\beta\gamma} \varphi(\alpha, \beta) \psi(\gamma, \zeta)$. Hence 
\[
\baln
{\rm tr}_{(\Gamma, \omega)}(\varphi(\alpha, \beta) \psi(\gamma, \zeta)) 
&=  \delta_{\alpha \zeta}\delta_{\beta \gamma}\omega(\alpha, \alpha) {\rm tr}(\varphi(\alpha, \beta) \psi(\gamma, \zeta))  \\
&=  \delta_{\alpha \zeta}\delta_{\beta \gamma}\omega(\alpha, \alpha) {\rm tr}(\psi(\gamma, \zeta) \varphi(\alpha, \beta)).
\ealn
\]
Re-write the expression in the second line in terms of the generalised trace, we obtain 
\[
\baln  
\omega(\alpha, \alpha) \omega(\beta, \beta){\rm tr}_{(\Gamma, \omega)}(\psi(\gamma, \zeta) \varphi(\alpha, \beta))  
= \omega(\alpha-\beta, \alpha-\beta) {\rm tr}_{(\Gamma, \omega)}(\psi(\gamma, \zeta) \varphi(\alpha, \beta) ), 
\ealn
\]
where the right hand side vanishes unless $\alpha-\beta=\delta-\gamma$. Thus $\omega(\alpha-\beta, \alpha-\beta)$ can be replaced by 
$\omega(\alpha-\beta, \delta-\gamma)$, proving \eqref{eq:om-sym}. The second statement is clear. 
\end{proof}

We refer to \eqref{eq:om-sym} as the $\omega$-symmetry of the generalised trace. 

Now we define the following bilinear form on $\End_\C(V)$.
\beq\label{eq:biform}
&&( \ , \ ): \End_\C(V)\times \End_\C(V)\lra \C, \\
&&( X,  Y )={\rm tr}_{(\Gamma, \omega)}(X Y), \quad \forall  X, Y\in \End_\C(V).\nonumber
\eeq
This form is homogeneous of degree $0$ and $\omega$-symmetric. By inspecting the relation
\beq\label{eq:form-basis}
\phantom{XX} 
( E(\alpha, \beta)_{i j},  E(\gamma, \zeta)_{k\ell}  ) = \omega(\alpha, \alpha) \delta_{\beta \gamma}\delta_{\alpha \zeta}\delta_{j k}\delta_{i \ell}, \quad \forall \alpha, \beta, \gamma, \zeta, i, j, k, \ell, 
\eeq
we easily see that the form is non-degenerate. 

\begin{lemma}[Definition] \label{lem:forms}
Equation \eqref{eq:biform} defines a non-degenerate $\omega$-symmetric bilinear form on $\gl(V)$, 
which is $ad$-invariant in the sense that 
\[
( [X,  Y], Z ) = ( X,  [Y, Z] ), \quad \forall X, Y, Z\in \gl(V). 
\]
Its restriction $( \ , \ )|_\fh: \fh\times \fh\lra\C$ to the Cartan subalgebra is
non-degenerate, and induces a non-degenerate symmetric bilinear form 
on $\fh^*$,
\beq\label{eq:inner-prod}
(\ , \ ): \fh^*\times \fh^*\lra \C
\eeq
 such that
\beq\label{eq:basis-prod}
(\varepsilon(\alpha)_i,  \varepsilon(\beta)_j) =\omega(\alpha, \beta) \delta_{\alpha \beta} \delta_{i j}, \quad \forall  \alpha, \beta, i, j.
\eeq
\end{lemma}
\begin{proof}
The $ad$-invariance of the bilinear form  \eqref{eq:biform} on $\gl(V)$ can be verified by direct calculations using 
the $\omega$-symmetry of ${\rm tr}_{(\Gamma, \omega)}$, and defining properties of the commutative factor $\omega$.  
It is clear from \eqref{eq:form-basis} that the restriction of the bilinear form to $\fh$ is non-degenerate, and symmetric. 

The induced form on $\fh^*$ is defined as follows. 
Associate to each $\mu\in\fh^*$ an element $t_\mu \in\fh$ such that $( t_\mu , h )=\mu(h)$ for all $h\in\fh$. 
Note that given any $\mu$, the corresponding Cartan element $t_\mu$ is unique because of the non-degeneracy 
of the form on $\fh$.   
Now the bilinear form \eqref{eq:inner-prod} on $\fh^*$ is defined by
\beq
(\mu, \nu)=( t_\mu , t_\nu ), \quad \forall \mu, \nu\in\fh^*.
\eeq

Consider the Cartan elements $t_{\varepsilon(\alpha)_i}$ associated with $\varepsilon(\alpha)_i$ for any $\alpha, i$. 
We have 
$
( t_{\varepsilon(\alpha)_i}, E(\beta, \beta)_{j j} ) 
= \varepsilon(\alpha)_i(E(\beta, \beta)_{j j})= \delta_{\alpha\beta}\delta_{i j}, 
$
for all $\beta, j.$
It follows from \eqref{eq:form-basis} that 
$t_{\varepsilon(\alpha)_i}= \omega(\alpha, \alpha)E(\alpha, \alpha)_{i i}$. 
Now 
\[
\baln
(\varepsilon(\alpha)_i,  \varepsilon(\beta)_j)&=( t_{\varepsilon(\alpha)_i}, t_{\varepsilon(\beta)_ j} )
=\omega(\alpha, \alpha) \omega(\beta, \beta) ( E(\alpha, \alpha)_{i i}, E(\beta, \beta)_{j j} )\\
&= \omega(\alpha, \beta) \delta_{\alpha \beta} \delta_{i j}, \qquad\qquad   \text{by \eqref{eq:form-basis}}.
\ealn
\] 
This proves \eqref{eq:basis-prod}, completing the proof of the lemma.  
\end{proof}

\subsubsection{The Weyl group of $\gl(V(\Gamma, \omega))$} \label{sect:W-gp}
We saw from Lemma \ref{lem:subalgs} that if $\alpha, \beta$ both belong to $\Gamma_R^+$ or $\Gamma_R^-$,
the Lie $(\Gamma, \omega)$-subalgebra $\gl(V_{\alpha, i; \beta, j})$
is isomorphic to the usual Lie algebra $\gl_2(\C)$.  In this case, $\Upsilon(\alpha, \beta)_{i j}\in\Phi^0$.

\begin{definition}\label{def:Weyl}
The  Weyl group  $W$ of $\gl(V(\Gamma, \omega))$ is the subgroup 
of the usual general linear group $\GL(\fh^*)$ generated 
by the reflections $\sigma_\Upsilon$ of $\fh^*$, for all  $\Upsilon\in\Phi_0^+$, 
defined by 
\beq
\sigma_\Upsilon(\mu)= \mu- \frac{2(\mu, \Upsilon)}{(\Upsilon, \Upsilon)}\Upsilon, \quad \forall \mu\in \fh^*. 
\eeq
\end{definition}

Since $(\varepsilon(\alpha)_i, \varepsilon(\beta)_k)=0$  if $\alpha\in\Gamma_R^+$ 
and $\beta\in\Gamma_R^-$ for all $i, k$,  we have 
$W=W_+\times W_-$, where $W_\pm$ are the Weyl groups for the root systems $\Phi_{0, \pm} = \Phi^+_{0, \pm}\cup(-\Phi^+_{0, \pm})$ of $\gl(V_\pm)$ respectively. Hence 
$W\cdot \Phi=\Phi$. Also note that
$W_\pm\simeq Sym_{M_\pm}$. 

A weight module $M$ for $\gl(V)$ is one such that 
 all $E(\alpha, \alpha)_{i i}$ act semi-simply.  Thus 
$M=\sum_{\mu\in\fh^*} M_\mu$, 
where $M_\mu=\{v\in M\mid h\cdot v=\mu(h)v, \ \forall h\in\fh\}$ is
the weight space of $\mu$. 
Denote by $\Pi(M)=\{\mu\in\fh^*\mid M_\mu\ne 0\}$ the set of weights of $M$. 
 
For any $\Upsilon=\varepsilon(\alpha)_ i-\varepsilon(\beta)_j\in\Phi_0^+$, let $V_\Upsilon=V_{\alpha, i; \beta, j}$.
Any finite dimensional weight module $M$ for $\gl(V)$ restricts to a semi-simple module 
for $\gl(V_\Upsilon)\simeq \gl_2(\C)$.  
Thus $\sigma_\Upsilon\cdot\Pi(M)=\Pi(M)$. This is true  for any   
$\gl(V_\Upsilon)\simeq \gl_2(\C)$, hence $\Pi(M)$ is stable under the action of $W$.

\subsection{Highest weight modules}\label{sect:modules}
The structure theory of $\gl(V)$ discussed in Section \ref{sect:roots},  and 
the existence of the Borel and Cartan subalgebras $\fb\supset \fh$ in particular, enable us to 
define highest weight $\gl(V)$-modules with respect to $\fb$ in the usual way.  

\medskip
{\bf We fix a distinguished order for $\Gamma_R$  hereafter}.
\medskip

We denote by $L_\lambda$ the simple $\gl(V)$-module with highest weight $\lambda\in\fh^*$. 
We can always write $\lambda= \sum_{\alpha\in\Gamma_R}\sum_{i=1}^{m_\alpha} \lambda(\alpha)_i \varepsilon(\alpha)_i$, 
with scalars $\lambda(\alpha)_i=\lambda(E(\alpha, \alpha)_{i i})$ for all  $i=1, 2, \dots, \dim V_\alpha$ and $\alpha\in\Gamma_R$.

Standard arguments from the theory of Lie algebras can show that 
a finite dimensional simple $\gl(V)$-module $L$ must be a highest weight module.  
Indeed, by using Lie's theorem (see, e.g., \cite[\S 4.1]{H}) to the Cartan subalgebra $\fh$, 
we conclude that there must exist in $L$ a common eigenvector $v$
of all elements of $\fh$ (since $\dim L<\infty$),  i.e., $v$ is a weight vector. 
As $\dim(\U(\fn)v)<\infty$, there exists a highest weight vector. This  
highest weight vector generates the simple module $L$ over  $\U(\gl(V))$.

\subsubsection{Parabolic induction}\label{sect:parabolic}

We have the Lie $(\Gamma, \omega)$-subalgebra $\fk=\gl(V_+)+ \gl(V_-)$.
Also $\fv:=\Hom_C(V_-, V_+)\subset \fn$ and $\ol\fv:=\Hom_C(V_+, V_-)\subset\ol\fn$ are subalgebras. 
In terms of root spaces, we have 
\[
\fv=\sum_{\Upsilon\in\Phi^+_1}\fg_\Upsilon, \quad  \ol{\fv}=\sum_{\Upsilon\in\Phi^+_1}\fg_{-\Upsilon}.
\]

Note that both $\fv$ and $\ol{\fv}$ are $\Gamma$-graded $\omega$-commutative, 
i.e., $[X, X']=0$ for any $X, X'\in\fv$,  and $[Y, Y']=0$ for $Y, Y'\in \ol\fv$.  
Thus in the respective universal enveloping algebras, 
homogeneous elements $X\in\fv$ and $Y\in \ol\fv$ satisfy
$X^2=0$  and  $Y^2=0$.  Hence by Theorem \ref{thm:PBW} (PBW theorem \cite{Sch79}), 
\beq
\dim \U(\fv)= \dim \U(\ol\fv) = 2^{M_+M_-}. 
\eeq
 We have the following decomposition of $\gl(V)$, 
\beq
\gl(V)=\ol\fv+\fk+\fv.
\eeq
Let $\fq=\fv+\fk$, which is a parabolic subalgebra. 

Observe that the universal enveloping algebra $\U(\gl(V))$ admits a $\Z$-grading, 
with $deg(\fv)=- deg(\ol{\fv})=1$, and $deg(\fk)=0$. 
Consider the subalgebras $\U(\fv)$ and $\U(\ol{\fv})$.   
Then $\U(\fv)=\sum_{d=0}^{d_{max}} \U(\fv)_d$ and $\U(\ol{\fv})=\sum_{d=0}^{d_{max}} \U(\ol{\fv})_{-d}$. 
Note that both $\U(\fv)_{d_{max}}$ and $\U(\ol{\fv})_{-d_{max}}$ are $1$-dimensional.

Given any simple $\fk$-module $L^0_\lambda(\fk)$ with highest weight $\lambda\in\fh^*$, 
we make it a $\fq$-module with $\fv\cdot L^0_\lambda(\fk)=0$. Construct the induced module 
\beq\label{eq:Verma}
V(L^0_\lambda(\fk))=\U(\gl(V))\ot_{\U(\fq)} L^0_\lambda(\fk), 
\eeq
which is isomorphic to $\U(\ol{\fv})\ot_{\C} L^0_\lambda(\fk)$ as vector space. 
It has a unique simple quotient $\gl(V)$-module $L_\lambda$.

The $\Z$-grading  of $\U(\gl(V))$ discussed in Section \ref{sect:parabolic}
 enables us to introduce a $\Z$-grading 
on $V(L^0_\lambda(\fk))$ by decreeing that 
$L^0_\lambda(\fk)$ is of degree $0$. 
Then
\beq\label{eq:Z-grad}
 V(L^0_\lambda(\fk))=\sum_{d=0}^{d_{max}} V(L^0_\lambda(\fk))_{-d}, \quad 
V(L^0_\lambda(\fk))_{-d}= \U(\ol{\fv})_{-d}\ot_{\C} L^0_\lambda(\fk).
\eeq
The simple module $L_\lambda$ inherits a $\Z$-grading from $V(L^0_\lambda(\fk))$ 
with $(L_\lambda)_0=1\ot  L^0_\lambda(\fk)$. 
Also note that $(L_\lambda)_0=\{ v\in L_\lambda\mid \fv\cdot v=\{0\}\}$. 

\subsubsection{Finite dimensional simple modules}
The analysis in this section is inspired by the Kac module construction 
for type I Lie classical superalgebras \cite{K}, and a generalisation of 
it to the associated quantum supergroups \cite{Z93, Z93b}. 

Since $V(L^0_\lambda(\fk))=\U(\ol{\fv})\ot L^0_\lambda(\fk)$,  
the simple quotient module $L_\lambda$ 
is finite dimensional if and only if $\dim L^0_\lambda(\fk)$ $<\infty$. 
Note that any finite dimensional simple $\fk$-module must be of the form 
$L^0_\lambda(\fk)$ for some $\lambda\in\fh^*$. 

Introduce the following subset of $\fh^*$. 
\begin{definition}
Let $\Lambda_{M_+|M_-}=\left\{\lambda\in\fh^*\left|\frac{2(\lambda, \Upsilon)}{(\Upsilon, \Upsilon)}\in\Z_+, \ \forall \Upsilon\in\Phi^+_0  \right\}\right.$. 
\end{definition}

Explicitly,  any element 
$\lambda\in \Lambda_{M_+|M_-}$  satisfies the following conditions: 
\beq\label{eq:domin}
\begin{array}{r l l}
(i). & \lambda(\alpha)_i - \lambda(\alpha)_j \in\Z_+,  & \forall \alpha \in \Gamma_R, \  i<j, \ \text{ and }\\
(ii). & \lambda(\alpha)_k - \lambda(\beta)_\ell\in\Z_+, & \forall k, \ell,  \text{ both $\alpha, \beta \in \Gamma_R^+$ or $\Gamma_R^-$,  $\alpha<\beta$}. 
\end{array}
\eeq

We have the following result. 
\begin{theorem} \label{thm:P}
Retain notation above. 
\begin{enumerate}
\item Any finite dimensional simple $\gl(V)$-module must be of highest weight type. 

\item A simple $\gl(V)$-module $L_\lambda$ with highest weight $\lambda$ is 
finite dimensional if and only if $\lambda\in \Lambda_{M_+|M_-}$. 
\end{enumerate}
\end{theorem}
\begin{proof}
Part (1) was already proved. 
By the preceding discussion, the proof of part (2) boils down to the special cases $V=V_+$ or $V_-$, 
where $\gl(V_\pm)$ are Lie colour algebras. 

Consider the subalgebras $\gl(V_\alpha)\subset \gl(V)$ for all $\alpha\in\Gamma_R$, which are usual general linear Lie algebras, and are contained in $\fk$.  Each $\gl(V_\alpha)$ inherits from $\gl(V)$ a triangular decomposition $\gl(V_\alpha)=\fn(\alpha)+\fh(\alpha)+\ol{\fn}(\alpha)$ with $ \fn(\alpha) = \gl(V_\alpha)\cap\fn$, $\fh(\alpha)=\gl(V_\alpha)\cap\fh$ and $\ol{\fn}(\alpha)=\gl(V_\alpha)\cap \ol{\fn}$.  A $\gl(V)$ highest weight vector in $L_\lambda$ is automatically a $\gl(V_\alpha)$ highest weight vector. It generates over $\gl(V_\alpha)$ a finite dimensional module only if condition \eqref{eq:domin}(i) is satisfied. 

Given any $\alpha, \beta\in\Gamma_R^+$ or $\alpha, \beta\in\Gamma_R^-$ such that $\alpha<\beta$, 
the elements $E(\alpha, \beta)_{m_\alpha, 1}$,  $E(\beta, \alpha)_{1, m_\alpha}$, 
$E(\alpha, \alpha)_{m_\alpha, m_\alpha}$ 
and $E(\beta, \beta)_{1 1}$ form a $\gl_2$ subalgebra in $\gl(V)$ by Lemma \ref{lem:subalgs}.   
A $\gl(V)$ highest weight vector in $L_\lambda$ is again automatically a highest weight vector for this $\gl_2$ (annihilated by $E(\alpha, \beta)_{m_\alpha, 1}$). Thus in order for $L_\lambda$ to be finite dimensional, we must have $\lambda(\alpha)_{m_\alpha} - \lambda(\beta)_1\in\Z_+$. This together with \eqref{eq:domin}(i) lead to \eqref{eq:domin}(ii),  proving the necessity of the condition
$\lambda\in \Lambda_{M_+|M_-}$ in order for $L_\lambda$ to be finite dimensional.  

We turn to the proof of the sufficiency of the conditions. 
Let $\fh_\pm =\gl(V_\pm) \cap \fh$. Given any $\lambda\in \fh^*$, we denote by $\lambda^\pm$  respectively the restrictions of $\lambda$ to $\fh_\pm$ as linear functions. More explicitly, $\lambda^\pm\in {\fh_\pm}^*$ are defined by $\lambda^+(h_+)=
\lambda(h_+)$ and $\lambda^-(h_-)= \lambda(h_-)$ for all $h_\pm\in\fh_\pm$. Then 
$L_\lambda^0(\fk) \simeq L^+_{\lambda^+}\ot L^-_{\lambda^-}$, 
where $L^\pm_{\lambda^\pm}$ are the simple $\gl(V_\pm)$-modules with highest weights $\lambda^\pm$ respectively.  

Assume that $\lambda\in \Lambda_{M_+|M_-}$. Then $\lambda^\pm$ satisfy the conditions  \eqref{eq:domin}(i)  and \eqref{eq:domin}(ii) 
for $\alpha, \beta \in \Gamma_R^+$ and  $\alpha, \beta \in \Gamma_R^-$ respectively. 
If we can show that both $L^\pm_{\lambda^\pm}$ are finite dimensional, i.e., 
the theorem is valid for $\gl(V_+)$ and $\gl(V_-)$, 
then $\dim L_\lambda^0(\fk)<\infty$, and the ``if part'' of the theorem follows.  
We will prove this in Corollary \ref{cor:non-super} below. 
\end{proof}

\subsection{Completing the proof of Theorem \ref{thm:P}}\label{sect:pf-P}

 \subsubsection{Refining  notation of Section \ref{sect:roots}}\label{sect:refined}
Let us first refine notation of Section \ref{sect:roots} to make it more convenient for explicit computations. 

Keep the chosen distinguished order for the set $\Gamma_R$.

Let $\aleph^\pm=|\Gamma_R^\pm|$ and $\aleph=|\Gamma_R|$. 
We write the elements of $\Gamma_R$ as $\alpha_t$ for $t=1, 2, \dots, |\Gamma_R|$ 
in such a way that 
$\alpha_t<\alpha_{t+1}$,  for all valid $t$. 
Thus $\alpha_t$ belongs to $\Gamma^+$  if $t\le \aleph^+$, and belongs to $\Gamma^-$ if $t> \aleph^+$.  

Given an element $\lambda\in\fh^*$, we denote $\lambda(\alpha_t)_i=\lambda(E(\alpha_t, \alpha_t)_{i i})$ for all $t$ and $i$. 
Then 
$\lambda= \sum_{t=1}^{\aleph} \sum_{i=1}^{m_{\alpha_t}} \lambda(\alpha_t)_i \varepsilon(\alpha_t)_i$. 
We introduce the 
$m_{\alpha_t}$-tuples
\[
\lambda(\alpha_t)= (\lambda(\alpha_t)_1, \lambda(\alpha_t)_2, \dots, \lambda(\alpha_t)_{m_{\alpha_t}}), \quad 1\le t\le |\Gamma_R|, 
\]
and write  $\lambda$ as
\beq\label{eq:wt-tuple}
\lambda =(\lambda(\alpha_1),\lambda(\alpha_2), \dots, \lambda(\alpha_{|\Gamma_R|})) \in \C^{\dim V}. 
\eeq
Note in particular that the highest weight of the natural $\gl(V)$-module $V$ is given by $(1, 0, \dots, 0)$, and the highest weight of the dual module $V^*$ is $(0, \dots, 0, -1)$. 

The elements of the basis of $V$ introduced  in Section \ref{sect:roots}
are now written as $e(\alpha_r)_i$ for $i=1, 2, \dots, m_r:=m_{\alpha_r}$ and $r=1, 2, \dots, \aleph$.  Let 
\[
\baln
B(\alpha_r)&:=(e(\alpha_r)_1, e(\alpha_r)_2, \dots, e(\alpha_r)_{m_r}), \quad \forall r, \\
B^+&:= (B(\alpha_1), B(\alpha_2), \dots, B(\alpha_{\aleph^+})), \\
B^-&:= (B(\alpha_{\aleph^++1}), B(\alpha_{\aleph^++2}), \dots, B(\alpha_{\aleph})), \\
B&:=(B^+, B^-).
\ealn
\]
Then $B^\pm$ and $B$ are ordered bases for $V_\pm$ and $V$ respectively. 
To simplify notation, we write 
\[
\baln
B^+=(b^+_1, b^+_2, \dots, b^+_{M_+}), \quad
B^-=(b^-_1, b^-_2, \dots, b^-_{M_-}), \quad
B=(b_1, b_2, \dots, b_{\dim V}). 
\ealn
\]
Note that $B$ is the same homogeneous basis of $V$ given in Section \ref{sect:roots}. 
Denote the degree of $b_a$ by $\gamma_a:=d(b_a)\in \Gamma$  for all $a$.

The matrix units relative to $B$ will now be denoted by $\BE_{a b}$ with $a, b=1, 2, \dots,$ $ \dim V$, 
which form a basis of $\End_\C(V)$. 
They are homogeneous with degrees $deg(\BE_{a b})$ $=\gamma_a-\gamma_b$. 
Their graded commutation relations are given by 
\beq\label{eq:BE-relats}
[\BE_{a b}, \BE_{c d}]=\delta_{b c} \BE_{a d} - \omega(\gamma_a -\gamma_b, \gamma_c -\gamma_d) \delta_{d a} \BE_{c b}, \quad \forall a, b.
\eeq
In terms of these elements, 
\beq
&&\fn=\sum_{a<b}\C \BE_{a b}, \quad \ol{\fn}=\sum_{a>b}\C \BE_{a b}, \quad \fh= \sum_{a=1}^{\dim V} \C \BE_{a a}, \\
&&\fk=\sum_{a, b\le M_+}\C \BE_{a b}+\sum_{a, b> M_+}\C \BE_{a b},  \\
&&\fv=\sum_{a\le M_+, b> M_+}\C \BE_{a b}, \quad  \ol{\fv}=\sum_{a\le M_+, b> M_+}\C \BE_{b a}.
\eeq

Denote the weight of $b_a$ by $\varepsilon_a\in\fh^*$. 
Then $\varepsilon_a(\BE_{b b})=\delta_{a b}$ and $(\varepsilon_a, \varepsilon_b)
=\omega(\gamma_a, \gamma_b)\delta_{a b}$ for all $a, b$. 
Write $\ol{r}=M_++r$ for any $r=1, 2, \dots, M_-$. We have 
\[
\baln
\Phi^+&=\{\varepsilon_a-\varepsilon_b\mid 1\le a<b\le \dim V\},\\
\Phi^+_0&=\{\varepsilon_i-\varepsilon_j\mid 1\le i<j\le M_+\}\cup 
\{\varepsilon_{\ol{r}}-\varepsilon_{\ol{s}}\mid 1\le r<s\le M_-\}, \\
\Phi^+_1&=\{\varepsilon_i-\varepsilon_{\ol{r}}\mid 1\le i<j\le M_+, 1\le r\le M_-\}.
\ealn
\]
In the new notation, \eqref{eq:wt-tuple} can be re-written as 
\beq
\lambda =(\lambda_1, \lambda_2, , \dots, \lambda_{\dim V}),  \text{ with $\lambda_a=\lambda(\BE_{a a})$, $\forall a$},
\eeq
and \eqref{eq:rho} becomes
\beq\label{eq:rho-ir}
\phantom{XX}
\rho&=&\frac{1}{2}\sum_{1\le a<b\le \dim V}\omega(\gamma_a, \gamma_a) 
\omega(\gamma_b, \gamma_b) (\varepsilon_a-\varepsilon_b)\\
&=& \frac{1}{2}\left(\sum_{i=1}^{M_+}(M_+-M_- -2i+1)\varepsilon_i   
+ \sum_{r=1}^{M_-} (M_++M_- -2 r +1)\varepsilon_{\ol{r}} \right).\nonumber
\eeq 
Note in particular that
$
(\rho, \varepsilon_i-\varepsilon_{\ol{r}})=M_+ -r -i+1. 
$

Observe that in the refined notation, many expressions formally look the same as those 
for the usual general linear Lie superalgebra $\gl_{M_+|M_-}(\C)$.

\subsubsection{Commutativity of the $\gl(V)$ and $\Sym_r$-actions on $V^{\ot r}$}\label{sect:sym-reps}
Recall that we have the representation $\nu_r: \C\Sym_r\lra \End_\C(V^{\ot r})$ of the group algebra 
of $\Sym_r$ constructed in Section \ref{sect:braid}. 
Observe the following fact.
\begin{lemma}\label{lem:gl-equiv}
The representation 
$\nu_r$ of $\C\Sym_r$ defined by Corollary \ref{cor:sym} has the property that 
$\im(\nu_r)\subset \End_{\gl(V)}(V^{\ot r})$. 
\end{lemma}

This is a consequence of the following more general fact:
given any $\Gamma$-graded vector spaces $W, W'$,  
the symmetry $\tau_{W, W'}: W\ot W'\lra W'\ot W$ satisfies
\beq\label{eq:tau-endo}
\tau_{W, W'}(A\ot B) =  \omega(d(A), d(B)) (B\ot A)\tau_{W, W'}, \label{eq:PAB} \\
 \forall A\in\End_\C(W), B\in\End_\C(W'). \nonumber
\eeq
This can be verified by the following direct computation: for any $v\in W, w\in W'$, 
\[
\baln
\tau_{W, W'}(A\ot B)(v\ot w)
&= \omega(d(B), d(v)) \omega(d(A)+d(v), d(B)+d(w)) B w\ot A v\\
&= \omega(d(A), d(B)) \omega(d(A)+d(v), d(w)) B w\ot A v\\
&=\omega(d(A), d(B)) \omega(d(v), d(w)) (B\ot A) (w\ot v)\\
&=\omega(d(A), d(B)) (B\ot A) \tau_{W, W'}(v\ot w). 
\ealn
\]

Assume that $W$ and $W'$ are $\gl(V)$-modules. We write the associated $\gl(V)$-representations  
as $\pi_{W}$ and $\pi_{W'}$ respectively. 
Then \eqref{eq:tau-endo} immediately implies 
\beq\label{eq:tau-gl}
\tau_{W, W'}(\pi_{W}\ot \pi_{W'})\Delta(u) = (\pi_{W'}\ot \pi_{W})\Delta(u) \tau_{W, W'}, \quad \forall u\in\U(\gl(V))
\eeq
(where we have use the fact that the opposite co-multiplication $\Delta'=\Delta$ in the present case).
Therefore,  $P:=\tau_{V, V}$ belongs to $\End_{\gl(V)}(V\ot V)$.

Denote by $\pi_r: \U(\gl(V))\lra \End_\C(V^{\ot r})$ the representation of $\U(\gl(V))$ on $V^{\ot r}$. 
Then $P\pi_2(u)=\pi_2(u)P$ for all $u\in \U(\gl(V))$, and hence
 the morphisms $\sigma_i$ satisfy 
\[
\sigma_i \pi_r(u)=  \pi_r(u) \sigma_i, \quad \forall u\in  \U(\gl(V)). 
\]
That is,  $\sigma_i\in \End_{\gl(V)}(V^{\ot r})$ for all $i$. 

\begin{remark}
For  finite dimensional $V$, we have 
\beq \label{eq:P}
P= \sum_{\alpha, \beta\in\Gamma_R} \sum_{i=1}^{m_\alpha}\sum_{j=1}^{m_\beta} \omega(\beta, \beta) E(\alpha, \beta)_{i j}\ot E(\beta, \alpha)_{j i}. 
\eeq
This can be verified by the following calculations.
\[
\baln
P(e(\gamma)_k\ot e(\zeta)_\ell)&=\sum  \omega(\beta, \beta) \omega(\beta-\alpha, \gamma)
E(\alpha, \beta)_{i j} e(\gamma)_k\ot E(\beta, \alpha)_{j i}e(\zeta)_\ell			\\ 
&=  \sum  \omega(\beta, \beta) \omega(\beta-\alpha, \gamma) \delta_{\beta \gamma} \delta_{j k}
\delta_{\alpha \zeta} \delta_{i\ell}
e(\alpha)_i \ot e(\beta)_j	 \\
&= \omega(\gamma, \gamma) \omega(\gamma-\zeta, \gamma)   e(\zeta)_\ell\ot e(\gamma)_k	 \\
&=\omega(\gamma, \zeta)e(\zeta)_\ell\ot e(\gamma)_k, \qquad\qquad \forall e(\gamma)_k, e(\zeta)_\ell.
\ealn
\] 
\end{remark}

We denote the total symmetriser and skew symmtriser respectively by 
\beq\label{eq:sym-skew}
\Sigma^\pm(r) =\sum_{\sigma\in\Sym_r} (\mp 1)^{|\sigma|}\sigma.
\eeq
Let $\{ v_1, v_2, \dots \}$ be a subset of homogeneous elements of $V$. For any sequence $(v_{i_1},  v_{i_2}, \dots, v_{i_n})$ of length $n$ in the set, we  write
${\bf v}(i_1, i_2, \dots, i_n) = v_{i_1}\ot v_{i_2}\ot \dots \ot v_{i_n}$.
Observe the following fact. 

\begin{lemma}\label{lem:a-symm}
Let ${\bf v}(i_1, i_2, \dots, i_n)$ be as defined above. Assume that for some $a \ne  b$,  the vectors $v_{i_a}$ and $v_{i_b}$ satisfy $v_{i_a}=v_{i_b}=v_j$ for some $j$. Then  
\beq
&\Sigma^+(n)({\bf v}(i_1, i_2, \dots, i_n)) = 0, \quad \text{if $v_j\in V_+$};
\label{eq:a-symm}\\
&\Sigma^-(n)({\bf v}(i_1, i_2, \dots, i_n))=0,  \quad \text{if $v_j\in V_-$}.
\label{eq:a-symm-1}
\eeq
\end{lemma}
\begin{proof}
We can assume $a<b$. Let us express the permutation $(a, b)$ as  $(a, b)=\sigma s_a \sigma'$, 
with $\sigma=s_{b-1} s_{b-2}  \dots s_{a+1}$ and $\sigma'=s_{a+1}\dots s_{b-2} s_{b-1}$. 
Then $s_a \sigma'({\bf v}(i_1, i_2, \dots, i_n)) $ is equal to 
\[
\baln
\omega\left(\sum_{\ell=a}^{b-1} d(v_{i_\ell}), d(v_{i_b})\right) {\bf v}(i_1, \dots, i_{a-1}, i_b, i_a, i_{a+1},  \dots, i_{b-1}, i_{b+1},\dots,  i_n), 
\ealn
\]
and $\sigma({\bf v}(i_1, \dots, i_{a-1}, i_b, i_a, i_{a+1},  \dots, i_{b-1}, i_{b+1},\dots,  i_n))$ 
is equal to 
\[
\omega\left(d(v_{i_a}), \sum_{\ell=a+1}^{b-1} d(v_{i_\ell})\right) {\bf v}(i_1, \dots, i_{a-1}, i_b, i_{a+1},  \dots, i_{b-1}, i_a, i_{b+1},\dots,  i_n). 
\]
Write $\omega_{(a, b)}({\bf i})= \omega\left(d(v_{i_a}),  d(v_{i_b})\right) \omega\left(d(v_{i_a})- d(v_{i_b}), \sum_{\ell=a+1}^{b-1} d(v_{i_\ell})\right)$.  We have 
\beq
&\ &(a, b)({\bf v}(i_1, i_2, \dots, i_n)) \\
&&=\omega_{(a, b)}({\bf i})  {\bf v}(i_1, \dots, i_{a-1}, i_b, i_{a+1},  \dots, i_{b-1}, i_a, i_{b+1},\dots,  i_n). \nonumber
\eeq

If $i_a=i_b=j$, then $\omega_{(a, b)}({\bf i})=\omega\left(d(v_j), d(v_j)\right)$.  
The  above equation reduces to
\[
(a, b)({\bf v}(i_1, i_2, \dots, i_n)) = \omega\left(d(v_j), d(v_j)\right) {\bf v}(i_1, i_2, \dots, i_n).
\]
Using $\Sigma^\pm (n) (a, b)=\mp \Sigma^\pm (n)$, we obtain 
\[
\baln
  \Sigma^\pm (n)({\bf v}(i_1, i_2, \dots, i_n)) 
 &=\mp \omega\left(d(v_j), d(v_j)\right) \Sigma^\pm (n) ({\bf v}(i_1, i_2, \dots, i_n)).
\ealn
\]
Since $\omega\left(d(v_j), d(v_j)\right) =1$ if $v_j\in V_+$ and $-1$ if 
$v_j\in V_+$, the above equation leads to 
\[
\baln
  \Sigma^\pm (n)({\bf v}(i_1, i_2, \dots, i_n)) 
 &=- \Sigma^\pm (n) ({\bf v}(i_1, i_2, \dots, i_n)), 
\ealn
\]
which immediately leads to the lemma.
\end{proof}

\subsubsection{Highest weight vectors for $\gl(V_\pm)$}\label{sect:hwv-0}
In the notation of Section \ref{sect:refined}, the Lie $(\Gamma, \omega)$-subalgebra $\gl(V_+)$ of $\gl(V)$ 
 is spanned by $\BE_{i j}$ for $i, j =1, 2, \dots, M_+$, and  
$\gl(V_-)$ by $\BE_{\ol{r} \, \ol{s}}$ for $r, s=1, 2, \dots, M_-$. 
We study representations of $\gl(V_\pm)$
by adapting symmetric group methods (see e.g., \cite[\S 9.1]{GW})
to the $\Gamma$-graded setting. 

Now consider ${V_\pm}^{\ot r}$ as $\gl(V_\pm)$-modules respectively. 
Then $\Lambda^r_\pm:=\Sigma^\pm(r)({V_\pm}^{\ot r})$ are $\gl(V_\pm)$-submodules of ${V_\pm}^{\ot r}$ respectively, 
which are non-zero if $r\le M_\pm$.  There are the following $\gl(V_\pm)$-highest weight vectors in the respective submodules
\beq
{\bf v}^\pm_{\omega^\pm_r} = \Sigma^\pm(r)(b^\pm_1\ot  b^\pm_2\ot  \dots\ot  b^\pm_r), 
\eeq
the weights of which are
\beq
\omega^+_r&= (\underbrace{1, \dots, 1}_r,  \underbrace{0, \dots, 0}_{M_+-r}), \quad  r=1, 2, \dots, M_+,\\
\omega^-_r&= (\underbrace{1, \dots, 1}_r,  \underbrace{0, \dots, 0}_{M_--r}), \quad  r=1, 2, \dots, M_-.
\eeq
Using Lemma \ref{lem:a-symm}, one can easily show
that ${\bf v}^\pm_{\omega^\pm_r}$ are $\gl(V_\pm )$ highest weight vectors. 
The following fact is also clear. 
\begin{lemma}
Given any $\mu^\pm_{i_\pm}\in \Z_+$ for $i_\pm=1, 2, \dots, M_\pm$, let ${\bf v}^\pm_{\mu'_\pm}\in V_\pm^{\ot\sum_{i=1}^{M_\pm} r \mu^\pm_r}$ be defined by 
\[
{\bf v}^\pm_{\mu'_\pm}:={{\bf v}^\pm_{\omega^\pm_{M_\pm}}}^{\ot \mu^\pm_{M_\pm}} 
 \ot {{\bf v}^\pm_{\omega^\pm_{M_\pm-1}}}^{\ot \mu^\pm_{M_\pm-1}} 
 \ot \dots
\ot {{\bf v}^\pm_{\omega^\pm_2}}^{\ot \mu^\pm_2}
\ot {{\bf v}^\pm_{\omega^\pm_1}}^{\ot \mu^\pm_1}.
\] 
Then the vectors are $\gl(V_\pm)$-highest weight vectors with weights $\mu'_\pm:=\sum_{r=1}^{M_\pm} \mu^\pm_r \omega^\pm_r$ respectively. 
\end{lemma}

These vectors ${\bf v}^\pm_{\mu'_\pm}$ generate $\gl(V_\pm)$-submodules 
in $V_\pm^{\ot\sum_{i=1}^{M_\pm} r \mu^\pm_r}$ respectively, which are finite dimensional. 
Thus we conclude that
\begin{corollary}\label{cor:non-super}
If $V_+$ or $V_-$ is zero, thus $\gl(V)$ is a Lie colour algebra, then
Theorem \ref{thm:P} holds.
\end{corollary}

\subsubsection{Completing the proof of Theorem \ref{thm:P}}\label{sect:pf-fd}
\begin{proof}
For any $\lambda =(\lambda(\alpha_1),\lambda(\alpha_2), \dots, \lambda(\alpha_\aleph))$ satisfying the conditions \eqref{eq:domin}, we have 
\[
\baln
\lambda^+&=(\lambda(\alpha_1),\lambda(\alpha_2), \dots, \lambda(\alpha_{\aleph^+})),  \\
\lambda^-&=(\lambda(\alpha_{{\aleph^+}+1}),\lambda(\alpha_{{\aleph^+}+2}), \dots, \lambda(\alpha_\aleph)).
\ealn
\]
Let $\wt\lambda^\pm= \lambda^\pm - {\mathbf d}^\pm$, with
\[
\baln
{\mathbf d}^+&=(\lambda(\alpha_{\aleph^+})_{m_{\alpha_{\aleph^+}}}, \dots, \lambda(\alpha_{\aleph^+})_{m_{\alpha_{\aleph^+}}})\in \C^{M_+},\\
{\mathbf d}^-&=(\lambda(\alpha_\aleph)_{m_{\alpha_\aleph}}, \dots, \lambda(\alpha_\aleph)_{m_{\alpha_\aleph}})\in \C^{M_-}.
\ealn
\]
Then $\wt\lambda^\pm\in\Z_+^{M_\pm}$, and there exists $\mu^\pm_r \in\Z_+$ such that $\wt\lambda^\pm = \sum_{r=1}^{M_\pm} \mu^\pm_r \omega^\pm_r$. 
We have  $L^\pm_{\lambda^\pm} \simeq L^\pm_{{\mathbf d}^\pm}\ot  L^\pm_{\wt\lambda^\pm}$, 
where $\dim L^\pm_{{\mathbf d}^\pm}=1$. Hence
\[
L_\lambda^0(\fk) \simeq \left(L^+_{{\mathbf d}^+}\ot L^-_{{\mathbf d}^-}\right)\ot  \big(L^+_{\wt\lambda^+}\ot L^-_{\wt\lambda^-}\big). 
\]
Thus $L_\lambda^0(\fk)$ is finite dimensional 
if and only if $\dim L^\pm_{\wt\lambda^\pm}<\infty$. 

The $\gl(V_\pm)$ highest weight vectors ${\bf v}^\pm_{\wt\lambda^\pm}$
generate $\gl(V_\pm)$-submodules in $V_\pm^{\ot | \wt\lambda^\pm |}$, where $| \wt\lambda^\pm |=\sum_{i=1}^{M_\pm} r \mu^\pm_r$. 
These submodules are finite dimensional, and have simple quotient modules respectively isomorphic to 
$L^\pm_{\wt\lambda^\pm}$ (with appropriate choices of gradings). 
Hence $L^\pm_{\wt\lambda^\pm}$ must be finite dimensional, and so are also $L^\pm_{\lambda^\pm}$. 
This suffices to garantee that the simple $\gl(V)$-module
$L_{(\lambda^+, \lambda^-)}$ is finite dimensional, 
completing the proof of Theorem \ref{thm:P}. 
\end{proof}

\subsection{Typical and atypical modules}\label{sect:typical}
We address the question when the generalised Verma module $V(L^0_\lambda(\fk))$ (cf. \eqref{eq:Verma}) is simple.  
Adopting terminology from the theory of Lie superalgebras \cite{K}, we will call a simple $\gl(V)$-module 
$L_\lambda$ typical if $L_\lambda=V(L^0_\lambda(\fk))$, and atypical otherwise. We will also call $\lambda$ 
 typical (resp. atypical) if $L_\lambda$ is. 

We have the following result. 

\begin{lemma}\label{lem:typical}
The simple $\gl(V)$-module 
$L_\lambda$ is typical if and only if 
\beq\label{eq:typical}
\prod_{\Upsilon\in \Phi^+_1}(\lambda+\rho, \Upsilon)\ne 0.
\eeq
\end{lemma}

Let us first make some preparations for the proof of the lemma. 
We maintain the notation of Section \ref{sect:refined}. 
Introduce the following elements of $\U(\gl(V))$. 
\[
\baln
&\ol{\BE}_i= \BE_{\ol{M_-}, i} \BE_{\ol{M_-}-1, i}\dots \BE_{\ol{1} i}, \quad \BE_i= \BE_{i \ol{1}} \BE_{i \ol{2}} \dots  \BE_{i, \ol{M_-}}, \\
&\ol{\BE}= \ol{\BE}^{(1)}= \ol{\BE}_1\ol{\BE}_2\dots \ol{\BE}_{M_+}, \quad \BE=\BE_{M_+}\BE_{M_+-1}\dots \BE_1, \\
&\ol{\BE}^{(i)}= \ol{\BE}_i\ol{\BE}_{i+1}\dots \ol{\BE}_{M_+}.
\ealn
\]
Then $\U(\fv)_{d_{max}} =\C\BE$ and $\U(\ol{\fv})_{-d_{max}}=\C\ol{\BE}$.
It is easy to see that 
\[
\baln
&\deg(\BE_i)=-deg(\ol{\BE}_i)= M_-\gamma_i-\sum_{r=1}^{M_-} \gamma_{\ol{r}}, \\
&\deg(\BE)=-deg(\ol{\BE})= M_-\sum_{i=1}^{M_+} \gamma_i- M_+ \sum_{r=1}^{M_-} \gamma_{\ol{r}}.
\ealn
\]

\begin{lemma}
The elements $\BE$, $\ol{\BE}$ and $\BE\ol{\BE}$ satisfy the following relations for all 
$i, j\le M_+$ and $r,  s\le  M_-$,
\beq
&& \BE_{\ol{r} \, \ol{s}} \ol{\BE}_i =\omega(\gamma_{\ol{r}}-\gamma_{\ol{s}},  d(\ol{\BE}_i)) \ol{\BE}_i \BE_{\ol{r}\,  \ol{s}}, \quad r\ne s, \label{eq:BE-0} \\
&&\BE_{i j} \BE = \omega(\gamma_i-\gamma_j, d(\BE)) \BE \BE_{i j}, \quad \BE_{i j} \ol{\BE} = \omega(d(\BE), \gamma_i-\gamma_j) \ol{\BE} \BE_{i j},  \quad  i\ne j, \label{eq:BE-1}\\
&&\BE_{\ol{r}\, \ol{s}} \BE = \omega(\gamma_{\ol{r}}-\gamma_{\ol{s}}, d(\BE)) \BE \BE_{\ol{r}\, \ol{s}}, \quad \BE_{\ol{r}\,  \ol{s}} \ol{\BE} = \omega(d(\BE), \gamma_{\ol{r}}-\gamma_{\ol{s}}) \ol{\BE} \BE_{\ol{r}\,  \ol{s}},\quad   r\ne s,  \label{eq:BE-2}\\
&&\BE_{i j} \BE\ol{\BE}=\BE\ol{\BE}\BE_{i j}, \quad  \BE_{\ol{r}\, \ol{s}} \BE\ol{\BE}=\BE\ol{\BE}\BE_{\ol{r}\, \ol{s}}.  \label{eq:BE-3}
\eeq
\end{lemma}
\begin{proof}
Equation \eqref{eq:BE-0} is equivalent to $[\BE_{\ol{r} \, \ol{s}},  \ol{\BE}_i ]=0$ for $r\ne s$, 
which is quite clear 
from the facts that the elements $E_{\ol{t} k}$ graded commute, and each satisfies $E_{\ol{t} k}^2=0$. 

Consider $[\BE_{i j},  \BE]$ for all $i, j\le M_+$. We have 
\beq
[\BE_{i j},  \BE] &= \omega\big(\gamma_i-\gamma_j, \sum_{\ell=j+1}^{M_+}d(\BE_\ell)\big)\BE_{M_+}\dots \BE_{j+1}[\BE_{i j},  \BE_j]
\BE_{j-1}\dots \BE_{1}. \nonumber
\eeq
Assume that $i\ne j$. Then 
\[
[\BE_{i j},  \BE_j]= \sum_{\ol{r}}  c_{\ol{r}}\BE_{j \ol{1}} \dots \BE_{j,  \ol{r}-1}  \BE_{i \ol{r}} \BE_{j,  \ol{r}+1} \dots \BE_{j \ol{M_-}}, 
\]
where each $c_{\ol{r}}$ is a scalar depending on $\ol{r}$ only. 
[An explicit formula for $c_{\ol{r}}$ can be easily obtained, but is not needed.]  
Since $\fv$ is graded commutative, and $\BE_i \BE_{i\ol{r}}=0$, we have $[\BE_{i j},  \BE]=0$ for all $i\ne j$, proving the first relation of \eqref{eq:BE-1}. We can similarly prove the second relation  of \eqref{eq:BE-1} and also the relations  \eqref{eq:BE-2}

If $i=j$, then $[\BE_{i i},  \BE_i]= M_- \BE_i$. Hence $[\BE_{i i},  \BE]=M_- \BE$. We can similarly prove $[\BE_{i i},  \ol{\BE}]=-M_- \ol{\BE}$. This and \eqref{eq:BE-1} together imply  the first relation of  \eqref{eq:BE-3}. The proof for the second 
relation of  \eqref{eq:BE-3} is similar but uses \eqref{eq:BE-2}.
\end{proof}

We can now prove Lemma \ref{lem:typical}.

\begin{proof}[Proof of Lemma \ref{lem:typical}]
Note that  $L_\lambda=V(L^0_\lambda(\fk))$ if and only if
$\U(\ol{\fv})_{-d_{max}}(L_\lambda)_0 \ne 0$. Since $L_\lambda$ is simple,  this happens  if and only if
\beq\label{eq:typical-2}
\U(\fv)_{d_{max}} \U(\ol{\fv})_{-d_{max}}(L_\lambda)_0 = \BE\ol{\BE}(L_\lambda)_0\ne 0.
\eeq
It follows \eqref{eq:BE-3} that $\BE\ol{\BE}$ commutes with $\fk$.  Since $(L_\lambda)_0$ is a simple $\fk$-submodule, $\BE\ol{\BE} $ acts on $(L_\lambda)_0$ by the multiplication of some scalar $\chi(\lambda)$. Thus the proof 
of Lemma \ref{lem:typical}  boils down to showing that  \eqref{eq:typical} is equivalent to $\chi(\lambda)\ne 0$.

Let $v^+$ be the highest weight vector of $L_\lambda$. Then $\BE\ol{\BE} v^+=\chi(\lambda) v^+$. We can determine $\chi(\lambda)$ by calculating $\BE\ol{\BE} v^+$. 

We have the following relations.
\beq
&&\BE_{i j} \ol{\BE}^{(\ell)}v^+=0, \quad 1\le i<j \le M_+, \label{eq:R-1}\\
&&\BE_{\ol{r}\, \ol{s}} \ol{\BE}^{(\ell)}v^+=0, \quad 1\le r<s \le M_-, \label{eq:R-2}\\
&&\BE_{i j} \BE_{\ol{t} \ell} \BE_{\ol{t}-1, \ell}\dots \BE_{\ol{1} \ell} \ol{\BE}^{(\ell+1)}v^+=0, \quad 1\le i<j \le M_+,\label{eq:R-3}\\
&&\BE_{\ol{r}\, \ol{s}} \BE_{\ol{t} \ell} \BE_{\ol{t}-1, \ell}\dots \BE_{\ol{1} \ell} \ol{\BE}^{(\ell+1)}v^+=0, \quad 1\le r<s \le M_-, s>t,
\label{eq:R-4}\\
&&\BE_{k \ol{s}} \ol{\BE}^{(\ell+1)}v^+=0, \quad k\le \ell. 
\label{eq:R-5}
\eeq
Note that \eqref{eq:R-3} and \eqref{eq:R-4} easily follow \eqref{eq:R-1} and \eqref{eq:R-2} respectively, and 
\eqref{eq:R-5} follows \eqref{eq:R-3} and \eqref{eq:R-4}. 
The relation \eqref{eq:R-2} is clear by \eqref{eq:BE-0}.  Let us prove \eqref{eq:R-1}. 
It is clear if $i<\ell$. If $i\ge \ell$, 
\beq
\BE_{i j} \ol{\BE}^{(\ell)}v^+&=\omega\big(\gamma_i-\gamma_j,  \sum_{k=\ell}^{i-1}d(\ol{\BE}_k) \big) \ol{\BE}_\ell\dots \ol{\BE}_{i-1} [\BE_{i j}, \ol{\BE}_i]  \ol{\BE}^{(i+1)}v^+. \nonumber
\eeq
Note that $[\BE_{i j}, \ol{\BE}_i]= \sum_r x_{\ol{r}} \BE_{\ol{r} j}$ for some $x_{\ol{r}}\in\U(\ol{\fv})$. As $\ol{\BE}_j \BE_{\ol{r} j}=0$, we conclude that the right hand side of the above equation vanishes, proving \eqref{eq:R-1}. 

Let us introduce further notation.  
For all $1\le i\le M_+$ and $1\le r\le M_-$, let 
\[
\baln
v^+_{i \ol{r}}&:=\BE_{\ol{r}-1, i}\dots \BE_{\ol{1} i} \ol{\BE}^{(i+1)}v^+, 
\ealn
\]
with $v^+_{i \ol{1}}= \ol{\BE}^{(i+1)}v^+$ and 
$v^+_{M_+,  \ol{r}}=\BE_{\ol{r}-1, M_+}\dots \BE_{\ol{1}, M_+} v^+$. 
In particular, $v^+_{M_+,  \ol{1}}=v^+$. 
Denote the weight of $v^+_{i \ol{r}}$ by $\lambda_{i \ol{r}}$, and define the scalar
$
\chi_{i \ol{r}}(\lambda) = (\lambda_{i \ol{r}}, \varepsilon_i- \varepsilon_{\ol{r}}).
$
Then $(\BE_{i i} + \BE_{\ol{r}\, \ol{r}})  v^+_{i \ol{r}} = \chi_{i \ol{r}}(\lambda) v^+_{i \ol{r}}$.

Now consider 
\[
\baln
\BE_{1 \ol{M_-}} \ol{\BE}v^+&= \BE_{1, \ol{M_-}} \BE_{\ol{M_-}, 1} \BE_{\ol{M_-}-1, 1}\dots \BE_{\ol{1} 1} \ol{\BE}^{(2)}v^+\\
&=(\BE_{1 1} + \BE_{\ol{M_-}\, \ol{M_-}}) v^+_{1 \ol{M_-}}- \BE_{\ol{M_-}, 1} \BE_{1, \ol{M-}} v^+_{1 \ol{M_-}}.
\ealn
\]
It is easy to show that  $\BE_{1, \ol{M_-}} v^+_{1 \ol{M_-}}=\BE_{1, \ol{M_-}} \BE_{\ol{M_-}-1, 1}\dots \BE_{\ol{1} 1} \ol{\BE}^{(2)}v^+=0$ by using  \eqref{eq:R-3}, \eqref{eq:R-4}  and \eqref{eq:R-5}.  Thus the second term vanishes, and we have 
\[
\baln
\BE_{1 \ol{M_-}} \ol{\BE}v^+&=\chi_{1 \ol{M_-}}(\lambda) v^+_{1 \ol{M_-}}.
\ealn
\]
The arguments can be repeated verbatim to show that for all $r$, 
\beq
\BE_{1 \ol{r}}v^+_{1, \ol{r}+1} = \chi_{1, \ol{r}}(\lambda) v^+_{1 \ol{r}}. 
\eeq
This then leads to
\[
\BE_1\ol{\BE}^{(1)}v^+= \left(\prod_{r=1} ^{M_-}\chi_{1, \ol{r}}(\lambda) \right)\ol{\BE}^{(2)}v^+.  
\]
Exactly the same calculations yield, for $1\le i \le M_+$, 
\beq\label{eq:key-aty}
\BE_i\ol{\BE}^{(i)}v^+= \left(\prod_{r=1} ^{M_-}\chi_{i, \ol{r}}(\lambda) \right)\ol{\BE}^{(i+1)}v^+.  
\eeq
This immediately leads to 
\[
\chi(\lambda)= \prod_{i=1}^{M_+} \prod_{r=1} ^{M_-}\chi_{i, \ol{r}}(\lambda).
\]

Let us now determine $\chi_{i, \ol{r}}$. 
We denote the weight of an element $x\in \U(\gl(V))$ by $wt(x)$. Then 
$wt(\ol{\BE}_i)=\sum_{r=1}^{M_-} \varepsilon_{\ol{r}} - M_- \varepsilon_i$, and hence 
$wt(\ol{\BE}^{(i+1)})=(M_+ -i)\sum_{r=1}^{M_-} \varepsilon_{\ol{r}} - M_- \sum_{j=i+1}^{M_+}\varepsilon_j$.
Let 
\[
\baln
\rho_{i\ol{r}}:=&wt(\BE_{\ol{r-1}, i}\dots \BE_{\ol{1} i} \ol{\BE}^{(i+1)}), 
\ealn
\]
which is clearly equal to $wt(\BE_{\ol{r-1}, i}\dots \BE_{\ol{1} i}) +wt(\ol{\BE}^{(i+1)})$. Thus we obtain
\[
\baln
\rho_{i\ol{r}}= &(M_+ -i)\sum_{s=1}^{M_-} \varepsilon_{\ol{s}} + \sum_{s=1}^{r-1} \varepsilon_{\ol{s}} 
- M_- \sum_{j=i+1}^{M_+}\varepsilon_j - (r-1)\varepsilon_i.
\ealn
\]
Now $(\rho_{i\ol{r}}, \varepsilon_i - \varepsilon_{\ol{r}}) =M_+-r  -i +1$. It follows \eqref{eq:rho-ir} that
\[
\baln
(\rho_{i\ol{r}}, \varepsilon_i - \varepsilon_{\ol{r}}) = (\rho, \varepsilon_i - \varepsilon_{\ol{r}}). 
\ealn
\]
The weight $\lambda_{i \ol{r}}$ of $v^+_{i \ol{r}}$ is given by $\lambda_{i \ol{r}}=\lambda+ \rho_{i\ol{r}}$. Hence  
$\chi_{i \ol{r}}(\lambda) = (\lambda+\rho, \varepsilon_i- \varepsilon_{\ol{r}})$,  and this leads to
\beq\label{eq:chi}
\chi(\lambda)=\prod_{i=1}^{M_+} \prod_{r=1}^{M_-}(\lambda+\rho, \varepsilon_i-\varepsilon_{\ol{r}})= \prod_{\Upsilon\in\Phi^+_1}(\lambda+\rho, \Upsilon). 
\eeq
Therefore, \eqref{eq:typical} is indeed equivalent to $\chi(\lambda)\ne 0$, proving the lemma. 
\end{proof}


\section{Invariant theory of general linear Lie $(\Gamma, \omega)$-algebras}

In this section, we study the classical invariant theory (in the sense of \cite{W}) 
for the general linear Lie $(\Gamma, \omega)$-algebras. 
Generalisations of Howe dualities \cite{Ho1, Ho}  in the context of  Weyl $(\Gamma, \omega)$-algebras will be developed, which include the Howe duality for 
general linear Lie superalgebras \cite{S, CLZ, CW}  as a special case 
(also see \cite{CZ} on Howe duality for orthosymplectic Lie superalgebras).  
We deduce from the generalised Howe dualities 
the first and second fundamental theorems of invariant theory 
for  $\gl(V(\Gamma, \omega))$ in the symmetric $(\Gamma, \omega)$-algebra setting.  
A Schur-Weyl duality between the general linear Lie $\gl(V(\Gamma, \omega))$-algebra and the symmetric group 
will also be obtained from generalised Howe dualities. 
This strengthens the double commutant theorem between them proved in \cite{FM}.

We mention that generalised Howe dualities for quantum general linear (super)groups 
were proved long ago \cite{LZ, WZ, Z03}, which have been much studied and generalised in recent years. 

We maintain notation of Section \ref{sect:modules}. In particular,  we write $V=V(\Gamma, \omega)$ and $\gl(V)=\gl(V(\Gamma, \omega))$.

\subsection{Symmetric $(\Gamma, \omega)$-algebras and Weyl $(\Gamma, \omega)$-algebras}
\label{sect:Weyl}

Given a positive integer $N$, denote $V^N=V\ot \C^N$, where $\C^N$ is regarded as a $\Gamma$-graded vector space consisting of degree $0$ elements only. For later use, we
denote by $\ol\C^N$ the dual space of $\C^N$ and by $V^*$ the $\Gamma$-graded dual space of $V$, and write $\ol{V}^N=V^*\ot \ol\C^N$.

Let $\{f^{(r)}\mid r=1, 2, \dots, N\}$ be the standard basis for $\C^N$. Then   
we have the homogeneous basis 
$\{e(\gamma)_i^r=e(\gamma)_i\ot f^{(r)}\mid 1\le r \le N, 1\le i\le m_\alpha, \alpha\in\Gamma\}$ for $V^N$. Let 
$T(V^N)=\sum_{n\ge 0} (V^N)^{\ot n}$ be the tensor algebra over $V^N$, which is an associative $(\Gamma, \omega)$-algebra.

The symmetric group $\Sym_n$ acts on $(V^N)^{\ot n}$  for any $n$ (with the action 
defined by Corollary \ref{cor:sym} but replacing $V$ by $V^N$). 
In particular for $n=2$, the operator $\Sigma^+(2)\in\C\Sym_2$ (cf. \eqref{eq:sym-skew}) 
acts on $V^N\ot V^N$ by $\id_{V^N}\ot \id_{V^N}- \tau_{V^N, V^N}$. Denote
$\bigwedge^2_\omega V^N = \Sigma^+(2)(V^N\ot V^N)$, 
which is spanned by the homogeneous elements  
\beq
e(\gamma)_i^r\ot  e(\zeta)_j^s - \omega(\gamma, \zeta)  e(\zeta)_j^s\ot e(\gamma)_i^r,  \quad \forall r, s,  i, j,  \gamma, \zeta. 
\eeq

Let $\mathcal{J}$ be the $2$-side ideal in $T(V^N)$ generated by $\bigwedge^2_\omega V^N$, 
and define the quotient algebra 
\[
S_\omega(V^N):=T(V^N)/\mathcal{J}.
\]
This is a  graded commutative associative $(\Gamma, \omega)$-algebra (cf. Section \ref{sect:assoc}). 
It is $\Z_+$-graded, with the homogeneous subspace 
of degree $n$ being $S^n_\omega(V^N):=\frac{(V^N)^{\ot n}}{(V^N)^{\ot n}\cap \mathcal{J}}$. 

We write $x(\gamma)_i^r:= e(\gamma)_i^r+ \mathcal{J}$ for all $r,  i,  \gamma$.

The algebra  $S_\omega(V^N)$ can be interpreted as the symmetric $(\Gamma, \omega)$-algebra over $V^N$
(see \cite[\S 12]{Sch83} for details). 
We  have $S_\omega(V^N)=\sum_{n\ge 0}S^n_\omega(V^N)$ with $S^n_\omega(V^N)=\Sigma^-(n)((V^N)^{\ot n})$. The multiplication is defined, 
for any $A\in S^p_\omega(V^N), B\in S^q_\omega(V^N)$,  by 
\[
A\ot  B \mapsto \frac{1}{(p+q)!}\Sigma^-(p+q)(A\ot B). 
\]
We can similarly define the skew symmetric $(\Gamma, \omega)$-algebra $\bigwedge_\omega V^N= \sum_{n\ge 0} \bigwedge^n_\omega V^N$ over $V^N$, where $\bigwedge^n_\omega V^N = 
\Sigma^+(n)((V^N)^{\ot n})$.

\begin{remark} Fix any $\gamma\in\Gamma_R$, and denote $V_\gamma^N=V_\gamma\ot\C^N$. The elements $x(\gamma)_i^r$, for $i=1, 2, \dots, m_\gamma$ and $r=1, 2, \dots, N$, generate the subalgebra 
$S_\omega(V_\gamma^N)\subset S_\omega(V^N)$. We have 
\[
S_\omega(V_\gamma^N)= \left\{
\begin{array}{l l}
S(V_\gamma^N), &\text{$\gamma\in\Gamma_R^+$}, \\
\bigwedge V_\gamma^N, &\text{$\gamma\in\Gamma_R^-$}, 
\end{array}
\right.
\]
where $S(V_\gamma^N)$ and $\bigwedge V_\gamma^N$ are respectively the usual symmetric algebra 
and exterior algebra of $V_\gamma^N$. 
\end{remark}

Hereafter we let ${\bf x}=\sum_\alpha\sum_{i, r} \C x(\alpha)_i^r$, 
and denote $S_\omega(V^N)$ by $\Px$.

We introduce another graded commutative associative $(\Gamma, \omega)$-algebra $\Pd$, 
which is generated by the elements  $\frac{\partial}{\partial x(\alpha)_ i^r}$ 
for $1\le r\le N, 1\le i\le m_\alpha$ and $\alpha\in\Gamma$, where $\frac{\partial}{\partial x(\alpha)_ i^r}$  
are of degree $-\alpha$.  
We can also grade $\Pd$ by $\Z_-$ with $\frac{\partial}{\partial x(\alpha)_ i^r}$  having degree $-1$.

Let $\CW_\omega$ be the associative algebra 
generated by the subalgebras $\Px$ and $\Pd$, with the following defining relations
\beq\label{eq:CCR}
\phantom{XX} 
\frac{\partial}{\partial x(\alpha)_ i^r} x(\beta)_j^s - \omega(-\alpha, \beta)  x(\beta)_j^s \frac{\partial}{\partial x(\alpha)_ i^r} =\delta_{\alpha \beta}\delta_{i j} \delta_{r s}, \quad \forall r, s, i, j, \alpha, \beta.
\eeq
We will  call $\CW_\omega$ the Weyl $(\Gamma, \omega)$-algebra.   
As vector space, $\CW_\omega\simeq\Px\ot\Pd$. 
 
The relations \eqref{eq:CCR} are clearly homogeneous of degree $0$ in the $\Gamma$-grading, 
thus $\CW_\omega$ is a 
$(\Gamma, \omega)$-algebra. It also admits a $\Z$-grading with $x(\alpha)_ i^r$ 
being of degree $1$, and $\frac{\partial}{\partial x(\alpha)_ i^r}$  of degree $-1$.

The results below are generalisations of standard facts for usual Weyl superalgebras to the present case. Their proofs are easy. 
\begin{lemma}\label{lem:Fock}
The subalgebra $\Px$ forms a simple $\CW_\omega$-module with $x(\alpha)_ i^r$ acting by multiplication, and  
 $\frac{\partial}{\partial x(\alpha)_ i^r}$ acting as $(\Gamma, \omega)$-derivations such that
\beq
\frac{\partial}{\partial x(\alpha)_ i^r}(1)=0, \quad \frac{\partial}{\partial x(\alpha)_ i^r}(x(\beta)_j^s)=\delta_{\alpha \beta}\delta_{i j} \delta_{r s}, \quad \forall r, s, i, j, \alpha, \beta. 
\eeq
\end{lemma}

\begin{remark}
 The requirement that $\frac{\partial}{\partial x(\alpha)_ i^r}$ be $(\Gamma, \omega)$-derivations on $\Px$ means  
\beq
\frac{\partial}{\partial x(\alpha)_ i^r}(A B) = \frac{\partial}{\partial x(\alpha)_ i^r}(A) B + \omega(\alpha, \zeta)A \frac{\partial}{\partial x(\alpha)_ i^r}(B)
\eeq
for any $A\in\Px_\zeta$, $\zeta\in\Gamma$,  and  $B\in \Px$ (cf. equation \eqref{eq:deriv}).
\end{remark}

As $\CW_\omega$-module,  $\Px$ is $\Gamma$-graded, and is also $\Z$-graded with vanishing  subspaces of negative degrees.  

\begin{definition}\label{rmk:F-space}
We call $\Px$ the Fock space for $\CW_\omega$, and write it as $\CF_{\CW_\omega}$
when it is necessary to make it explicit that $\Px$ is used as the Fock space
for $\CW_\omega$. Call $1\in\Px$  the vacuum vector of $\CF_{\CW_\omega}$. 
\end{definition}

The Weyl $(\Gamma, \omega)$-algebra $\CW_\omega$ has another simple module. 
\begin{lemma}\label{lem:Fock-dual}
The subalgebra $\Pd$ forms a simple $\CW_\omega$-module with $\frac{\partial}{\partial x(\alpha)_ i^r}$  acting by multiplication, and  $x(\alpha)_ i^r$
 acting as $(\Gamma, \omega)$-derivations such that
\beq
\phantom{XXX} x(\alpha)_ i^r(1)=0, \ x(\beta)_j^s\left(\frac{\partial}{\partial x(\alpha)_ i^r}\right)=- \omega(\alpha, \alpha) \delta_{\alpha \beta}\delta_{i j} \delta_{r s}, \ \forall r, s, i, j, \alpha, \beta. 
\eeq
\end{lemma}
We denote by $\ol\CF_{\CW_\omega}$ the $\CW_\omega$-module $\Pd$.

Now $\ol\CF_{\CW_\omega}$ can be identified with the $\Gamma\times \Z$-graded dual $\CW_\omega$-module of  the Fock space $\CF_{\CW_\omega}$. Let $p: \CF_{\CW_\omega}\lra \C$ be the projection onto the  
subspace spanned by the vacuum vector. 
\begin{lemma}
The following bilinear map 
 \beq
&&( \ , \ ): \ol\CF_{\CW_\omega} \times \CF_{\CW_\omega}\lra \C, \label{eq:F-pair}\\
&& (D, f)= p(D(f)), \quad \forall  f\in\CF_{\CW_\omega},  
D\in \ol\CF_{\CW_\omega}, \nonumber
\eeq
gives rise to a non-degenerate pairing between  $\ol\CF_{\CW_\omega}$ and 
 $\CF_{\CW_\omega}$, thus identifies  $\ol\CF_{\CW_\omega}$ with 
 the $\Gamma\times\Z$-graded dual space of $\CF_{\CW_\omega}$. 
\end{lemma}
Clearly $(\Pd_{-m}, \Px_n)=0$ if $m\ne n$, and $(\Pd_{-m}, \Px_m)$ is non-degenerate for all $m$. 
Note that for all $x(\alpha)_i^r$ and 
$\partial_{x(\beta)_j^s} :=\frac{\partial}{\partial x(\beta)_j^s}$,
\beq
 (D, x(\alpha)_i^r  f)) &=& - \omega(d(D), \alpha) (x(\alpha)_i^r (D), f), \label{eq:x-to-x}\\
 (D,  \partial_{x(\beta)_j^s} (f))&=&
 \omega(\beta, d(D))(\partial_{x(\beta)_j^s} D, f). \label{eq:d-to-d}
 \eeq

\subsection{Colour Howe duality  of type $(\gl(V(\Gamma, \omega)), \gl_N(\C))$}
Note that $\Px_1\simeq V^N$ and  $\Pd_{-1}\simeq \ol{V}^N$ as $\gl(V(\Gamma, \omega))\times \gl_N(\C)$-modules. 

\subsubsection{Colour Howe duality  of type $(\gl(V(\Gamma, \omega)), \gl_N(\C))$}\label{sect:HD}
Consider the following elements of $\Px_1\Pd_{-1}\subset\CW_\omega$:
\beq
&&E^{r s} := \sum_{\alpha\in\Gamma_R} \sum_{i=1}^{m_\alpha} x(\alpha)_ i^r \frac{\partial}{\partial x(\alpha)_ i^s}, \ r, s=1, 2, \dots, N;  \label{eq:glN}\\
&& \SE(\alpha, \beta)_{i j} =  \sum_{r=1}^N x(\alpha)_ i^r \frac{\partial}{\partial x(\beta)_ j^r}, \ 
1\le i\le m_\alpha,  1\le j\le m_\beta, \alpha, \beta \in \Gamma_R. \label{eq:glV}
\eeq
They are all homogeneous of degree $0$ in the $\Z$-grading. They are also homogeneous in the $\Gamma$-grading, 
with $E^{r s}$ being of degree $0$, and $\SE(\alpha, \beta)_{i j}$ of degree $\alpha-\beta$.  
Denote by $\CU_{\gl_N}$ the subalgebra of $\CW_\omega$ generated by the elements $E^{r s}$, 
and by $\CU_{\gl(V)}$ the subalgebra generated by the elements $ \SE(\alpha, \beta)_{i j}$. 

\begin{remark}\label{rmk:gl-gl}
To understand the rationale of the above construction, 
we observe that $\Px_1\Pd_{-1}\simeq V\ot\C^N\ot V^*\ot \ol\C^N$ as $\gl(V)\times \gl_N(\C)$-module. 
This isomorphism maps the elements $E^{r s}$ to 
$\Big(V\ot\C^N\ot V^*\ot \ol\C^N\Big)^{\gl(V)}$, and the elements $\SE(\alpha, \beta)_{i j} $ to $\Big(V\ot\C^N\ot V^*\ot \ol\C^N\Big)^{\gl_N(\C)}$.
\end{remark}

Denote by $[ \ , \ ]$ the $(\Gamma, \omega)$-commutator on $\CW_\omega$.  
We have the following result.

\begin{theorem}\label{thm:invariants}
Retain notation above.  The following statements are true. 
\begin{enumerate}
\item The elements in \eqref{eq:glN} satisfy the following relations. 
\beq\label{eq:gl-N}
[E^{r s}, E^{t u}]= \delta_{s t} E^{r u} - \delta_{ r u} E^{t s}, \quad \forall r, s, t, u. 
\eeq
Thus they span a usual Lie algebra isomorphic to the standard general linear Lie  algebra $\gl_N(\C)$, 
and $\CU_{\gl_N}$ is a quotient algebra of $\U(\gl_N(\C))$.  

\item
The elements in \eqref{eq:glV} satisfy the following relations. 
\[
\phantom{XX}
[\SE(\alpha, \beta)_{i j},  \SE(\gamma, \delta)_{k\ell} ]= \delta_{\beta \gamma} \delta_{j k} \SE(\alpha, \delta)_{i \ell} 
-\omega(\alpha-\beta, \gamma - \delta) \delta_{\alpha \delta} \delta_{\ell i} \SE(\gamma, \beta)_{k j} 
\]  
for $1\le i\le m_\alpha,  1\le j\le m_\beta$,  $1\le k\le m_\gamma,  1\le \ell\le m_\delta$,  and $\alpha, \beta, \gamma, \delta \in \Gamma_R$. 
Thus they span a Lie $(\Gamma, \omega)$-algebra isomorphic to $\gl(V)$, and $\CU_{\gl(V)}$ is a quotient algebra of $\U(\gl(V))$. 

\item The subalgebras $\CU_{\gl_N}$ and $\CU_{\gl(V)}$ of $\CW_\omega$ commute, i.e., 
\[
A Q - Q A=0, \quad \forall A\in \CU_{\gl_N},  Q\in \CU_{\gl(V)}. 
\]

\item Let $\CW_\omega^{\CU_{\gl_N}}=\{ Q\in \CW_\omega \mid [A, Q]=0,  \ \forall A\in\CU_{\gl_N} \}$.  Then $\CW_\omega^{\CU_{\gl_N}}$ is a subalgebra of $\CW_\omega$,  and $\CW_\omega^{\CU_{\gl_N}}= \CU_{\gl(V)}$.
\end{enumerate}
\end{theorem}
\begin{proof}
We can prove \eqref{eq:gl-N}  by direct calculations using \eqref{eq:CCR}.  
Note that the $\Gamma$-graded commutator reduces to the usual commutator for the 
elements $E^{r s}$  since they are of degree $0$ in the $\Gamma$-grading. Thus the elements $E^{r s}$
span a Lie algebra, which is isomorphic to the usual general linear Lie algebra $\gl_N(\C)$. 
The rest of part (1) easily follows. 

The proof of part (2) boils down to verifying the claimed commutation relations, which can be easily done by direct calculations.

The conceptual reason for part (3) and part (4) to work is Remark \ref{rmk:gl-gl}. 

To prove part (3), it suffices to show that the elements $\SE(\alpha, \beta)_{i j}$ commute with the elements $E^{r s}$.
This can be verified by a direct calculation. 

To prove part (4), note that $\CW_\omega$ is a module algebra for the Lie algebra $\gl_N(\C)$ of part (1), 
with the action given by the commutator. Furthermore, both $\Px$ and $\Pd$ are $\gl_N(\C)$-module subalgebras. 
Now
$
\CW_\omega =\sum_{m, n\ge 0}\Px_{m} \Pd_{-n},   
$
where $\Px_{m} \Pd_{-n}$ are finite dimensional $\gl_N(\C)$-submodules, thus are semi-simple. This shows that  
the $\gl_N(\C)$-action on $\CW_\omega$  is semi-simple. 

For any $Q, Q'\in \CW_\omega^{\CU_{\gl_N}}$, 
\[
[A, Q Q']= [A, Q]Q' + Q[A, Q']=0, \quad \forall A\in\CU_{\gl_N}. 
\]
Hence $Q Q'\in \CW_\omega^{\CU_{\gl_N}}$, proving that $\CW_\omega^{\CU_{\gl_N}}$ is a subalgebra of $\CW_\omega$.

We have $\Px_k = \Sigma^-(k)((V^N)^{\ot k})$ by making use of the $\Sym_k$-action on $(V^N)^{\ot k}$. 
Since $\gl_N(\C)$ naturally acts on the factor $\C^N$ in $V^N=V\ot \C^N$, there is a $\gl_N(\C)$-action on $(V^N)^{\ot k}$, which commutes with the $\Sym_k$-action. Thus $\Px_k$ is a $\gl_N(\C)$-module. 

Now $\gl_N(\C)$ acts on the factor $\ol\C^N$ of $\ol{V}^N$, and hence acts on  $(\ol{V}^N)^{\ot k}$. It commutes with the 
$\Sym_k$-action on $(\ol{V}^N)^{\ot k}$ (similarly defined as in Corollary \ref{cor:sym}). As $\gl_N(\C)$-module, 
\[
\Pd_{-k} = \Sigma^-(k)((\ol{V}^N)^{\ot k}). 
\]
Hence $\Px_{m}\Pd_{-n} \simeq (\Sigma^-(m)\ot \Sigma^-(n))((V^N)^{\ot m}\ot (\ol{V}^N)^{\ot n})$. 
Since the $\gl_N(\C)$-modules  $(V^N)^{\ot k}$ and $(\ol{V}^N)^{\ot k}$ are
semi-simple, 
\beq\label{eq:inv}
\phantom{XX} (\Px_{m}\Pd_{-n})^{\CU_{\gl_N}} \simeq (\Sigma^-(m)\ot \Sigma^-(n))\left( ((V^N)^{\ot m}\ot (\ol{V}^N)^{\ot n})^{\CU_{\gl_N}} \right), 
\eeq
which vanishes unless $m=n$. Thus $\CW_\omega^{\CU_{\gl_N}} =\sum_{n\ge 0}(\Px_{n} \Pd_{-n})^{\CU_{\gl_N}}$. 

Write $\CF_n\CW_\omega^{\CU_{\gl_N}}= \sum_{k=0}^n (\Px_{k} \Pd_{-k})^{\CU_{\gl_N}}$ for all $n\ge 0$. We have the following filtration for $\CW_\omega^{\CU_{\gl_N}}$:
\[
\dots \supset \CF_n\CW_\omega^{\CU_{\gl_N}}\supset \CF_{n-1}\CW_\omega^{\CU_{\gl_N}}\supset \dots\supset \CF_1\CW_\omega^{\CU_{\gl_N}}
\supset \CF_0\CW_\omega^{\CU_{\gl_N}} =\C. 
\]   
By applying the First Fundamental Theorem of invariant theory of $\gl_N(\C)$, 
we obtain 
$
 ((V^N)^{\ot m}\ot (\ol{V}^N)^{\ot m})^{\CU_{\gl_N}} \subset \C\Sym_{2m} \Big(\Big(\big(V^N\ot\ol{V}^N\big)^{\CU_{\gl_N}}\Big)^{\ot m}\Big)
$.
Translating this information to the left hand side of \eqref{eq:inv} and noting the obvious fact  that 
$(\Px_1 \Pd_{-1})^{\CU_{\gl_N}}\simeq \big(V^N\ot\ol{V}^N\big)^{\CU_{\gl_N}}$ is spanned by the elements $\SE(\alpha, \beta)_{i j}$,  we conclude that 
for any $n\ge 1$, monomials of order $n$ in the elements $\SE(\alpha, \beta)_{i j}$ span $\CF_n\CW_\omega^{\CU_{\gl_N}}$ modulo 
$\CF_{n-1}\CW_\omega^{\CU_{\gl_N}}$. In particular, 
$\CF_1\CW_\omega^{\CU_{\gl_N}}$ is spanned by the elements $\SE(\alpha, \beta)_{i j}$ together with the identity.  
Now a simple induction on $n$ 
proves part (4), completing the proof of the theorem. 
\end{proof}

Denote by $P$ the set of all partitions. Write a partition $\lambda$ as $\lambda=(\lambda_1, \lambda_2, \dots )$, 
where $\lambda_i\in \Z_+$ and $\lambda_i\ge \lambda_{i+1}$ for all $i$.  Let 
\beq
&&\text{$P_N=\{\lambda\in P\mid  \lambda_{N+1}=0\}$}, \\
&&\text{$P_{M_+|M_-}=\{\lambda\in P\mid \lambda_{M_++1}\le M_-\}$}. \label{eq:partits} 
\eeq
Note that $P_{M_+|M_-}$ is the set of $M_+|M_-$-hook partitions introduced in \cite{BR} for characterising simple 
tensor modules of the general linear Lie superalgebra $\gl_{M_+|M_-}(\C)$.

We think of a partition $\lambda=(\lambda_1, \lambda_2, \dots )$ as being represented by a Young diagram with $\lambda_i$ box in the $i$-th row for each $i$. 
Denote by $\lambda'=(\lambda'_1, \lambda'_2, \dots )$ the transpose partition. Then $\lambda'_j$ is the number of boxes in the $j$-th column of the Young diagram of $\lambda$. 
Call $\lambda'_1$ the depth of $\lambda$.  For example, the transpose partition  of $\lambda=(4, 3, 1)\vdash 8$ 
is $\lambda'=(3, 2, 2, 1)$. Their
Young diagrams are as shown below.  
\[
\lambda=\ \ytableausetup{mathmode, boxsize=2em}
\ytableaushort{
\ \ \ \ , \ \ \ ,  \ 
}\ ,  \qquad \qquad 
\lambda'=\ \ytableausetup{mathmode, boxsize=2em}
\ytableaushort{
\ \ \  , \ \  ,  \  \ , \
}\ . 
\]

We interpret a partition $\lambda\in P_N$ as a $\gl_N(\C)$ weight in the standard way. 
Denote by $L_\lambda(\gl_N)$ the simple $\gl_N(\C)$-module with highest weight $\lambda$.  

The following result is a generalised Howe duality between $\gl(V)$ and $\gl_N(\C)$. 
\begin{theorem}\label{thm:commut}
Retain notation above.
\begin{enumerate}
\item $\CU_{\gl(V)}$ and $\CU_{\gl_N}$ are mutual centralisers in $\End_\C(\Px)$, i.e., 
\[
\End_{\CU_{\gl_N}}(\Px) =  \CU_{\gl(V)}, \quad \End_{\CU_{\gl(V)}}(\Px) =  \CU_{\gl_N}. 
\]

\item As $\gl(V)\times\gl_N(\C)$-module, $\Px$ has the multiplicity free decomposition
\beq\label{eq:decomp}
\Px=\sum_{\lambda\in P_{M_+|M_-;N}} L_{\lambda^\sharp}\ot L_\lambda(\gl_N),
\eeq
where the set $P_{M_+|M_-;N}$, and $\lambda^\sharp$ as a function of 
$\lambda$, are explicitly given in Theorem \ref{thm:H-wts} (also see \eqref{eq:H-duality}).
\end{enumerate}
\end{theorem}
\begin{proof}
Let us start by proving the first relation in part (1). 
We have 
\[
\End_{\gl_N(\C)}(\Px)=\sum_{m, n} \Hom_{\gl_N(\C)}(\Px_m, \Px_n) 
= \sum_{m} \End_{\gl_N(\C)}(\Px_m). 
\] 
Since $\Px$ is a simple $\CW_\omega$-module,  the restriction of the action of 
$
\CW_\omega(0) =\sum_{k\ge 0}\Px_{k} \Pd_{-k}
$
to $\Px_m$  coincides with $\End_\C(\Px_m)$ for any $m$.  Note that 
\[
\CW_\omega(0)^{\CU_{\gl_N}} =\sum_{k\ge 0}\left(\Px_{k} \Pd_{-k} \right)^{\CU_{\gl_N}}= \CW_\omega^{\CU_{\gl_N}}.
\]
This shows that $\CW_\omega^{\CU_{\gl_N}}|_{\Px_m}=\End_{\gl_N(\C)}(\Px_m)$ for all $m$, and hence 
$\CW_\omega^{\CU_{\gl_N}}=\End_{\gl_N(\C)}(\Px)$.  Since $\CW_\omega^{\CU_{\gl_N}}=\CU_{\gl(V)}$ 
by Theorem  \ref{thm:invariants}.(4), we obtain the first relation of part (1).  

We also have $\End_{\gl(V)}(\Px) = \sum_{m} \End_{\gl(V)}(\Px_m)$. 
Since the $\gl_N(\C)$-action on $\Px_m$ is semi-simple for all $m$, 
 the double commutant  theorem \cite[Theorem 3.5.B]{W}  
 holds (by the usual arguments, 
 see, e.g.,  \cite[\S3.3.2]{GW} and the proof of \cite[Theorem 3.3.7]{GW} in particular). 
 Thus
\begin{itemize}
\item $\End_{\gl_N(\C)}(\Px_m)=\CU_{\gl(V)}|_{\Px_m}$ is semi-simple, and 

\item $\End_{\gl(V)}(\Px_m)= \CU_{\gl_N}|_{\Px_m}$. 
\end{itemize}
It follows the second bullet point that the first relation of part (1) implies the second one,  completing the proof of part (1) of the theorem. 

The first bullet item above shows that the $\gl(V)$-action on $\Px$ is also semi-simple. 
Then it immediately follows part (1) of the theorem  that $\Px$ has a multiplicity free decomposition 
into simple modules for $\gl(V)\times \gl_N(\C)$ as given by \eqref{eq:decomp}. 
We will determine the partitions $\lambda$ and the related $\lambda^\sharp$ appearing in the decomposition \eqref{eq:decomp} after proving Theorem \ref{thm:SW} below. 
\end{proof}

\begin{remark}
The proof of Theorem \ref{thm:commut} relies on the fact that 
$\Px_m$ are semi-simple $\gl_N(\C)$-modules for all $m$. 
This is the reason why we prove the colour Howe duality of type 
$(\gl(V(\Gamma, \omega)), \gl_N(\C))$ first.  
The  general case with $\gl_N(\C)$ replaced by 
$\gl(V'(\Gamma, \omega))$
will be treated in Theorem \ref{thm:HD-general}. 
\end{remark}

\begin{remark}\label{rmk:power-ss}
Note that Theorem \ref{thm:commut}.(2) implies that tensor powers of $V$ are semi-simple $\gl(V)$-modules. 
\end{remark}

\subsubsection{Dual version of the colour Howe duality}
%
%
%
It is clear from Theorem \ref{thm:invariants} that 
\begin{enumerate}[i).]
\item
the following elements span a Lie $(\Gamma, \omega)$-algebra isomorphic to $\gl(V)$,
\[
\wt\SE(\alpha, \beta)_{i j} = \sum_{r=1}^N \omega(\beta, \alpha)  \frac{\partial}{\partial x(\beta)_ j^r} x(\alpha)_ i^r, \quad \forall \alpha, \beta, i, j;  
\]

\item  the following elements span a Lie algebra isomorphic to $\gl_N(\C)$,
\[
\wt{E}^{r s} = \sum_{\alpha\in\Gamma_R} \sum_{i=1}^{m_\alpha} \omega(\alpha, \alpha) \frac{\partial}{\partial x(\alpha)_ i^s} x(\alpha)_ i^r, \quad \forall r, s.
\]
\end{enumerate}

Denote by
$\wt\CU_{\gl_N}$ the subalgebra of $\CW_\omega$ generated by the elements $\wt{E}^{r s}$, 
and by $\wt\CU_{\gl(V)}$ the subalgebra  generated by the elements $\wt\SE(\alpha, \beta)_{i j}$.  
Restrictions of the $\CW_\omega$ action on $\Pd$ to the subalgebras  $\wt\CU_{\gl_N}$  
and $\wt\CU_{\gl(V)}$ lead to  commuting
actions of $\gl_N(\C)$ and $\gl(V)$.

\begin{theorem}\label{thm:Howe-d}
As $\gl(V)\times\gl_N(\C)$-module, $\Pd$ has the following multiplicity free decomposition
\beq\label{eq:decomp-d}
\Pd=\sum_{\lambda\in P_{M_+|M_-;N}} L_{\lambda^\sharp}^*\ot L_\lambda(\gl_N)^*,
\eeq
where  
 $L_{\lambda^\sharp}^*$ and $L_\lambda(\gl_N)^*$ respectively denote the duals of the $\gl(V)$-module $L_{\lambda^\sharp}$, and $\gl_N(\C)$-module $L_\lambda(\gl_N)$.   
\end{theorem}
\begin{proof}
We have seen that $\Pd$ can be identified with the $\Gamma\times \Z$-graded dual space of $\Px$. The theorem immediately follows from Theorems \ref{thm:invariants} and \ref{thm:H-wts}, if we can show that $\Px$ and $\Pd$ are graded duals to each other as $\gl(V)\times\gl_N(\C)$-modules. 
To prove this, it suffices to show that for all $\wt\SE(\alpha, \beta)_{i j}$ and 
$\wt{E}^{r s}$, 
\beq
&&(\wt\SE(\alpha, \beta)_{i j}(D), f) = \omega(\alpha-\beta, d(D)) (D, -\SE(\alpha, \beta)_{i j}(f)), \label{eq:dual-1}\\
&&(\wt{E}^{r s}(D), f) = (D, - E^{r s}(f)), \quad \forall f\in\Px, D\in\Pd. \label{eq:dual-2}
\eeq
By using \eqref{eq:x-to-x} and \eqref{eq:d-to-d}, we obtain, for all $x(\alpha)_i^r$ and 
$\partial_{x(\beta)_j^s} :=\frac{\partial}{\partial x(\beta)_j^s}$,
\[
\baln
 (D, x(\alpha)_i^r\partial_{x(\beta)_j^s} (f))&=
 - \omega(d(D), \alpha-\beta)  \omega(\beta, \alpha)  (\partial_{x(\beta)_j^s} x(\alpha)_i^r (D), f).
 \ealn
\] 
Set $\alpha=\beta$ and $j=i$, then sum both sides over $\alpha$ and $i$, we obtain 
\eqref{eq:dual-2}. Similarly, set $s=r$ and then sum over $r$, we obtain 
\eqref{eq:dual-1}.   

This completes the proof of the theorem.
\end{proof}

\begin{remark}
It follows the above theorem that the actions of $\gl_N(\C)$ and $\gl(V)$ centralise each other over $\End_\C(\Pd)$. One can also prove this directly  
by reasoning similarly as in the proof of Theorem \ref{thm:commut}. 
\end{remark}

\begin{remark}\label{rmk:Pd}
Note in particular that  
$\Pd_{-1} = L_{(1, 0, \dots, 0)}^*\ot L_{(1, 0, \dots, 0)}(\gl_N)^*$, where $L_{(1, 0, \dots, 0)}^*=
V^*$, and $L_{(1, 0, \dots, 0)}(\gl_N)^*=\ol{\C}^N$. 
Hence $\Pd_{-1}=\ol{V}^N$, and it follows that
$\Pd \simeq S_\omega(\ol{V}^N)$ 
as $\gl(V)\times \gl_N(\C)$-module algebra. 
Thus \eqref{eq:decomp-d} is equivalent to 
\beq\label{eq:decomp-d-1}
S_\omega(\ol{V}^N)=\sum_{\lambda\in P_{M_+|M_-;N}} L_{\lambda^\sharp}^*\ot L_\lambda(\gl_N)^*.
\eeq
\end{remark}

\subsection{Colour Schur-Weyl duality}\label{sect:SW}

We deduce a colour Schur-Weyl duality between 
$\gl(V)$ and $\Sym_N$ over $V^{\ot N}$ from part (2) of Theorem \ref{thm:commut}. 

\begin{theorem} \label{thm:SW}
Retain notation of Corollary \ref{cor:sym}, and let $S_\lambda$ be the simple $\C\Sym_N$ - module (isomorphic to the Specht module over $\C$) associated with the partition $\lambda$. 
\begin{enumerate}
\item
There is the following multiplicity free decomposition of $V^{\ot N}$ with respect to the joint action of $\gl(V)$ and $\Sym_N$, 
\beq\label{eq:gl-S}
V^{\ot N}= \sum_{\lambda} L_{\lambda^\sharp}\ot S_\lambda, 
\eeq
where  the sum is over the set $\{\lambda \in P_{M_+|M_-}\mid \lambda\vdash N\}$, and 
$\lambda^\sharp=(\lambda^+, \lambda^-)$ is uniquely determined by $\lambda$ as follows. 
The Young diagram of $\lambda^+$ is the sub-diagram of $\lambda$ consisting of the first $M_+$ rows of $\lambda$, 
and  ${\lambda^-}'$ is equal to the skew diagram $\lambda/\lambda^+$. 

\item As a consequence of part (1), the following relations hold.
\beq
\phantom{XX} \End_{\Sym_N}(V^{\ot N})= \pi_N(\U(\gl(V)), \quad  \End_{\gl(V)}(V^{\ot N})=\nu_N(\C\Sym_N). 
\eeq
\end{enumerate}
\end{theorem}
\begin{remark}
Part (2) of the theorem was proved by Fischman and Montgomery
\cite[Theorem 2.15]{FM}, who used Scheunert's theorem \cite[Theorem 2]{Sch79} 
to reduce it to the double commutant property between the universal enveloping algebra of the 
general linear Lie superalgebra and the symmetric group \cite{BR}. 
\end{remark}
\begin{proof}[Proof of Theorem \ref{thm:SW}]
Note that part (2) immediately follows from part (1). 

Let us prove part (1).
Recall that the Weyl group of $\gl_N(\C)$ is $\Sym_N$, which will be taken as consisting of the matrices
permuting the elements of standard basis $\{f^{(r)}\mid r=1, 2, \dots, N\}$ of $\C^N$.  
Let $L_\lambda(\gl_N)_{(1^N)}\subset L_\lambda(\gl_N)$ be the ``zero weight'' space, 
that is, $\gl_N(\C)$-weight space 
of weight $(1, 1, \dots, 1)$. Then  
$L_\lambda(\gl_N)_{(1^N)}\ne 0$ if and only if $\lambda \vdash N$, and in this case 
$L_\lambda(\gl_N)_{(1^N)}\simeq S_\lambda$ as a simple module for 
$\Sym_N$ \cite{Ko} (also see \cite[\S2.4.5.3]{Ho}).

Denote by $\Px_{(1^N)}$  the  $\gl_N$-weight space of weight $(1, 1, \dots, 1)$ in $\Px$.  Then
\[
\baln
\Px_{(1^N)} &=\sum \C x(\gamma_1)_{i_1}^1 x(\gamma_2)_{i_2}^2\dots x(\gamma_N)_{i_N}^N
  & \simeq V^{\ot N} . 
  \ealn
\]
The $\Sym_N$-action on $\Px_{(1^N)}$ can be described as follows. Write $M(\gamma_a, i_a; \dots; \gamma_b, i_b)$ $= x(\gamma_a)_{i_a}^a x(\gamma_{a+1})_{i_{a+1}}^{a+1}\dots x(\gamma_b)_{i_b}^b$ for $a\le b$. 
Then for any $s_j$, 
\[
\baln
&s_j\cdot M(\gamma_1, i_1; \dots; \gamma_N, i_N)\\ 
&\phantom{X} =M(\gamma_1, i_1; \dots; \gamma_{j-1}, i_{j-1})
 x(\gamma_j)_{i_j}^{j+1}  x(\gamma_{j+1})_{i_{j+1}}^{j}M(\gamma_{j+2}, i_{j+2}; \dots; \gamma_N, i_N)\\
 &\phantom{X} = \omega(\gamma_j, \gamma_{j+1}) M(\gamma_1, i_1; \dots; \gamma_{j-1}, i_{j-1}; \gamma_{j+1}, i_{j+1}; \gamma_{j}, i_{j}; \gamma_{j+2}, i_{j+2}; \dots; \gamma_N, i_N).
 \ealn
\]
Under the isomorphism $\Px{(1^N)}\simeq V^{\ot N}$, this leads to the $\Sym_N$-action 
given by Corollary \ref{cor:sym}. 

Taking the $\gl_N(V)$-zero-weight spaces of both sides of \eqref{eq:decomp}, 
we obtain
\beq\label{eq:decomp-pf}
V^{\ot N}\simeq \sum_{\lambda} L_{\lambda^\sharp}\ot L_\lambda(\gl_N)_{(1^N)} 
= \sum_{\lambda} L_{\lambda^\sharp}\ot S_\lambda, 
\eeq
where $\lambda$ and $\lambda^\sharp$ uniquely determine each other 
by Theorem \ref{thm:commut}.
Now we determine the partitions $\lambda$ and associated $\lambda^\sharp$ which appear in the decomposition. 

We adopt  notation of Section \ref{sect:refined}. 
For any sequence ${\bf i}=(i_1, i_2, \dots, i_N)$ with $i_r\in\{1, 2, \dots, \dim V\}$ for all $r$, let ${\bf b}_{\bf i}:=b_{i_1}\ot  b_{i_2}\ot \dots \ot   b_{i_N}$, and call it a ``monomial'' for easy reference. We will refer to the subscripts $1, 2$ and etc.  as the position labels of $b_{i_1}$, $b_{i_2}$ and etc. in ${\bf b}_{\bf i}$. The elements ${\bf b}_{\bf i}$, for all  ${\bf i}$, form a basis of $V^{\ot N}$. 
Let $\alpha_j=\sum_{r=1}^{\dim V}\delta_{i_r, j}$ fo $j=1, 2, \dots, \dim V$, then  
there are $\alpha_i$ factors of $b_i$ in ${\bf b}_{\bf i}$. We say that the monomial ${\bf b}_{\bf i}$ is of type $\alpha=(\alpha_1, \dots, \alpha_{\dim V})$.
Denote ${\bf b}(\alpha)= b_1^{\ot \alpha_1}\ot b_2^{\ot \alpha_2}\ot \dots \ot  b_{\dim V}^{\ot \alpha_{\dim V}}$. Then the $\C\Sym_N$-submodule $M_\alpha:=\C\Sym_N({\bf b}(\alpha))$ of $V^{\ot N}$ 
is spanned by the monomials of type $\alpha$.

Let $\beta_j=\sum_{i=1}^{j} \alpha_i$ with $\beta_0=0$, where $\beta_{\dim V}=N$. For each $i=1, 2, \dots, \dim V$, we introduce the subgroup 
$H_i:= \Sym\{\beta_{i-1}+1, \beta_{i-1}+2, \dots, \beta_i\}$
of $\Sym_N$,  which is the symmetric group on the letters $\beta_{i-1}+1, \beta_{i-1}+2, \dots, \beta_i$. If $\beta_{i-1}=\beta_i$, then $H_i$ is the trivial group. 
Define  
$
H_{\alpha}=H_1\times H_2 \times \dots \times H_{\dim V}.  
$
For any $r$, we denote by $S_{(r)}$ and $S_{(1^r)}$ the $1$-dimensional $\C\Sym_r$-modules corresponding to the identity and sign representations respectively. Let $S_\alpha:=S_{(\alpha_1)}\ot \dots\ot  S_{(\alpha_{M_+})}\ot S_{\left(1^{\alpha_{\bar 1}}\right)}\ot \dots\ot S_{\left(1^{\bar{M}_-}\right)}$, 
where $\bar j = M_++j$. One can easily see that $M_\alpha$ is isomorphic to the 
induced module ${\rm Ind}_{H_\alpha}^{\Sym_N}(S_\alpha)$. A generalisation of Young's rule \cite[Lemma 3.23]{BR} states that  
\beq\label{eq:Young}
{\rm Ind}_{H_\alpha}^{\Sym_N}(S_\alpha) = \sum_{\substack{\lambda\in P_{M_+|M_-}\\ \lambda\vdash N}} k_\alpha(\lambda) S_\lambda, 
\eeq
where $S_\lambda$ is the simple $\C\Sym_N$-module associated with the partition $\lambda$, and $k_\alpha(\lambda)$ is the multiplicity of $S_\lambda$, which is some generalised Kostka number equal to the number of  semi-standard $(M_+, M_-)$-tableaux of shape $\lambda$ (in the sense of \cite[Definition 2.1]{BR}) and type $\alpha$ \cite[Definition 3.22]{BR}.  
As $V^{\ot N}=\sum_\alpha M_\alpha$, we obtain the $\C\Sym_N$-module isomorphism
\beq\label{eq:Sym-decomp}
V^{\ot N}\simeq \sum_{\substack{\lambda\in P_{M_+|M_-}\\ \lambda\vdash N}} k(\lambda) S_\lambda, 
\eeq
where the multiplicity  $k(\lambda)=\sum_\alpha k_\alpha(\lambda)$ of $S_\lambda$  is equal to the number of semi-standard $(M_+, M_-)$-tableaux of shape $\lambda$ by \cite[Lemma 3.23]{BR}.

We now determine $\lambda^\sharp$ for any given partition  $\lambda=(\lambda_1, \lambda_2, \dots )\vdash N$ in $P_{M_+|M_-}$. Note that the only condition which the depth $\lambda'_1$ of $\lambda$ needs to satisfy is that $\lambda'_1\le N$ as required by $\lambda\vdash N$.  
Let $\lambda^+=(\lambda_1, \lambda_2, \dots, \lambda_{M_+})$, and consider the skew diagram $\mu:=\lambda/\lambda^+=(\lambda_{\bar 1}, \lambda_{\bar 2}, \dots)$, where $\bar r=M_++r$. It is empty if $\lambda'_1\le M_+$; otherwise it is a partition, whose transpose $\lambda^-:=\mu'$ has depth $\le M_-$. Write $\lambda^-=(\lambda^-_1, \lambda^-_2, \dots, \lambda^-_{M_-})$, and define 
\beq\label{eq:lam-shp}
\lambda^\sharp=(\lambda^+, \lambda^-)\in\Z_+^{\dim V}. 
\eeq

We construct a monomial in $V^{\ot N}$ of type $\lambda^\sharp$ as follows. 
Let
\[
\baln
&\wt{\bf v}_{\lambda^\sharp}=\wt{\bf v}_1\ot \wt{\bf v}_2\ot \dots\ot \wt{\bf v}_{\lambda'_1}, \quad \lambda'_1=\text{depth of $\lambda$}, \\
&\wt{\bf v}_i=  b_i^{\ot \lambda_i}, \quad i\le M_+,  \\
&\wt{\bf v}_{\bar j}=b_{\bar1}\ot b_{\bar2}\ot\dots 
\ot b_{\bar{\lambda}_{\bar j}}, \quad {\bar j}=M_++ j, \ 1\le j\le M_-.
\ealn
\]
It is clear that the $\gl(V)$-weight of $\wt{\bf v}_{\lambda^\sharp}$ is equal to $\lambda^\sharp$. 

\medskip

Denote by $T_\lambda$ the canonical tableau of shape $\lambda$, that is,  the standard tableau 
with the first row of boxes filled with the numbers $1, 2, \dots, \lambda_1$ from left to right, the second row of boxes with 
$\lambda_1+1, \lambda_1+2, \dots, \lambda_1+\lambda_2$, and etc..  
Let $P_\lambda$ (resp.  $Q_\lambda$) be the row (resp. column) 
subgroup of $T_\lambda$ in $\Sym_N$, and  
denote by $C_\lambda=B_\lambda A_\lambda$ the Young symmetriser associated with $T_\lambda$, where 
\[
A_\lambda= \sum_{\sigma\in P_\lambda} \sigma, \quad 
B_\lambda= \sum_{\sigma\in P_\lambda} (-1)^{|\sigma|}\sigma.
\]
[We have adopted the definition of the Young symmetriser in \cite[\S 2, Chapter IV] {W}.]
For example, the  canonical tableau of shape $\lambda=(4, 3, 1)\vdash 8$ is
\[
\ytableausetup{mathmode, boxsize=2em}
\ytableaushort{
1234, 567, 8
}.
\]
The row subgroup is $P_\lambda=\Sym\{1,2,3, 4\}\times \Sym\{5, 6, 7\}$, and the column subgroup is $Q_\lambda=\Sym\{1,5,8\}\times \Sym\{2, 6\} \times \Sym\{3, 7\}$, where we have omitted the trivial groups $\Sym\{8\}$ and $\Sym\{4\}$ in $P_\lambda$ and $Q_\lambda$ respectively.

Return to a general partition $\lambda\in P_{M_+|M_-}$ of $N$. 
Let ${\bf v}_{\lambda^\sharp}=C_{\lambda}(\wt{\bf v}_{\lambda^\sharp})$. 
We can easily see that ${\bf v}_{\lambda^\sharp}\ne 0$ by inspecting the construction. 
It is not difficult to show that ${\bf v}_{\lambda^\sharp}$ is a $\gl(V)$-highest weight vector by using Lemma \ref{lem:a-symm}.  
Consider, for example, a positive root vector $\BE_{i t}$ of $\gl(V)$ which maps a $b_t$ with $t>M_+$ to a $b_i$ with $i\le M_+$. Because $\lambda\in P_{M_+|M_-}$, the position label of $b_t$ must appear in a column of $T_\lambda$ of length greater than $M_+$. This means that in $\BE_{i t}\cdot {\bf v}_{\lambda^\sharp}=C_\lambda (\BE_{i t}\cdot \wt{\bf v}_{\lambda^\sharp})$, the column skew symmetriser $B_\lambda$ skew symmetrizes,  in a $\Gamma$-graded fashion, a tensor product of elements of $B$ with at least two $b_i$'s, thus annihilates it by \eqref{eq:a-symm}. Hence $\BE_{i t}\cdot {\bf v}_{\lambda^\sharp}=0$. 

The $\gl(V)$-weight of ${\bf v}_{\lambda^\sharp}$ is the same as that of 
$\wt{\bf v}_{\lambda^\sharp}$, which is equal to $\lambda^\sharp$. Thus 
$\C\Sym_N({\bf v}_{\lambda^\sharp})$ is the subspace of $\gl(V)$-highest weight vectors of $L_{\lambda^\sharp}\ot S_\lambda\subset V^{\ot N}$. 

This completes the proof of the theorem. 
\end{proof}

Theorem \ref{thm:SW} and equation \eqref{eq:Sym-decomp} together imply the following result. 
\begin{corollary}

For any $\lambda\in P_{M_+|M_-}$, let $\lambda^\sharp$ be defined by \eqref{eq:lam-shp}. 
Then the dimension of the simple $\gl(V)$-module $L_{\lambda^\sharp}$ is given by 
\[\dim L_{\lambda^\sharp}= k(\lambda), \]
where $k(\lambda)$ is the number of  semi-standard $(M_+, M_-)$-tableaux of shape $\lambda$.
\end{corollary}

Note that $\dim L_{\lambda^\sharp}$ is the same as the dimension of the simple 
$\gl_{M_+|M_-}(\C)$-module with  highest weight $\lambda^\sharp$. 

\subsubsection{Simple tensor modules for $\gl(V(\Gamma, \omega))$}\label{sect:H-wts}

Now we give an explicit characterisation of  the simple $\gl(V)\times\gl_N(\C)$ submodules 
of $S_\omega(V^N)$.

\begin{theorem}\label{thm:H-wts}
Let $P_{M_+|M_-;N}=P_{M_+|M_-}\cap P_N$. 
A simple $\gl(V)\times \gl_N(\C)$-submodule $L_{\lambda^\sharp}\ot L_\lambda(\gl_N)$ appears in the decomposition 
\eqref{eq:decomp} of $S_\omega(V^N)$ if and only if $\lambda\in P_{M_+|M_-;N}$, where  
$\lambda^\sharp=(\lambda_+, \lambda_-)$ 
is  defined as follows.  
Write $\lambda=(\lambda_1, \lambda_2, \dots)$, then  
\[
\baln
&\lambda_+=(\lambda_1, \lambda_2, \dots, \lambda_{M_+}), 
&\lambda_-=(\theta(\lambda'_1-M_+), \theta(\lambda'_2-M_+), \dots, \theta(\lambda'_{M_-}-M_+)), 
\ealn
\]
where $\lambda'_j$ is the length of the $j$-th column 
of the Young diagram of $\lambda$ for all $j$, 
and 
$\theta$ is the function on $\R$ defined by $\theta(x)=0$ if $x\le 0$ and $\theta(x)= x$ if $x>0$. 
\end{theorem}
Note that if $\lambda\vdash N$, the $\lambda^\sharp$ so defined coincides 
with that in Theorem \ref{thm:SW}.(1). 

To prove the theorem, we will need some functors between categories 
of tensor representations of general linear Lie algebras. 
Given any $N$, we consider the complete flag 
$\C^N\supset \C^{N-1}\supset \C^{N-2}\supset \dots \supset \C$, 
where $\C^k$ is the subspace of $\C^N$ spanned by the first $k$ standard basis elements of $\C^N$. 
Corresponding to the flag, there is the chain of subalgebras 
$\gl_N(\C)\supset \gl_{N-1}(\C)\supset \gl_{N-2}(\C)\supset \dots \supset \gl_1(\C)$. 
Denote by $\SP_k$ the category of tensor (also called polynomial) representations of $\gl_k(\C)$ for any $k$, i.e., direct sums of submodules of $(\C^k)^{\ot r}$ for all $r$. 
The set of weights of any object in $\SP_k$ is contained in $\Z_+^k$.  

Given any module $M_N\in \SP_N$, let $M_{N-1}$ be the subspace spanned by the weight vectors with weights $(\mu_1,\dots, \mu_{N-1}, 0)$, where $\mu=(\mu_1,\dots, \mu_{N-1})\in \Z_+^{N-1}$. Then $M_{N-1}$ is a $\gl_{N-1}(\C)$-submodule, 
and the restriction of $M_N$ to a $\gl_{N-1}(\C)$-module decomposes into a direct sum 
$M_N=M_{N-1}\oplus M_{N-1}^\bot$ of submodules.  
Define the truncation functor ${\ST}_{N, N-1}: \SP_N\lra \SP_{N-1}$, which maps any object $M_N$ to $M_{N-1}$, 
and any  morphism $\varphi: M_N\lra M'_N$ to the restriction  $\varphi|_{M_{N-1}}: M_{N-1} \lra M'_{N-1}$. 
This is an exact functor, which is essentially surjective (i.e., all objects of $\SP_{N-1}$ are images of objects of $\SP_N$).  
Compositing successive functors, we obtain truncation functors
\[
{\ST}_{N, k}= {\ST}_{k+1, k} \circ {\ST}_{k+2, k+1}\circ\dots\circ  {\ST}_{N, N-1}: \SP_N\lra \SP_k. 
\] 
Note in particular that for $\lambda\in P_N$, 
\[
{\ST}_{N, k}:  L_\lambda(\gl_N)\mapsto \left\{
\begin{array}{l l}
L_\lambda(\gl_k), & \text{if depth of $\lambda\le k$},\\
0, &\text{otherwise}.
\end{array}
\right.
\]

\begin{proof}[Proof of Theorem \ref{thm:H-wts}]
We find all $\lambda$ appearing in \eqref{eq:decomp}, and determine $\lambda^\sharp$ 
in the equation as a function of $\lambda$.  
To avoid confusing the $\lambda^\sharp$ in \eqref{eq:decomp} with that defined in the present theorem, 
we temporarily denote it by $\wt\lambda$.  

For any $k\ge 0$, we take the degree $N+k$ subspace $S^{N+k}_\omega(V^{N+k})$
of $S_\omega(V^{N+k})$. We obtain from  \eqref{eq:decomp} the following relation.
\beq\label{eq:decomp+}
S^{N+k}_\omega(V^{N+k})=\sum_{\text{some } \lambda\vdash N+k} L_{\wt\lambda}\ot L_\lambda(\gl_{N+k}).
\eeq
Since the zero weight space of $L_\lambda(\gl_{N+k})$ is isomorphic to 
the simple $\Sym_{N+k}$-module $S^\lambda$ because $\lambda\vdash N+k$,  it follows Theorem \ref{thm:SW} that $\lambda \in P_{M_+|M_-}$ and $\wt\lambda=\lambda^\sharp$. 
Furthermore, for any $\lambda \in P_{M_+|M_-}$ satisfying $\lambda\vdash N+k$, the submodule 
$L_{\lambda^\sharp}\ot L_\lambda(\gl_{N+k})$ appears on the right hand side of the above equation. 

Applying the truncation functor ${\ST}_{N+k, N}$  to \eqref{eq:decomp+}, we arrive at 
\beq\label{eq:decomp+1}
S^{N+k}_\omega(V^{N})=\sum_{\substack{\lambda\in P_{M_+|M_-;N}, \\  \lambda\vdash N+k}} L_{\lambda^\sharp}\ot L_\lambda(\gl_{N}), \quad \forall k\ge 0.
\eeq

Now we consider the degree $j$ subspace $S^j_\omega(V^N)$
of $S_\omega(V^N)$ for $j<N$. Equation \eqref{eq:decomp} leads to   
\beq\label{eq:decomp-}
S^j_\omega(V^N)=\sum_{\text{some } \lambda\vdash j} L_{\wt\lambda}\ot L_\lambda(\gl_N).
\eeq
Applying the truncation functor $\ST_{N, j}$ to \eqref{eq:decomp-}, we obtain
$
S^j_\omega(V^j)=\sum L_{\wt\lambda}\ot L_\lambda(\gl_j).
$
Note that any $L_\lambda(\gl_N)$ on the right hand side of \eqref{eq:decomp-} truncates to $L_\lambda(\gl_j)$ since $\lambda\vdash j$.
Thus it follows Theorem \ref{thm:SW} that
\beq\label{eq:decomp-1}
S^j_\omega(V^N)=\sum_{\substack{\lambda\in P_{M_+|M_-;N}, \\  \lambda\vdash j}} L_{\lambda^\sharp}\ot L_\lambda(\gl_N), \quad   j<N. 
\eeq
Equations \eqref{eq:decomp+1} and \eqref{eq:decomp-1} together imply that
\beq\label{eq:H-duality}
S_\omega(V^{N})=\sum_{\lambda\in P_{M_+|M_-;N}} L_{\lambda^\sharp}\ot L_\lambda(\gl_{N}),
\eeq
where $\lambda^\sharp$ is as defined in the statement of the theorem. 
\end{proof}

\begin{remark}\label{rmk:non-super}
If $\Gamma_R^-=\emptyset$, then $V_-=0$, and $\gl(V)=\gl(V_+)$ 
is a Lie colour algebra. 
A simple $\gl(V)$-module $L_\mu$  is finite dimensional 
if an only if its highest weight $\mu$ satisfies condition \eqref{eq:domin}. In this case, 
we can always find some  
$c\in \C$ and $\lambda\in P_{M_+}$ such that $\mu=(c, c, \dots, c)+ \lambda$. 
This in particular implies that
$L_\mu \simeq L_{(c, c, \dots, c)}\ot L_\lambda$, where $\dim L_{(c, c, \dots, c)}=1$.  
By Theorem \ref{thm:SW}, $L_\lambda\subset V^{\ot|\lambda|}$ is a simple tensor module.  
Since tensor powers of $V$ are semi-simple by Remark \ref{rmk:power-ss}, 
the tensor product of any number of 
finite dimensional simple modules is semi-simple.  
\end{remark}

\subsection{Fundamental theorems of invariant theory}\label{sect:FT}

We return to Remark \ref{rmk:Pd}. 
Let $\{\ol{f}^{(r)}\mid r=1, 2, \dots, N\}$ be the basis of $\ol{\C}^N$ dual to the standard basis for $\C^N$, thus $\ol{f}^{(r)}(f^{(s)})=\delta_{r s}$ for all $r, s$. 
Also let 
$\{\ol{e}(-\gamma)_i\mid 1\le i\le m_\gamma, \gamma\in\Gamma_R\}$ be the basis of $V^*$ dual to 
the basis $\{e(\gamma)_i\mid 1\le i\le m_\gamma, \gamma\in\Gamma_R\}$ of $V$, hence 
$\ol{e}(-\beta)_j(e(\alpha)_i)=\delta_{\alpha\beta}\delta_{ij}$ for all $i, j, \alpha, \beta$. 
Now
$\{\ol{e}(-\gamma)_i^r=\ol{e}(-\gamma)_i\ot \ol{f}^{(r)}\mid 1\le r \le N, 1\le i\le m_\gamma, \gamma \in\Gamma_R\}$ is a basis for $\ol{V}^N$. 

We denote the images of the elements $\ol{e}(-\gamma)_i^r$ in $S_\omega(\ol{V}^N)$
by $\ol{x}(-\gamma)_i^r$.  
Then the algebra isomorphism $S_\omega(\ol{V}^N)\simeq \Pd$ is given by 
$\ol{x}(-\gamma)_i^r\mapsto \frac{\partial}{\partial x(\beta)_j^s}$ for all $i, r, \gamma$. 

Consider $S_\omega(V^N\oplus \ol{V}^N)=S_\omega(V^N) \ot S_\omega(\ol{V}^N)$ 
as a module algebra for the Lie $(\Gamma, \omega)$-algebra
 $
 \fg=(\gl(V)\times \gl_N(\C))\times (\gl(V)\times \gl_N(\C)).
 $
 Then $\fg\supset \gl(V)\times \gl(V) \supset$ $\text{diagonal subalgebra $\gl(V)$}$. 
 Denote the subspace of invariants with respect the  action of the diagonal subalgebra $\gl(V)$ by 
\[
S_\omega(V^N\oplus\ol{V}^N)^{\gl(V)}
:=\left\{\left. z\in S_\omega(V^N\oplus\ol{V}^N)\right| X\cdot z= 0, \ \forall X\in \gl(V)\right\}.
\]
One can easily see that the following elements belong to  $S_\omega(V^N\oplus\ol{V}^N)^{\gl(V)}$. 
\beq\label{eq:zrs}
z_{r s} = \sum_{\gamma\in\Gamma_R}\sum_{i=1}^{m_\gamma} x(\gamma)_i^r \ol{x}(-\gamma)_i^s, \quad \forall r, s. 
\eeq
They have $\Gamma$-degree $0$, thus commute with all elements 
of $S_\omega(V^N\oplus\ol{V}^N)$. 

\begin{theorem}\label{thm:FFT}
The subalgebra $S_\omega(V^N\oplus\ol{V}^N)^{\gl(V)}$ is generated by the elements $z_{r s}$ for $r, s=1, 2, \dots, N$, and  
has the following multiplicity free decomposition as a $\gl_N(\C)\times \gl_N(\C)$-module algebra. 
\beq\label{eq:SFT}
S_\omega(V^N\oplus\ol{V}^N)^{\gl(V)}
=\sum_{\lambda\in P_{M_+|M_-;N}} L_\lambda(\gl_N)\ot L_\lambda(\gl_N)^*. 
\eeq
\end{theorem}
\begin{proof}
It follows Theorems \ref{thm:invariants},  \ref{thm:H-wts} and \ref{thm:Howe-d} that
\[
\baln
S_\omega(V^N\oplus\ol{V}^N)&=S_\omega(V^N) \ot S_\omega(\ol{V}^N)\\
&\simeq \sum_{\lambda, \mu \in P_{M_+|M_-;N}}
L_{\lambda^\sharp}\ot L_{\mu^\sharp}^*\ot L_\lambda(\gl_N)\ot L_\mu(\gl_N)^*
\ealn
\]
as $\gl(V)\times\gl(V)\times\gl_N(\C)\times\gl_N(\C)$-module. Under the diagonal $\gl(V)$-action, 
\[
\baln
S_\omega(V^N\oplus\ol{V}^N)^{\gl(V)}
&\simeq \sum_{\lambda, \mu \in P_{M_+|M_-;N}}
(L_{\lambda^\sharp}\ot L_{\mu^\sharp}^*)^{\gl(V)}\ot L_\lambda(\gl_N)\ot L_\mu(\gl_N)^*\\
&\simeq \sum_{\lambda \in P_{M_+|M_-;N}}
 L_\lambda(\gl_N)\ot L_\lambda(\gl_N)^*. 
\ealn
\]
The equalities are isomorphisms of $\gl_N(\C)\times\gl_N(\C)$-modules.
This proves \eqref{eq:SFT}. 

Now we prove the first statement of the theorem. 

Assume that $M_+\ge N$. 
In this case, we have  $P_{M_+|M_-;N}=P_N$, and the last line of the above equation 
becomes $\sum_{\lambda \in P_N}
 L_\lambda(\gl_N)\ot L_\lambda(\gl_N)^*$, which is known to be isomorphic to 
$S(\C^N\ot\ol{\C}^N)$
as $\gl_N(\C)\times\gl_N(\C)$-module algebra. 
This shows that 
\beq
 S_\omega(V^N\oplus\ol{V}^N)^{\gl(V)}\simeq S(\C^N\ot\ol{\C}^N).
\eeq

Denote by $Z_{r s}$ the image of $f^{(r)}\ot \ol{f}^{(s)}$ in $S(\C^N\ot\ol{\C}^N)$ for any $r, s$. 
Then $S(\C^N\ot\ol{\C}^N)$ is the polynomial algebra in the variables $Z_{r s}$, 
and the above algebra isomorphism is given by 
\[
Z_{r s}\mapsto z_{r s}, \quad \forall r, s.
\] 
This proves the first statement of the theorem in the case $M_+\ge N$. 

Now we prove the case $M_+<N$. 

Let $B'$ be a subset of the ordered basis $B$ for $V$ such that $B'\supset B^-$. Denote $V'=\C B'$. Then 
$V'=\sum_{\alpha} V'_\alpha$. Let 
\[
\Gamma_R'=\{\alpha\in \Gamma_R\mid \dim{V'_\alpha}>0\}, \quad M'_+=\dim{V'\cap V_+}.
\]

We have the subalgebra $\gl(V')\subset \gl(V)$. We want to show that the 
first statement of the theorem is true for 
$S_\omega({V'}^N\oplus \ol{V'}^N)^{\gl(V')}$, where $M'_+<N$ is an allowed possibility. 

We introduce truncation ``functors'' sending $\gl(V)$-modules under consideration 
to $\gl(V')$-modules. This time we define the functors concretely as follows (cf. Section \ref{sect:H-wts}). 
Let $I'$ be the  projection map $I'(V)=V'$, which is the identity element of $\gl(V')\subset \gl(V)$. 
For each $r$, we have the map
\[
\baln
\I^\prime_r=(I'\ot\id_{\C^N})^{\ot r}:  (V^N)^{\ot r}\lra (V^{\prime N})^{\ot r}, 
\ealn
\]
which descends to a map on $S_\omega(V^N)_r$, 
where $S_\omega(V^N)_r$ is regarded as a subspace of $(V^N)^{\ot r}$. 
Define the map 
\beq\label{eq:I-funct}
\I^\prime: S_\omega(V^N)\lra S_\omega({V^\prime}^N), \quad  \text{$\I^\prime|_{S_\omega(V^N)_r}=\I^\prime_r$, $\forall r$.}
\eeq
Let $\ol{V'}=\sum_{\alpha>\alpha_c} V^*_{-\alpha}$, which is the dual space of $V'$. 
Also introduce the maps 
\beq
&&\ol{\I}^\prime_r=(-I'\ot\id_{\ol{\C}^N})^{\ot r}:  (\ol{V}^N)^{\ot r}\lra (\ol{V^\prime}^N)^{\ot r}, \\
&&\ol{\I}^\prime: S_\omega(\ol{V}^N)\lra S_\omega(\ol{V^\prime}^N), \quad \text{$\ol{\I}^\prime|_{S_\omega(\ol{V}^N)_r}=\ol{\I}^\prime_r$, $\forall r$.} \label{eq:Ibar-funct}
\eeq
It is important to note that 
$\I^\prime$ and $\ol\I^\prime$ are surjective algebra homomorphisms; and 
they commute with the actions of $\gl(V')\times \gl_N(\C)$
 on $S_\omega(V^N)$ and $S_\omega(\ol{V}^N)$ respectively.

We have 
{\small
\beq\label{eq:trun-diag}
\begin{tikzcd}
 \I^\prime(S_\omega(V^N)) \arrow[equal]{d}\arrow[equal]{r} &\sum\limits_{\lambda\in P_N}
		\I^\prime(L_{\lambda^\sharp} \ot L_\lambda(\gl_N)) \arrow[equal]{d}\\
S_\omega({V'}^N) \arrow[equal]{r} & \sum\limits_{\lambda\in P_{M_+'|M_-; N}} 
		L_{\lambda_{red}^\sharp}(\gl(V'))\ot L_\lambda(\gl_N)),
\end{tikzcd}
\eeq
}
where $\lambda_{red}^\sharp\in\Z_+^{M'_++M_-}$ associated with $\lambda=(\lambda_1, \lambda_2, \dots)\in P_{M_+'|M_-; N}$ is defined by  
\[
\baln
&\lambda_{red}^\sharp=(\lambda_+, \lambda_-) \quad \text{with}\\ 
&\lambda_+=(\lambda_1, \lambda_2, \dots, \lambda_{M'_+}), \\
&\lambda_-=(\theta(\lambda'_1-M'_+), \theta(\lambda'_2-M'_+), \dots, \theta(\lambda'_{M_-}-M'_+))
\ealn
\]
(cf. Theorem \ref{thm:H-wts}).
The map $\I^\prime$ does not affect the second factor when applied to  $L_{\lambda^\sharp} \ot L_\lambda(\gl_N)\subset S_\omega(V^N)$. Hence $\I^\prime(L_{\lambda^\sharp} \ot L_\lambda(\gl_N)) = M'(\lambda)\ot L_\lambda(\gl_N))$ for some $\gl(V')$-module $M'(\lambda)$. 
It follows the right column of the diagram \eqref{eq:trun-diag} that 
\[
\I^\prime(L_{\lambda^\sharp} \ot L_\lambda(\gl_N)) =\left\{
\begin{array}{l l}
L_{\lambda_{red}^\sharp}(\gl(V')) \ot L_\lambda(\gl_N), &\text{if $\lambda\in P_{M_+'|M_-; N}$}, \\
0, &\text{otherwise}, 
\end{array}
\right.
\]
By similarly considering the actions of $\ol{\I}^\prime$ on both sides of \eqref{eq:decomp-d-1}, 
we obtain 
\[
\ol{\I^\prime}(L_{\lambda^\sharp} ^*\ot L_\lambda(\gl_N)^*) =\left\{
\begin{array}{l l}
L_{\lambda_{red}^\sharp}(\gl(V'))^*\ot L_\lambda(\gl_N)^*, &\text{if $\lambda\in P_{M_+'|M_-; N}$}, \\
0, &\text{otherwise}. 
\end{array}
\right.
\]

If we regard $L_{\lambda^\sharp}$ as a $\gl(V)$-submodule of $V^{\ot |\lambda|}$, then 
$
M'(\lambda)\simeq {I^\prime}^{\ot |\lambda|}(L_{\lambda^\sharp}),
$
and similarly 
$
\ol{\I^\prime}(L_{\lambda^\sharp}^*\ot L_\lambda(\gl_N)^*) \simeq 
(-I')^{\ot |\lambda|}(L_{\lambda^\sharp} ^*)\ot L_\lambda(\gl_N)^*.
$
Hence 
\beq
{I^\prime}^{\ot |\lambda|}(L_{\lambda^\sharp})&\simeq& 
\left\{
\begin{array}{l l}
L_{\lambda_{red}^\sharp}(\gl(V')), &\text{if $\lambda\in P_{M_+'|M_-; N}$}, \\
0, &\text{otherwise};
\end{array}
\right.\\
{\ol{I^\prime}}^{\ot |\lambda|}(L_{\lambda^\sharp}^*)&\simeq& 
\left\{
\begin{array}{l l}
L_{\lambda_{red}^\sharp}(\gl(V'))^*, &\text{if $\lambda\in P_{M_+'|M_-; N}$}, \\
0, &\text{otherwise}. 
\end{array}
\right.
\eeq 

Let us write $Z:=\left(S_\omega(V^N)\ot S_\omega(\ol{V}^N)\right)^{\gl(V)}$, which has generators  
$z_{r s}$ for all $r, s$. 
Let $Z':=(\I'\ot \ol{\I^\prime})(Z)$, and introduce the elements 
\[
z'_{r s} := \sum_{\gamma\in\Gamma'_R}\sum_{i=1}^{\dim V'_\gamma} x(\gamma)_i^r \ol{x}(-\gamma)_i^s, \quad \forall r, s. 
\]
Then it is clear that 
\[
(\I'\ot \ol{\I^\prime})(z_{r s}) = z'_{r s},  \quad \forall r, s. 
\]
As $\I'\ot \ol{\I^\prime}$ commutes with the action of 
$\gl(V')\times \gl_N(\C)\times \gl(V')\times \gl_N(\C)$, 
we have $Z'\subset
\left(S_\omega({V'}^N)\ot S_\omega(\ol{V'}^N)\right)^{\gl(V')}$.
Also, since $\I'\ot \ol{\I^\prime}$ is an algebra homomorphism, 
we conclude that  $Z'$ is generated by $z'_{r s} $ for $r, s=1, 2, \dots, N$, 
as a subalgebra of $S_\omega({V'}^N)\ot S_\omega(\ol{V'}^N)$. 

Denote $\Omega_\lambda
= ({I^\prime}^{\ot |\lambda|}\ot {\ol{I^\prime}}^{\ot |\lambda|})(L_{\lambda^\sharp} \ot L_{\lambda^\sharp} ^*)^{\gl(V)}$,  then clearly $\dim \Omega_\lambda\le 1$.  
We have 
\[
\baln
Z'
\simeq \sum_{\lambda \in P_{M_+|M_-;N}} \Omega_\lambda\ot 
 L_\lambda(\gl_N)\ot L_\lambda(\gl_N)^*.
\ealn
\] 
Let  $\iota: L_{\lambda^\sharp} \ot L_{\lambda^\sharp} ^*\lra 
\End_\C(L_{\lambda^\sharp})$ be the canonical isomorphism. 
Then  
\[
(L_{\lambda^\sharp} \ot L_{\lambda^\sharp} ^*)^{\gl(V)}=\C \iota^{-1}(\id_{L_{\lambda^\sharp}}). 
\]
If ${I^\prime}^{\ot |\lambda|}(L_{\lambda^\sharp})= 0$,  then clearly $\Omega_\lambda=0$. 
Now ${I^\prime}^{\ot |\lambda|}(L_{\lambda^\sharp} )
=L_{\lambda_{red}^\sharp}(\gl(V'))\ne 0$
if and only if $\lambda \in \lambda \in P_{M'_+|M_-;N}$.
In this case, $\ol{I^\prime}^{\ot|\lambda|}(L_{\lambda^\sharp}^*)
=L_{\lambda_{red}^\sharp}(\gl(V'))^*\ne 0$. 
Thus as $\gl(V')$-modules,
\[
L_{\lambda^\sharp}  =L_{\lambda_{red}^\sharp}(\gl(V'))\oplus M_\bot, \quad 
L_{\lambda^\sharp}^* =L_{\lambda_{red}^\sharp}(\gl(V'))^*\oplus M_\bot^*
\]
for some submodule $M_\bot$ and its dual module $M_\bot^*$ such that  ${I^\prime}^{\ot |\lambda|}(M_\bot)=0$, and hence $\ol{I^\prime}^{\ot |\lambda|}(M_\bot^*)=0$. Now one has
\beq
({I^\prime}^{\ot |\lambda|}\ot {\ol{I^\prime}}^{\ot |\lambda|})(\iota^{-1}(\id_{L_{\lambda^\sharp}}))= 
\iota^{-1}(\id_{L_{\lambda_{red}^\sharp}(\gl(V')}).
\eeq
This is particularly easy to see if one expresses $\iota^{-1}(\id_{L_{\lambda^\sharp}})$ in terms of some bases of $L_{\lambda_{red}^\sharp}(\gl(V'))$ and $M_\bot$, and their dual bases.  

Hence, as $\gl_N(\C)\times\gl_N(\C)$-module, 
\beq\label{eq:key-FFT}
Z'
&\simeq& \sum_{\lambda \in P_{M'_+|M_-;N}} 
 L_\lambda(\gl_N)\ot L_\lambda(\gl_N)^*\\
 &\simeq& \sum_{\lambda \in P_{M'_+|M_-;N}} 
 \left(L_{\lambda_{red}^\sharp} \ot L_{\lambda_{red}^\sharp} ^*\right)^{\gl(V')}\ot L_\lambda(\gl_N)\ot L_\lambda(\gl_N)^*
 \nonumber\\
 &\simeq& 
 \left(S_\omega({V'}^N)\ot S_\omega(\ol{V'}^N)\right)^{\gl(V')} = S_\omega({V'}^N\oplus\ol{V'}^N)^{\gl(V')}. \nonumber
\eeq
Thus $S_\omega({V'}^N\oplus\ol{V'}^N)^{\gl(V')}$ is generated by $z'_{r s} $ for $r, s=1, 2, \dots, N$, 
proving the first statement of the theorem. 

This completes the proof of the theorem. 
\end{proof}

Given another positive integer $N'$, there is the $\gl(V)\times \gl_{N'}(\C)$-module algebra
$S_\omega(\ol{V}^{N'})$ with the set of generators
$\{\ol{x}(-\gamma)_i^s\mid  \gamma\in \Gamma_R,  1\le i\le m_\alpha, 1\le s\le N'\}$. 
We have the elements $z_{r s}$,  for $1\le r\le N$ and 
$1\le s\le N'$,  defined by the formula \eqref{eq:zrs}.
Now $S_\omega(V^N\oplus\ol{V}^{N'})=S_\omega(V^N)\ot S_\omega(\ol{V}^{N'})$. 
We can easily deduce from Theorem \ref{thm:FFT} the following result.

\begin{corollary}\label{thm:FFT-1}
Retain notation above. 
\begin{enumerate}
\item (First fundamental theorem)
The subalgebra  $S_\omega(V^N\oplus\ol{V}^{N'})^{\gl(V)}$ 
is generated by the elements $z_{r s}$ for $1\le r\le N, 1\le s\le N'$. 

\item (Second fundamental theorem)
$S_\omega(V^N\oplus\ol{V}^{N'})^{\gl(V)}$ has the following multiplicity free decomposition 
as $\gl_N(\C)\times \gl_{N'}(\C)$-module.
\beq\label{eq:SFT-1}
S_\omega(V^N\oplus\ol{V}^{N'})^{\gl(V)}
=\sum_{\lambda\in P_{M_+|M_-;N_{min}}} L_\lambda(\gl_N)\ot L_\lambda(\gl_{N'})^*, 
\eeq
where $N_{min}=min(N, N')$ is the minimum of  $N, N'$. 
\end{enumerate}
\end{corollary}
\begin{proof}
This can be easily deduced from Theorem \ref{thm:FFT} by applying truncation functors.
Note that the truncation functors defined in Section \ref{sect:H-wts} can be generalised 
to the category of dual tensor representations without any change. 

To be more precise, we let
$N_{max}=max(N, N')$, and consider Theorem \ref{thm:FFT} for $N_{max}$. 
It already covers the case with $N=N'$.  By applying ${\ST}_{N_{max}, N_{min}}\times \id$
or $\id\times {\ST}_{N_{max}, N_{min}}$ to the decomposition of 
$
S_\omega(V^{N_{max}}\oplus\ol{V}^{N_{max}})^{\gl(V)}
$
given by Theorem \ref{thm:FFT}, 
we obtain \eqref{eq:SFT-1} for both cases with $N\ne N'$. 
Similar reasoning as in the proof of the first part of Theorem \ref{thm:FFT} proves 
part (1) of the corollary.
\end{proof}

We remark that analogous 
FFTs and SFTs were established in recent years for
Lie superalgebras \cite{DLZ, LZ17, LZ21, Zy1} 
and quantum (super)groups \cite{CW23, LZZ11, LZZ20, Zy},

\subsection{Colour Howe duality of type $(\gl(V(\Gamma, \omega)), \gl(V'(\Gamma, \omega)))$}
\subsubsection{Colour Howe duality of type $(\gl(V(\Gamma, \omega)), \gl(V'(\Gamma, \omega)))$}
We retain notation of Section \ref{sect:FT}. 
Consider
$S_\omega(\ol{V}^N\oplus V^N)=S_\omega(\ol{V}^N\oplus V^N)$
as module algebra for $\fg_N=$ $
 (\gl(V)\times \gl_N(\C))\times (\gl(V)\times \gl_N(\C)) 
 $ for all $N$.
Now $\fg_N\supset \gl_N(\C)\times \gl_N(\C)\supset \text{diagonal subalgebra $\gl_N(\C)$}$. 
Let $F(V)^{(N)}:=S_\omega(\ol{V}^N\oplus V^N)^{\gl_N(\C)}$ be the subspace of invariants with respect to the action of the diagonal subalgebra $\gl_N(\C)$, i.e., 
\[
F(V)^{(N)}
=\left\{\left. z\in S_\omega(\ol{V}^N\oplus V^N)\right| X\cdot z= 0, \ \forall X\in \gl_N(\C)\right\}.
\]
This is a $\gl(V)\times \gl(V)$-module subalgebra of $S_\omega(\ol{V}^N\oplus V^N)$.

\begin{lemma}\label{lem:at-N} Let $N$ be any fixed positive integer. The $\gl(V)\times \gl(V)$-module algebra 
$F(V)^{(N)}=S_\omega(\ol{V}^N\oplus V^N)^{\gl_N(\C)}$ is generated by 
$x(\beta, \alpha)_{j i}^{(N)}=\sum_{r=1}^N \ol{x}(-\beta)_j^r x(\alpha)_i^r$ for all $\alpha, \beta, i, j$, and 
decomposes into
\[
F(V)^{(N)}=\sum_{\lambda\in P_{M_+|M_-; N}} L_{\lambda^\sharp}^*\ot L_{\lambda^\sharp}. 
\]
\end{lemma}
\begin{proof}
It follows from the FFT of invariant theory for $\gl_N(\C)$ that the elements $x(\beta, \alpha)_{j i}^{(N)}$ 
generate $S_\omega(V^N\oplus \ol{V}^N)^{\gl_N(\C)}$. 
Using the colour  Howe dualities of type $(\gl(V), \gl_N(\C))$ on $S_\omega(V^N)$ and $S_\omega(\ol{V}^N)$
(see Theorem \ref{thm:commut} (2) and Theorem \ref{thm:Howe-d}), 
we obtain 
\[
\baln
F(V)^{(N)}&\simeq \sum_{\lambda, \mu\in P_{M_+|M_-; N}} L_{\lambda^\sharp}^*\ot \left(L_\lambda(\gl_N)^*\ot L_\mu(\gl_N)\right)^{\gl_N(\C)}\ot  L_{\mu^\sharp}\\
&\simeq\sum_{\lambda\in P_{M_+|M_-; N}} L_{\lambda^\sharp}^*\ot L_{\lambda^\sharp}, 
\ealn
\]
as $\gl(V)\times\gl(V)$-module. This proves the claimed decomposition. 
\end{proof}

For any $N, N'$ such that $N\le N'$, we have the surjective   
algebra homomorphism 
$f_{N, N'}: F(V)^{(N')}\lra F(V)^{(N)}$ such that $x(\beta, \alpha)_{j i}^{(N')}\mapsto x(\beta, \alpha)_{j i}^{(N)}$. 
This map  commutes with the $\gl(V)\times\gl(V)$-action. 
\begin{lemma}\label{lem:poly-alg}
There is an inverse system $(F(V)^{(N)}, f_{N, N'})$ of $\gl(V)\times\gl(V)$-module algebras, whose 
 inverse limit 
\[
F(V)=\lim_{\substack{ \longleftarrow\\ N\in\Z_{>0} }}  F(V)^{(N)} 
\]
has the following multiplicity free decomposition as $\gl(V)\times\gl(V)$-module.
\beq\label{eq:F(V)}
F(V)=\sum_{\lambda\in P_{M_+|M_-}} L_{\lambda^\sharp}^*\ot L_{\lambda^\sharp}. 
\eeq
\end{lemma}
\begin{proof}
For any $N\le N'\le N''$, we clearly have $f_{N, N''}=f_{N, N'}\circ f_{N', N''}$. Hence  we have the inverse system  
$(F(V)^{(N)}, f_{N, N'})$ of $\gl(V)\times\gl(V)$-module algebras. 
The decomposition \eqref{eq:F(V)} for the inverse limit follows from Lemma \ref{lem:at-N}.   
\end{proof}

We denote $ x(\beta, \alpha)_{j i}=\lim\limits_{\longleftarrow} x(\beta, \alpha)_{j i}^{(N)}$. 

The symmetric $(\Gamma, \omega)$-algebra $S_\omega(V^*\ot V)$ over $V^*\ot V$
 clearly has the structure of a $\gl(V)\times\gl(V)$-module algebra.
 
 \begin{lemma}\label{lem:at-infty}
 As $\gl(V)\times\gl(V)$-module algebra, $S_\omega(V^*\ot V)$ is isomorphic to $F(V)$, 
 and hence has the following multiplicity free decomposition
 \[
S_\omega(V^*\ot V)= \sum_{\lambda\in P_{M_+|M_-}} L_{\lambda^\sharp}^*\ot L_{\lambda^\sharp}. 
 \]
 
 \end{lemma}
\begin{proof}
Denote by $X(\beta, \alpha)_{j i}$ the image of $e(-\beta)_j\ot e(\alpha)_i$ in $S_\omega(V^*\ot V)$ for all $\alpha, \beta, i, j$. Then for any $N$, there is the surjective homomorphism of $\gl(V)\times\gl(V)$-module algebras
\[
g_N: S_\omega(V^*\ot V)\lra F(V)^{(N)}, \quad X(\beta, \alpha)_{j i}\mapsto x(\beta, \alpha)_{j i}^{(N)}, \quad \forall \alpha, \beta, i, j.
\] 
This shows that $S_\omega(V^*\ot V)$ is isomorphic to the inverse limit $F(V)$ of 
the inverse system $(F(V)^{(N)}, f_{N, N'})$. Now the second statement follows 
from Lemma \ref{lem:poly-alg}, completing the prove. 
\end{proof}

\begin{remark} 
Lemma \ref{lem:at-infty} can also be proved,  in a much simpler way, 
by  using the coordinate algebra of the $\Gamma$-graded ``general linear group''; 
see Section \ref{sect:coord-HD}. 
\end{remark}
 
Let $V$ and $V'$ be $\Gamma$-graded vector spaces with $M_\pm = \dim V_\pm$ 
 and $M'_\pm =\dim V'_\pm$.  Given any $\lambda \in P_{M_+|M_-}\cap P_{M'_+|M'_-}$, 
 we denote by $\lambda^\sharp(M_+|M_-)$ and $\lambda^\sharp(M'_+|M'_-)$ respectively 
 the associated $\gl(V)$-weight and $\gl(V')$-weight 
 as defined in Section \ref{sect:H-wts} (see Theorem \ref{thm:H-wts} in particular). 
 
 We have the following result.

 \begin{theorem} \label{thm:HD-general} 
 The $\gl(V)\times\gl(V')$-module algebra $S_\omega(V^*\ot V')$ has the following multiplicity free decomposition
 \[
 S_\omega(V^*\ot V')=\sum_{\lambda\in P_{M_+|M_-}\cap P_{M'_+|M'_-}} L_{\lambda^\sharp(M_+|M_-)}^*(\gl(V))\ot L_{\lambda^\sharp(M'_+|M'_-)}(\gl(V')). 
 \]
\end{theorem}
\begin{proof}
 This follows from Lemma \ref{lem:at-infty} by using truncation arguments. 
 To indicate the pertinent steps of the proof, 
 let $\hat{V}$ be a $\Gamma$-graded vector space with $\hat{M}_\pm =\dim \hat{V}^\pm$. 
It follows \eqref{eq:F(V)} and Lemma \ref{lem:at-infty} that 
\beq\label{eq:hat-V}
S_\omega(\hat{V}^*\ot\hat{V})= \sum_{\lambda\in P_{\hat{M}_+|\hat{M}_-}} L_{\lambda^\sharp(\hat{M}_+|\hat{M}_-)}^*(\gl(\hat{V}))\ot L_{\lambda^\sharp(\hat{M}_+|\hat{M}_-)}(\gl(\hat{V})). 
\eeq

Assume that $V\subset \hat{V}\supset  V'$, thus $\hat{V}= V\oplus V_\bot$ and $\hat{V}= V'\oplus V'_\bot$ for some subspaces $V_\bot$ and $V'_\bot$ of $\hat{V}$. The general linear Lie $(\Gamma, \omega)$-algebra $\gl(\hat{V})$ contains the subalgebras $\gl(V)$ and $\gl(V')$. 
 Let $J\in \gl(\hat{V})$ be the element which acts on $V$ by the identity map, and annihilates $V_\bot$. 
 We also have $J'\in \gl(\hat{V})$, which is similarly defined.   
 Now $J$ also acts on $\hat{V}^*$ since $\gl(\hat{V})$ does. 
 
For each $r$, we have the  surjective map
 $
 (J\ot J')^{\ot r}:  (\hat{V}^*\ot\hat{V})^{\ot r}\lra (V^*\ot V')^{\ot r},
 $
 which commutes with the action of $\gl(V)\times \gl(V')$. It restricts to a surjective map 
 \beq\label{eq:JJ-r}
 (J\ot J')^{\ot r}: S^r_\omega(\hat{V}^*\ot\hat{V})\lra S^r_\omega(V^*\ot V').
 \eeq  
 [The homogeneous component  of degree $r$ of a symmetric $(\Gamma, \omega)$-algebra over a vector space is regarded as an $\Sym_r$-submodule of the $r$-th power of the vector space.]  
 For any $\lambda\vdash r$, the  $(J\ot J')^{\ot r}$  image of   
 $L_{\lambda^\sharp(\hat{M}_+|\hat{M}_-)}^*(\gl(\hat{V}))\ot L_{\lambda^\sharp(\hat{M}_+|\hat{M}_-)}(\gl(\hat{V}))$ $\subset S^r_\omega(\hat{V}^*\ot\hat{V})$ is  
$
 L_{\lambda^\sharp(M_+|M_-)}^*(\gl(V))\ot L_{\lambda^\sharp(M'_+|M'_-)}(\gl(V')). 
$
 
Let $\J: S_\omega(\hat{V}^*\ot\hat{V})\lra S_\omega(V^*\ot V')$ be the surjective algebra homomorphism which restricts to the maps \eqref{eq:JJ-r} on homogeneous components.    
Applying $\J$ to \eqref{eq:hat-V}, we easily obtain the theorem. 
 \end{proof}

\subsubsection{Generalisation of Theorem \ref{thm:FFT}}
 Retain notation of Theorem \ref{thm:HD-general}. 
 Let $V''$ be another finite dimensional $\Gamma$-graded vector space, and denote 
 $M''_\pm =\dim V''_\pm$. 
We write $m_\alpha=\dim V_\alpha$ as before, and denote $m'_\alpha=\dim V'_\alpha$, 
and $m''_\alpha=\dim V''_\alpha$. Let 
\[
\baln
&\Gamma_R(V)=\{\alpha\in \Gamma\mid \dim V_\alpha>0\},\\ 
&\Gamma_R(V')=\{\alpha\in \Gamma\mid \dim V'_\alpha>0\},\\
&\Gamma_R(V'')=\{\alpha\in \Gamma\mid \dim V''_\alpha>0\}.
\ealn
\]
We have the following 
 \[
 \baln
\text{basis  for $V$}: & \text{ $\{e(\alpha)_i\mid 1\le i\le m_\alpha, \alpha\in \Gamma_R(V)\}$}, \\
\text{dual basis  for $V^*$}: & \text{ $\{\ol{e}(-\alpha)_i\mid 1\le i\le m_\alpha, \alpha\in \Gamma_R(V)\}$};\\
\text{basis  for $V'$}: & \text{ $\{e'(\alpha)_i\mid 1\le i\le m_\alpha, \alpha\in \Gamma_R(V')\}$}, \\
\text{dual basis  for ${V'}^*$}: & \text{ $\{\ol{e}'(-\alpha)_i\mid 1\le i\le m''_\alpha, \alpha\in \Gamma_R(V')\}$};\\
\text{basis  for $V''$}: & \text{ $\{e''(\alpha)_i\mid 1\le i\le m_\alpha, \alpha\in \Gamma_R(V'')\}$}, \\
\text{dual basis  for ${V''}^*$}: & \text{ $\{\ol{e}''(-\alpha)_i\mid 1\le i\le m''_\alpha, \alpha\in \Gamma_R(V'')\}$}.
 \ealn
 \]
Let $pr: T(V^*\ot V'')\lra S_\omega (V^*\ot V'')$ and $pr': T({V''}^*\ot V')\lra S_\omega ({V''}^*\ot V')$ be the canonical surjections, and denote 
\[
\baln
&x(\alpha, \beta)_{i j}=pr\big( \ol{e}(-\alpha)_i\ot e''(\beta)_j\big), \quad 
y(\beta, \gamma)_{j k}=pr'\big( \ol{e}''(-\beta)_j\ot e'(\gamma)_k\big), \\
&z(\alpha, \gamma)_{i k}= \sum_{\beta\in \Gamma_R(V'')}\sum_{j=1}^{m''_\beta} x(\alpha, \beta)_{i j}y(\beta, \gamma)_{j k},
\ealn
\]
for  $\alpha\in \Gamma_R(V), 
1\le i\le m_\alpha, \  \beta \in \Gamma_R(V''),  1\le j\le m''_\beta, \ 
\gamma \in \Gamma_R(V'),  1\le k\le m'_\gamma. $

Let $S_\omega(V, V', V'')=S_\omega (V^*\ot V'')\ot S_\omega ({V''}^*\ot V')= S_\omega(V^*\ot V''\oplus {V''}^*\ot V')$, and consider the elements 
\[
\baln
&z(\alpha, \gamma)_{i k}= \sum_{\beta\in \Gamma_R(V'')}\sum_{j=1}^{m''_\beta} x(\alpha, \beta)_{i j}y(\beta, \gamma)_{j k}, 
\ealn
\]
for  $\alpha\in \Gamma_R(V), 
1\le i\le m_\alpha, \  \gamma \in \Gamma_R(V'),  1\le k\le m'_\gamma. $ 

We have the following result.  
\begin{corollary}\label{cor:gFFT}
 Retain notation above, and let $Z_{V''}(V^*, V'):=S_\omega(V, V', V'')^{\gl(V'')}$. 
 \begin{enumerate}[i)]
 \item (FFT)
 $Z_{V''}(V^*, V')$ is generated by the elements 
$z(\alpha, \gamma)_{i k}$, for $\alpha\in \Gamma_R(V), 
 1\le i\le m_\alpha, \  \gamma \in \Gamma_R(V'),  1\le k\le m'_\gamma$.  

\item (SFT)
As a $\gl(V)\times \gl(V')$-module algebra, it has the following multiplicity free decomposition 
\beq\label{eq:gSFT}
Z_{V''}(V^*, V')
=\sum_{\lambda} L_{\lambda^\sharp(M_+|M_-)}(\gl(V))^*\ot L_{\lambda^\sharp(M'_+|M'_-)}(\gl(V')), 
\eeq
where the sum is over all $\lambda \in P_{M_+|M_-}\cap P_{M''
_+|M''_-}\cap P_{M'_+|M'_-}$. 
\end{enumerate}
\end{corollary}
\begin{proof}
The case with 
$ P_{M_+|M_-}\subset P_{M''_+|M''_-}\supset P_{M'_+|M'_-}$ easily follows from Theorem \ref{thm:HD-general}. Truncation arguments can prove the other cases. 
We omit the details. 
\end{proof}

\subsection{Example: a Lie $q$-superalgebra $\gl_q(m|n)$}\label{sect:q-alg}

We consider a very simple example of the colour Howe duality, 
which has much similarity to the quantum Howe duality for  
quantum general linear (super)groups \cite{WZ, Z03, Zy}.
The commutative factor involved is a generalisation of that for superalgebras 
depending on a parameter $q\in \C^*$. 
The fact that this rather simplistic example can accounts for aspects 
of quantum supergroups and related non-commutative geometry illustrates 
the power of the theory of Lie $(\Gamma, \omega)$-algebras. 

Let us first discuss some commutative factors, which will be needed below. 

\begin{example}\label{eg:factors}
Since any abelian group $\Gamma$ is a $\Z$-module by definition, we consider $\Z$-bilinear  forms $\Gamma \times \Gamma \lra \Z$. 
\begin{enumerate}[a).]
\item \label{a)}   Any symmetric or skew symmetric $\Z$-bilinear  form $(\ , \ )_0: \Gamma \times \Gamma \lra \Z$ leads to a commutative factor
\beq\label{eq:factor-s}
\omega_0: \Gamma\times \Gamma \lra \{1, -1\}, \quad \omega_0(\alpha, \beta) = (-1)^{(\alpha, \beta)_0}, \quad \alpha, \beta\in\Gamma. 
\eeq

A familiar special case is $\Gamma=\Z_2$, and $\omega=\omega_0$ with $(\alpha, \beta)_0=\alpha\beta$, where
algebras over $\Vect$ are superalgebras. 

Another interesting special case is $\Gamma=\Z^n$ for some $n\ge 1$, 
and $\omega=\omega_0$ with $(\alpha, \beta)_0=\sum_{i=1}^n \alpha_i \beta_i$
for any $\alpha= (\alpha_1, \dots, \alpha_n)$ and $\beta = (\beta_1, \dots, \beta_n)$. 
Algebras over $\Vect$ in this case are the $\Z^n$-graded algebras studied in \cite[Appendix A]{Z24}, 
which provide a suitable framework for Green's ansatz for parafermions \cite{G}. 

\item \label{b)}
Fix $q\in\C^*=\C\backslash\{0\}$ such that $q\ne 1$.  Any skew symmetric $\Z$-bilinear form
$
( \ , \  ): \Gamma\times \Gamma \lra \Z
$
leads to a commutative factor
 \beq\label{eq:factor-prime}
\omega_q: \Gamma\times \Gamma \lra \C^*, \quad \omega_q(\alpha, \beta) = q^{(\alpha, \beta)}, \quad \alpha, \beta\in\Gamma.
\eeq
Algebras over $\Vect$ in this case are $q$-algebras analogous to those in the context of quantum groups. 

\item 
The product of the above commutative factors is also a commutative factor. 
\beq \label{eq:factor}
\wt\omega_q: \Gamma\times \Gamma \lra \C^*, \quad  \wt\omega_q(\alpha, \beta)=  \omega_0(\alpha, \beta)\omega_q(\alpha, \beta),
\quad \alpha, \beta\in\Gamma.
\eeq

We will study a special case of this, which 
provides an alternative setting for the quantised coordinate algebra 
of $\C^{m|n}$ \cite{Z98, Zy} independent of the theory of quantum supergroups \cite{BGZ, Z93, Z98}.  
\end{enumerate}

\end{example}

Fix non-negative integers $m, n$. 
Let $\Gamma=\Z^{m+n}$, and choose the ordered $\Z$-basis 
$(\varepsilon_1,  \varepsilon_2, \dots, \varepsilon_{m+n})$ 
with $\varepsilon_i = (\underbrace{0, \dots, 0}_{i-1}, 1, 0, \dots, 0)$ for $i=1, 2, \dots, m+n$. 
Write $[\varepsilon_i]=0$ if $i\le m$ and $1$ otherwise. 

Introduce $\Z$-bilinear forms $( \ , \ )_0,\  ( \ , \ ): \Z^{m+n}\times\Z^{m+n}\lra \Z$ such that 
\beq\label{eq:q-form}
(\varepsilon_i, \varepsilon_j)_0&=&
\left\{
\begin{array}{l l}
1, & \text{both $i, j>m$},\\
0, &\text{otherwise}, 
\end{array}
\right.\\
(\alpha, \beta) &=& \begin{pmatrix}
\alpha_1 &  \alpha_2 & \dots & \alpha_{m+n}
\end{pmatrix}J
\begin{pmatrix}
\beta_1 \\
\beta_2 \\
\vdots \\
\beta_{m+n} \\
\end{pmatrix},
\eeq
for $\alpha= \sum_{i=1}^{m+n}
\alpha_i \varepsilon_i$ and $\beta=
\sum_{i=1}^{m+n}
\beta_i \varepsilon_i$, 
where 
\[
J=\begin{pmatrix}
	0&	1&	1&	\dots & 1\\
-1&		0&	1&	\dots & 1\\
-1&		-1&	0&	\dots & 1\\
\dots & \dots & \dots & \dots & \dots \\
-1 & -1& -1& \dots & 0
\end{pmatrix}.
\]
%
Clearly $(\ , \ )_0$ is symmetric and $(\ , \ )$ is skew symmetric. 
By part $c)$ of Example \ref{eg:factors}, we have the following commutative factor  
for any given $q\in\C^*$.  
\beq\label{eq:q-factor}
\omega: \Z^{m+n}\times\Z^{m+n}\lra\C^*, \quad
\omega(\alpha, \beta)= (-1)^{(\alpha, \beta)_0} q^{(\alpha, \beta)}.
\eeq

Consider the $\Z^{m+n}$-graded vector space 
$V=\sum_{\alpha\in\Z^{m+n}} V_\alpha$ such that  
$V_\alpha=\C$ if $\alpha= \varepsilon_i$ for all $i$, and $0$ otherwise.  
Then $\Z^{m+n}_R= \{\varepsilon_1,  \varepsilon_2, \dots, \varepsilon_{m+n}\}$.  
We order its elements by $\varepsilon_i< \varepsilon_{i+1}$ for all valid $i$. 

\begin{definition}
Retain notation above. 
Denote by $\gl_q(m|n)$ the general linear Lie $(\Z^{m+n}, \omega)$-algebra of $V$, 
and refer to it as the general linear Lie $q$-superalgebra. 
\end{definition}

Justification for the terminology is that $\gl_q(m|n)$ 
has a close connection to the quantum general linear supergroup \cite{Z93, Z98}, 
as we will see from Remark \ref{rmk:q-quantum}. 

To describe the structure of $\gl_q(m|n)$, 
we write $\alpha_{i j}=\varepsilon_i-\varepsilon_j$ for all $i, j$. 
Now $\gl_q(m|n)$ has a basis  $\{ \SE_{i j}\mid i, j=1, 2, \dots, m+n \}$, whose elements obey the following generalised commutation relations
\beq
[\SE_{i j},  \SE_{k\ell} ]= \delta_{j k} \SE_{i \ell} 
-(-1)^{([\varepsilon_i]+[\varepsilon_j])([\varepsilon_k]+[\varepsilon_\ell])} q^{(\alpha_{i j},  \alpha_{k \ell})}  \delta_{\ell i} \SE_{k j}, 
\eeq 
The  skew $\omega$-symmetry of the generalised Lie bracket is given by
\[
[\SE_{i j},  \SE_{k\ell} ]= -(-1)^{([\varepsilon_i]+[\varepsilon_j])([\varepsilon_k]+[\varepsilon_\ell])} q^{(\alpha_{i j}, \alpha_{k\ell})} [\SE_{k\ell}, \SE_{i j}], 
\]
and the generalised Jacobian identity is given by 
\[
[\SE_{i j},  [\SE_{k\ell}, \SE_{h g}]] = [[\SE_{i j},  \SE_{k\ell}], \SE_{h g}] 
+ (-1)^{([\varepsilon_i]+[\varepsilon_j])([\varepsilon_k]+[\varepsilon_\ell])} q^{(\alpha_{i j}, \alpha_{k\ell})}  [\SE_{k\ell}, [\SE_{i j},  \SE_{h g}]].
\]

Note that
$\Px=S_\omega(V^N)$  in the present case is a graded commutative associative $(\Z^{m+n}, \omega)$-algebra. 
It is  generated by $x_i^r:=x(\varepsilon_i)^r$, for $1\le i \le m+n$ and $1\le r\le N$, with $d(x_i^r)= \varepsilon_i$. 
[Here $x_i^r$ and $x(\varepsilon_i)^r$ are symbols, not $r$-th powers.]
The corresponding graded commutative $(\Z^{m+n}, \omega)$-algebra $\Pd$ is generated by the elements  $\partial_{i, r}=\frac{\partial}{\partial x_ i^r}$ 
for $1\le i\le m+n$ and $1\le r\le m+n$, where $\partial_{i, r}$  
are of degree $-\varepsilon_i$ for all $r$. 
The Weyl $(\Gamma, \omega)$-algebra $\CW_\omega$ is 
generated by the subalgebras $\Px$ and $\Pd$, subject to the relations \eqref{eq:CCR}. 
The complete set of defining relations of $\CW_\omega$ are  
\beq
&x_i^r x_i^s =  (-1)^{[\varepsilon_i]} x_i^s x_i^r, 
\quad x_i^r x_j^s =  (-1)^{[\varepsilon_i] [\varepsilon_j]} q x_j^s x_i^r, \quad \forall r, s,  i<j;
\label{eq:x-relations}\\
&\partial_{i, r} \partial_{i, s} =  (-1)^{[\varepsilon_i]} \partial_{i, s} \partial_{i, r}, 
\quad \partial_{i, r} \partial_{j, s} =  (-1)^{[\varepsilon_i] [\varepsilon_j]} q \partial_{j, s} \partial_{i, r}, \quad \forall r, s,  i<j;\\
&\partial_{i, r} x_i^s - (-1)^{[\varepsilon_i][\varepsilon_j]} x_i^s \partial_ {i, r} =\delta_{r s} , 
\quad \forall i,  r, s,\\
&\partial_{i, r} x_j^s - (-1)^{[\varepsilon_i][\varepsilon_j]} q^{-1} x_j^s \partial_ {i, r} = 0, 
\quad \forall r, s,  i<j. 
\eeq

Now we have the pair of Lie colour (super)algebras in $\CW_\omega$, 
\beq
&&\text{$\gl_N(\C)$  spanned by }  E^{r s} := \sum_{i=1}^{m+n} x_ i^r \partial_{i, s}, \quad r, s = 1, 2, \dots, N, \\
&&\text{$\gl_q(m|n)$ spanned by } 
\SE_{i j} =  \sum_{r=1}^N x_ i^r \partial_ {j, r}, \quad i, j=1, 2, \dots, m+n, 
\eeq
which commute with each other. Here $\gl_N(\C)$ is the usual general linear Lie algebra.

Theorems \ref{thm:commut} and \ref{thm:H-wts} immediately leads to the following result. 
\begin{lemma}\label{lem:q-alg}
The $q$-commutative superalgebra $\Px$ is a  $\U(\gl_q(m|n))\ot \U(\gl_N(\C))$
module algebra, which has the following multiplicity free decomposition.
\[
\Px=\sum_{\substack{\lambda\in P_{m|n, N}}} L_{\lambda^\sharp}\ot L_\lambda(\gl_{N}).
\]
\end{lemma}

\begin{remark}\label{rmk:q-quantum}
One can see from the relations \eqref{eq:x-relations} that in the $N=1$ case,  $\Px$ coincides with
the quantum coordinate algebra of the superspace $\C^{m|n}$ \cite{Z98, Zy}  (also see \cite[\S6]{M}). 
This fact connects $\gl_q(m|n)$ to the quantum general linear supergroup \cite{Z93, Z98} (also see \cite[\S6]{M}).
\end{remark}

\begin{remark}
Lemma \ref{lem:q-alg} remains valid for $q$ being a root of unity. This is very different from the case of quantum groups \cite{D, J} and quantum supergroups \cite{BGZ, Y91, Y94, ZGB, Z93}, 
where the analysis of the analogue of $\C_\omega[{\bf x}]$ at roots of unity becomes exceedingly difficult. 
\end{remark} 

\begin{remark}
The general linear Lie $q$-superalgebra $\gl_q(m|n)$ for generic $q$ may be regarded as a
deformation of the general linear Lie superalgebra $\gl(\C^{m|n})$ in the category of 
Lie $(\Gamma, \omega)$-algebras. 
We recover from the lemma the usual $(\gl_{m|n}(\C), \gl_N(\C))$ Howe duality  
when we specialise $q\to 1$.
\end{remark} 

We mention that a different connection between quantum groups and Lie  colour algebras was discovered in \cite{AlI}.

\section{Unitarisable modules for general linear Lie $(\Gamma, \omega)$-algebras}

We study unitarisable modules for general linear Lie $(\Gamma, \omega)$-algebras.
The main result obtained is the classification of unitarisable modules 
with respect to two $\ast$-structures referred to as ``compact". 

A general theory of $\ast$-structures and their unitarisable modules will be briefly discussed for  
various types of algebras over $\Vect$. It is interesting in its own right. 
The guiding example for building such a theory is the Weyl  
$(\Gamma, \omega)$-algebra as an associative $\ast$-$(\Gamma, \omega)$-algebra and the Fock space as a unitarisable module for it, see Example \ref{eg:Fock}. 

\subsection{Hopf $\ast$-$(\Gamma, \omega)$-algebras and unitarisable modules}\label{sect:ast-Hopf}

As we will see presently, the existence of any $\ast$-structure requires
the commutative factor have the Unit Modulus Property (see \eqref{eq:inv-ast}).
For Hopf  $(\Gamma, \omega)$-algebras, 
it further requires the Antipode Conditions \ref{assum}. 

\subsubsection{Associative $\ast$-$(\Gamma, \omega)$-algebras and unitarisable modules}
Let $\Gamma$ be an abelian group with a commutative factor $\omega: \Gamma\times \Gamma\lra \C^*$.
In order to define $\ast$-structures for any given type of algebras over $\Vect$, we need to assume the following 

\smallskip

\noindent{\bf Unit Modulus Property} of the commutative factor $\omega$:
\beq\label{eq:inv-ast}
\omega(\alpha, \beta)^* = \omega(\alpha, \beta)^{-1}, \quad \forall \alpha, \beta\in\Gamma,
\eeq
where $\omega(\alpha, \beta)^*$ is the complex conjugate of $\omega(\alpha, \beta)$. 

The necessity of the condition will be explained in Remark \ref{rmk:ump}.

\begin{example} 
The commutative factor in Example \ref{eg:factors}.(a) has the unit modulus property.  
However commutative factors in Examples \ref{eg:factors}.(b), (c) have the property only for $|q|=1$. 
\end{example}

\begin{definition}\label{def:ast-asso}
A $\ast$-structure for an associative $(\Gamma, \omega)$-algebra $A=\sum_{\alpha\in\Gamma} A_\alpha$ is a $\C$-conjugate linear involution
$\ast: A \lra A $ such that   
\begin{enumerate}
\item $\ast(A _\alpha) =A _{-\alpha}$ for all $\alpha\in\Gamma$; and 
\item 
$\ast(x y) = \ast(y)\ast(x)$ for all $x, y\in A $.  
\end{enumerate}
Call $A$ an associative $\ast$-$(\Gamma, \omega)$-algebra if it is equipped with such a $\ast$-structure.
\end{definition}
The requirement that $\ast$ be $\C$-conjugate linear means that $\ast(c x)=c^* \ast(x)$ for all $c\in\C$ and $x\in A$, where $c^*$ is the complex conjugate of $c$. Also $^\ast$ being an involution means that $\ast(\ast(x))=x$ for all $x\in A$. Note that $\ast(1)=1$ by condition (2).

\begin{remark}\label{rmk:ump}
The unit modulus property of the commutative factor 
arises from the requirement that $\ast$-conjugates of $\omega$-commuting elements
be $\omega$-commuting. 
Consider $\omega$-commuting elements $x$ with $d(x)=\alpha$ and $y$ with $d(y)=\beta$ in an 
associative $\ast$-$(\Gamma, \omega)$-algebra $A$. Then 
$x y = \omega(\alpha, \beta) y x$ implies
\[
 \ast(y) \ast(x) = \omega(\alpha, \beta)^*  \ast(x) \ast(y). 
\] 
However, we also have 
$
 \ast(y) \ast(x) = \omega(-\beta, -\alpha)  \ast(x) \ast(y),  
$ 
since $d(\ast(x))=-\alpha$ and $d(\ast(y))=-\beta$. 
These relations together require
\[
\omega(\alpha, \beta)^* = \omega(-\beta, -\alpha)= \omega(\beta, \alpha)= \omega(\alpha, \beta)^{-1}.
\]

\end{remark}

Let $A$ and $B$ be  $\ast$-$(\Gamma, \omega)$-algebras, and let 
$\varphi: A\lra B$ be a $(\Gamma, \omega)$-algebra homomorphism. Call  $\varphi$ a $\ast$-$(\Gamma, \omega)$-algebra homomorphism if 
\beq
\ast\circ\varphi = \varphi  \circ \ast. 
\eeq

\medskip
\noindent
{\bf Notation}. It is customary to write $\ast(x)$ as $x^+$, and we shall adopt this notation hereafter. 

\medskip

Recall that the tensor product $A\ot_\C B$ of any two associative $(\Gamma, \omega)$-algebras $A$ and $B$  is again an 
associative $(\Gamma, \omega)$-algebra with the multiplication defined by \eqref{eq:multip-tensor}. 

\begin{lemma} \label{lem:tensor-ast}
Let $A$ and $B$ be associative $\ast$-$(\Gamma, \omega)$-algebras. Then $A\ot_\C B$ is an associative $\ast$-$(\Gamma, \omega)$-algebra with the $\ast$-structure 
\beq\label{eq:ast-tensor}
 \ast: A\ot B\lra A\ot B, \quad  (x\ot y)^+=\omega(d(y^+), d(x^+))x^+\ot y^+. 
\eeq
[Observe that this is equivalent to
$((x\ot 1)(1\ot y))^+= (1\ot y^+)  (x^+\ot 1)$.]

\begin{proof}
The map clearly satisfies condition (1) of Definition \ref{def:ast-asso}.

As $d(u^+)=-d(u)$ for any homogeneous $u$,  we have  
$
(x\ot y)^+= \omega(d(y), d(x))x^+\ot y^+ 
$
for any $x\ot y$ of $A\ot B$.
Let $P=u\ot v$ and $Q=x\ot y$ be elements of $A\ot B$. We have
\[
\baln
(PQ)^+
&=\omega(d(v), d(x))^{-1} (u x \ot  v y)^+ \\
&=\omega(d(v), d(u)) \omega(d(y), d(x)) \omega(d(y), d(u))x^+u^+\ot y^+ v^+, \\
Q^+ P^+
&=  \omega(d(y), d(x)) \omega(d(v), d(u))  (x^+\ot y^+) (u^+ \ot v^+) \\
&=  \omega(d(y), d(x)) \omega(d(v), d(u)) \omega(d(y), d(u)) x^+u^+\ot y^+ v^+.
\ealn
\]
Hence $(PQ)^+= Q^+ P^+$, proving condition (2).   This proves the lemma. 
\end{proof}
\end{lemma}

\begin{definition}
An $A $-module $V$ is unitarisable if there exists a positive definite sesquilinear form 
$\langle \ , \ \rangle: V\times V \lra \C$, which is contravariant in the sense that  
\beq\label{eq:contravar}
\langle x v,  v' \rangle = \langle v,  x^+ v' \rangle, \quad \forall x\in A, \ v, v' \in V.
\eeq
If a unitarisable module $V=\sum_{\alpha}V_\alpha$ is 
a $\Gamma$-graded $A$-module, and the positive definite contravariant 
sesquilinear form satisfies $\langle V_\alpha, V_\beta \rangle=\{0\}$ for all $\alpha\ne\beta$, 
we call $V$ a $\Gamma$-graded unitarisable $A$-module.
\end{definition}

\begin{remark}
Throughout this section,  $\langle \ , \  \rangle$ denotes a sesquilinear form. 
There should not be any danger of confusing it with the notation for dual space pairing used in other sections, 
as the distinction will be clear from the context. 
\end{remark}

\begin{lemma} 
Unitarisable modules for any associative $\ast$-$(\Gamma, \omega)$-algebra are semi-simple. 
\end{lemma}
\begin{proof}
The is a standard fact, and its proof is well known. 
Let $A$ be a $\ast$-$(\Gamma, \omega)$-algebra, and let $V$ be a unitarisable $A$-module. 
If $V_0\subset V$ is a submodule, let $V^\bot = \{v\in V\mid \langle V_0,  v \rangle =\{0\}\}$. Then $V^\bot$ is a submodule of $V$. Positive definiteness of the sesquilinear implies that $V^\bot\cap V_0=\{0\}$. Hence $V=V_0\oplus V^\bot$.
\end{proof}

\begin{theorem}\label{thm:tensor-prod}
Let $A$ and $B$ be  associative $\ast$-$(\Gamma, \omega)$-algebras, and let $V$ and $W$ be $\Gamma$-graded unitarisable modules for $A$ and $B$ respectively. Then the tensor product $V\ot_\C W$ is a $\Gamma$-graded unitarisable $A\ot B$-module.
\end{theorem}
\begin{proof}
Now $V\ot_\C W$ is a $\Gamma$-graded $A\ot B$-module with the action defined by \eqref{eq:act-tensor}. 
Denote by
$\langle \ , \ \rangle_V$ and $\langle \ , \ \rangle_W$ the positive definite contravariant  sesquilinear forms on 
$V$ and  $W$ respectively. 
Define a sesquilinear form $\Langle \ , \ \Rangle$ on $V\ot W$ by
\beq
\Langle v\ot w, v'\ot w' \Rangle = \langle v, v' \rangle_V \langle w, w' \rangle_W, \quad v, v'\in V, \ w, w'\in W. 
\eeq
This is clearly positive definite. Now we show that the form is contravariant with respect to the $A\ot B$-action. 

For $x, y\in  A\ot B$,  $v, v'\in V$ and $w, w'\in W$, 
\[
\baln
&\Langle (x\ot y)(v\ot w), v'\ot w' \Rangle \\
&=\omega(d(y), d(v))^{-1} \Langle xv\ot y w, v'\ot w' \Rangle \\
&=\omega(d(y), d(v))^{-1} \langle xv, v' \rangle_V \langle  y w,  w' \rangle_W\\
&=\omega(d(y), d(v))^{-1} \langle v, x^+ v' \rangle_V \langle  w, y^+ w' \rangle_W\\
&=\omega(d(y), d(v))^{-1} \Langle v\ot w, x^+ v'\ot y^+ w' \Rangle.
\ealn
\]
Note that $\omega(d(y), d(v))^{-1} \Langle v\ot w, x^+ v'\ot y^+ w' \Rangle$ can be rewritten as 
\[
\baln
&\omega(d(y), d(v))^{-1}  \omega(d(y), d(v')) \Langle v\ot w,( x^+\ot y^+) (v'\ot  w') \Rangle\\
&= \omega(d(y), d(x))  \Langle v\ot w,( x^+\ot y^+) (v'\ot  w') \Rangle\\
&=  \Langle v\ot w,( x\ot y)^+(v'\ot  w') \Rangle. 
\ealn
\]
Hence 
$
\Langle (x\ot y)(v\ot w), v'\ot w' \Rangle =  \Langle v\ot w,( x\ot y)^+(v'\ot  w') \Rangle. 
$

This proves the contravariance of the positive definite sesquilinear $\Langle \ , \  \Rangle$, 
showing that $V\ot W$ is a $\Gamma$-graded unitarisable $A\ot B$-module. 
\end{proof}

\begin{example}\label{eg:Fock}

Any theory of $\ast$-$(\Gamma, \omega)$-algebras and unitarisable modules should account for
the following primary example.
Assume that the commutative factor $\omega$ satisfies the unit modulus property. Then   
the Weyl $(\Gamma, \omega)$-algebra $\CW_\omega$ introduced in Section \ref{sect:Weyl} 
is a $\ast$-$(\Gamma, \omega)$-algebra with the $\ast$-structure 
\[
\ast(x(\alpha)_i^r) = \frac{\partial}{\partial x(\alpha)_ i^r},  \quad \ast\left(\frac{\partial}{\partial x(\alpha)_ i^r}\right)=x(\alpha)_i^r, \quad \forall \alpha, i, r.
\]
The Fock space $\CF_\omega$ is a unitarisable $\CW_\omega$-module. 

\end{example}

\subsubsection{Hopf  $\ast$-$(\Gamma, \omega)$-algebras}
Let us introduce the following conditions on the antipode of a given 
Hopf $(\Gamma, \omega)$-algebra $(A, \mu; \Delta, \varepsilon, S)$.
\begin{assumption} \label{assum} {\ } 
\begin{enumerate}[a)]
\item 
The antipode $S$ is bijective; and
\item there exists a group-like element $K\in A$ such that  
\beq
S^2(x)= K^2 x K^{-2}, \quad \forall x \in A. 
\eeq
\end{enumerate}
[The element $K$ being group-like means that  $\Delta(K)=K\ot K$ and $\varepsilon(K)=1$.]
\end{assumption}

\begin{definition}
Let $(A, \mu; \Delta, \varepsilon, S)$ be a Hopf $(\Gamma, \omega)$-algebra, which satisfies the 
Antipode Conditions \ref{assum}. Assume that 
$(A, \mu)$ is an associative $\ast$-$(\Gamma, \omega)$-algebra with the $\ast$-structure $\ast: A\lra A$.
Call $(A, \mu; \Delta, \varepsilon, S, \ast)$ a Hopf $\ast$-$(\Gamma, \omega)$-algebra 
if  $\Delta$ and $\varepsilon$ are 
$\ast$-$(\Gamma, \omega)$-algebra homomorphisms, that is, 
\beq\label{eq:Hopf-ast}
\ast\circ \Delta = \Delta\circ \ast, \quad \ast\circ \varepsilon = \varepsilon \circ \ast;
\eeq
and $K$ is self-conjugate in the sense that
\beq \label{eq:K-real}
\ast(K)=K.
\eeq
\end{definition}

Note that the first relation of \eqref{eq:Hopf-ast} means that  for all $x\in A$, 
\[
\sum \omega(d(x_{(2)}), d(x_{(1)})) \ast(x_{(1)})\ot \ast(x_{(2)})= \sum \ast(x)_{(1)}\ot \ast(x)_{(2)},
\]
where $\Delta(x)=\sum x_{(1)}\ot x_{(2)}$ and $\Delta(\ast(x)) =\sum \ast(x)_{(1)}\ot \ast(x)_{(2)}$  in Sweedler's  notation. 

\begin{lemma} The following relations hold for any Hopf $\ast$-$(\Gamma, \omega)$-algebra. 
\beq\label{eq:ast-S}
S\circ \ast \circ S\circ \ast = \ast\circ  S\circ \ast\circ  S = \id_A. 
\eeq
\end{lemma}
\begin{proof}
Since $\mu(S\ot\id)\Delta=\id$, we have 
\[
\baln
\sum S(x_{(1)}) x_{(2)}\ot  x_{(3)}=1\ot x, \quad x\in A.  
\ealn
\]
Applying $\ast$ to both sides, we obtain
\[
\baln
& \sum  \omega(d(x_{(3)}^+), d(x_{(1)}^+)+d(x_{(2)}^+)) x_{(2)}^+ S(x_{(1)})^+\ot  x_{(3)}^+=1\ot x^+. 
\ealn
\]
Then by applying $\mu(S\ot\id)$ to both sides of this equation, we obtain  
\[
\baln
&\sum  \omega(d(x_{(3)}^+), d(x_{(1)}^+)+d(x_{(2)}^+))  \omega(d(x_{(2)}^+), d(x_{(1)}^+))
S(S(x_{(1)})^+) S(x_{(2)}^+) x_{(3)}^+=x^+. 
\ealn
\]
Note that
\[
\baln
& \omega(d(x_{(3)}^+), d(x_{(1)}^+)+d(x_{(2)}^+))  \omega(d(x_{(2)}^+), d(x_{(1)}^+))\\
&= \omega(d(x_{(3)}^+)+d(x_{(2)}^+), d(x_{(1)}^+))  \omega(d(x_{(3)}^+), d(x_{(2)}^+))\\
&=\omega(d(x^+)-d(x_{(1)}^+), d(x_{(1)}^+))  \omega(d(x_{(3)}^+), d(x_{(2)}^+)). 
\ealn
\]
Denote $\omega_{(1)}=\omega(d(x^+)-d(x_{(1)}^+), d(x_{(1)}^+))$. We can re-write the previous equation as 
\beq \label{eq:S-ast-proof}
&\sum  \omega_{(1)}  \omega(d(x_{(3)}^+), d(x_{(2)}^+))S(S(x_{(1)})^+) S(x_{(2)}^+)x_{(3)}^+=x^+. 
\eeq
Let us simplify the left hand side. 
\[
\baln 
\text{LHD  of \eqref{eq:S-ast-proof}}&=\sum  \omega_{(1)}  S(S(x_{(1)})^+) 
\mu(S\ot\id)(x_{(2)}\ot x_{(3)})^+\\ 
&=\sum  \omega_{(1)}  S(S(x_{(1)})^+) 
\mu(S\ot\id)\ast\Delta(x_{(2)})\\ 
&=\sum  \omega_{(1)}  S(S(x_{(1)})^+) 
\mu(S\ot\id)\Delta(x_{(2)}^+) \\ 
&=\sum  \omega_{(1)}  S(S(x_{(1)})^+) 
\varepsilon(x_{(2)}^+)\\ 
&=\sum  \omega(d(x^+)-d(x_{(1)}^+), d(x_{(1)}^+))  S(S(x_{(1)} \varepsilon(x_{(2)}))^+). 
\ealn
\]
In the final expression, we have $d(x_{(1)})+d(x_{(2)}) =d(x)$ and $d(x_{(2)})=0$. Hence 
\[
\omega(d(x^+)-d(x_{(1)}^+), d(x_{(1)}^+))=\omega(d(x^+)-d(x^+), d(x^+))=1.
\]
This leads to 
\[
\baln
\text{LHD  of \eqref{eq:S-ast-proof}}&=\sum S(S(x_{(1)} \varepsilon(x_{(2)}))^+) =S\circ \ast \circ S (x).
\ealn
\]
Therefore, 
\[
S\circ \ast \circ S (x) = \ast(x), \quad x\in A, 
\]
which implies  the lemma. 
\end{proof}

\begin{example}

A $\ast$-structure for the universal enveloping algebra $\U(\fg)$ of a Lie $(\Gamma, \omega)$-algebra $\fg$ is assumed to satisfy the condition
\beq
\ast(\fg)=\fg, 
\eeq
where $\fg$ is regarded  as  embedded in $\U(\fg)$ canonically.  
It is easy to verify that 
\[
[X, Y]^+=[Y^+, X^+], \quad \forall X, Y\in\fg. 
\]

If  $\U(\fg)$ has a $\ast$-structure, it follows Lemma \ref{lem:tensor-ast} that there is a natural $\ast$-structure for $\U(\fg)\ot \U(\fg)$. 
For any $X\in\fg$, we have $X^+\in\fg$, and hence
$
\Delta(X)^+= \Delta(X^+). 
$
Since $\U(\fg)$ is generated by $\fg$, 
\beq
\Delta(u)^+= \Delta(u^+), \quad \forall u\in\U(\fg). 
\eeq
\end{example}

\begin{remark}\label{rmk:ast'}
If $\ast: A\lra A$ is a $\ast$-structure for a Hopf $(\Gamma, \omega)$-algebra $A$, then $\ast': A\lra A$, given  by $\ast'(x)=\omega(d(x), d(x)) \ast(x)$ for $x\in A$,  also defines a $\ast$-structure by noting that $\omega(\alpha, \alpha)=1$ or $-1$ for any $\alpha\in\Gamma$.  
\end{remark}

\subsection{Operations on unitarisable modules}\label{sect:Hopf-mod-u}

We have the following result for any given Hopf $\ast$-$(\Gamma, \omega)$-algebra 
$(A, \mu; \Delta, \varepsilon, S, \ast)$. 
\begin{theorem} \label{thm:uni-s-s}
Tensor products of unitarisable $\Gamma$-graded modules for a Hopf $\ast$-$(\Gamma, \omega)$-algebra $A$ are unitarisable
$\Gamma$-graded $A$-modules. 
\end{theorem}
\begin{proof}
If  $V$ and $W$ are $\Gamma$-graded unitarisable $A$-modules, the tensor product $V\ot W$ is a $\Gamma$-graded unitarisable $A\ot A$-module by Theorem \ref{thm:tensor-prod}. As $\Delta(A)$ is a $\ast$-subalgebra of $A\ot A$, the restriction of $V\ot W$ to $\Delta(A)$ is a $\Gamma$-graded unitarisable module. This proves the theorem.
\end{proof}

Assume that the antipode of the 
Hopf $\ast$-$(\Gamma, \omega)$-algebra $A$ satisfies the Antipode Conditions \ref{assum}.
Given a unitarisable module $V$ with the positive definite $A$-invariant sesquilinear form $\langle \ , \ \rangle$, we denote $V^+:=\{v^+:=\langle v, - \rangle\mid v\in V\}$, which is canonically embedded in the dual space $V^*$ of $V$. Thus we shall regard $V^+$ as a subspace of $V^*$ so that $v^+(w)=\langle v, w \rangle$ for all $w\in V$. Note that $deg(v^+)=-deg(v)$. 

Now for all $w\in V$, 
\[
\baln
(x\cdot v^+)(w) &= \omega(d(x), d(v^+))v^+(S(x)\cdot w) 
= \omega(d(x), d(v^+)) \langle v, S(x)\cdot w \rangle \\
&= \omega(d(x), d(v^+)) \langle \ast S(x)\cdot v,  w \rangle 
= \omega(d(x), d(v^+)) (S(x)^+\cdot v)^+(w). 
\ealn
\]
Hence 
\beq\label{eq:act-dual}
x\cdot v^+ = \omega(d(x), d(v^+)) (S(x)^+\cdot v)^+, \quad  v\in V. 
\eeq

Define a sequilinear form $\langle \ , \ \rangle': V^+\ot V^+\lra \C$ by
\beq\label{eq:form-dual}
\langle v^+ , w^+ \rangle' = \langle K^2\cdot w, v \rangle, \quad \forall v, w\in V. 
\eeq
Clearly this is positive definite as $\langle v^+ , v^+ \rangle'=\langle K\cdot v , K \cdot v\rangle$, and $K\cdot v\ne 0$ if $ v\ne 0$. 

It will become clear in the proof of the lemma below that condition \ref{assum}.(b) on $K^2$ is required by the contravariance of the sequilinear form. 

\begin{lemma}  \label{lem:unit-dual}
Let $A$ be a Hopf $(\Gamma, \omega)$-algebra with $\ast$-structures $\ast, \ast': A\lra A$, 
where $\ast'$ is related to $\ast$ as described in Remark \ref{rmk:ast'}.
Assume that the antipode satisfies  
the Antipode Conditions \ref{assum}. Let  $V$ be a unitarisable $\Gamma$-graded 
$A$-module with respect to $\ast$.  
Then $V^+$ is a unitarisable $\Gamma$-graded $A$-module for $\ast'$ 
with the positive definite contravariant sequilinear form defined by \eqref{eq:form-dual}. 
\end{lemma}
\begin{proof}
The sesquilinear form on $V^+$ is defined by \eqref{eq:form-dual}, which is positive definite. Now 
for any $v^+, w^+\in V^+$ and $x\in A$, 
\[
\baln
\langle x\cdot v^+ , w^+ \rangle' &= \langle\omega(d(x), d(v^+)) (S(x)^+\cdot v)^+, w^+ \rangle' 
	&\quad &\text{by \eqref{eq:act-dual}}\\
&= \omega(d(x), d(v^+))^{-1} \langle K^2 \cdot w, S(x)^+\cdot v\rangle. 
	&\quad &\text{by \eqref{eq:form-dual}}
\ealn
\]
Using properties of the sesquilinear form on $V$ and $S(x)K^2= K^2 S^{-1}(x)$, we obtain 
\[
\baln
 \langle K^2 \cdot w, S(x)^+\cdot v\rangle 
&=  \langle S(x) K^2 \cdot w,  v\rangle 
= \langle  K^2 S^{-1}(x) \cdot w,  v\rangle.
\ealn
\]
Now
\[
\baln
 \langle  K^2 S^{-1}(x) \cdot w,  v\rangle 
&=  \langle v^+,  (S^{-1}(x) \cdot w)^+\rangle' 
	&\quad &\text{by \eqref{eq:form-dual}}\\
&= \langle v^+,  (S(x^+)^+ \cdot w)^+\rangle' 	
	&\quad &\text{by \eqref{eq:ast-S}}\\ 
&=  \omega(d(x), d(w^+))\langle v^+, x^+\cdot w^+\rangle'. 
	&\quad &\text{by \eqref{eq:form-dual}}.
\ealn
\]
Hence 
\[
\langle x\cdot v^+ , w^+ \rangle'  
=  \omega(d(x), d(v^+))^{-1} \omega(d(x), d(w^+))\langle v^+, x^+\cdot w^+\rangle'. 
\]
Note that the right hand side is non-zero only if $d(v)=d(x)+ d(w)$. Under this condition,  
$
\omega(d(x), d(v^+))^{-1} \omega(d(x), d(w^+))
=\omega(d(x), d(x)).
$
Hence 
\[
\langle x\cdot v^+ , w^+ \rangle' =  \langle v^+, \omega(d(x), d(x)) x^+\cdot w^+\rangle' = \langle v^+, \ast'(x)\cdot w^+\rangle'.
\]
This shows that the form \eqref{eq:form-dual} is contravariant with respect to $\ast'$, proving the lemma. 
\end{proof}

\subsection{Unitarisable  $\gl(V(\Gamma, \omega))$-modules for  compact $\ast$-structures}

Throughout this section,  we write
 $V=V(\Gamma, \omega)$, $\gl(V)=\gl(V(\Gamma, \omega))$ and $\U=\U(\gl(V))$. 

\subsubsection{$\ast$-structures of $\U(\gl(V))$}\label{sect:V-power}
As a Hopf $(\Gamma, \omega)$-algebra, $\U$ fulfils the Antipode Conditions \ref{assum} 
with $K=1$, 
since the antipode satisfies $S^2=\id_\U$. 

Let us first consider $\ast$-structures for $\U$ as a Hopf $(\Gamma, \omega)$-algebra. 
We assume that the commutative factor $\omega$ satisfies the ``Unit Modulus Condition''. 

As usual, we choose a distinguished order for $\Gamma_R$. 

We embed $\gl(V)$ in $\U$ in the canonical way, thus the elements $E(\alpha, \beta)_{i j}$ 
are regarded as generators of $\U$. 

\begin{lemma}  Fix $\alpha_c\in \Gamma_R^+$ and $p\in\{1, 2, \dots, m_{\alpha_c}\}$, 
and let 
\[
\chi(\beta, j) = \left\{
\begin{array}{l l}
1, &\forall j, \text{ if } \beta> \alpha_c,  \\
-1, &\forall j, \text{ if } \beta< \alpha_c,  \\
1, & \forall j\ge p, \text{ if } \beta= \alpha_c,  \\
-1, & \forall j< p, \text{ if } \beta= \alpha_c.
\end{array}
\right. 
\]
Then there exist  the following  $\ast$-structures $\ast, \ast': \U(\gl(V))\lra \U(\gl(V))$.
\beq
\ast(E(\alpha, \beta)_{i j}) &=&  \chi(\alpha, i) \chi(\beta, j)E(\beta, \alpha)_{j i}, \quad \forall \alpha, \beta, i, j;  
\label{eq:non-cpc}\\
\ast'(E(\alpha, \beta)_{i j}) &=& \omega(\alpha-\beta, \alpha-\beta) \ast(E(\beta, \alpha)_{j i}),\quad \forall \alpha, \beta, i, j. 
\label{eq:non-cpc-1}
\eeq
\end{lemma}
\begin{proof}
If \eqref{eq:non-cpc} defines a $\ast$-structure, so does also  \eqref{eq:non-cpc-1} by Remark \ref{rmk:ast'}.  

The simplicity of the co-multiplication $\Delta(E(\alpha, \beta)_{i j})=E(\alpha, \beta)_{i j}\ot 1 + 1\ot E(\alpha, \beta)_{i j}$ makes it very easy to check that 
the map $\ast$ satisfies \eqref{eq:Hopf-ast}, if we can show that it is indeed a 
$\ast$-structure for $\U$ as an associative $(\Gamma, \omega)$-algebra. Thus we only need to verify that
the map $\ast$  preserves the commutation relations 
\[
\baln
&E(\alpha, \beta)_{i j} E(\gamma, \delta)_{k \ell}  - \omega(\alpha-\beta, \gamma-\delta) E(\gamma, \delta)_{k \ell} E(\alpha, \beta)_{i j}\\
&=\delta_{\gamma\beta}\delta_{k j}E( \alpha, \delta)_{i \ell } - \omega(\alpha-\beta, \gamma-\delta)\delta_{\alpha \delta}\delta_{i\ell}E(\gamma, \beta)_{k j}
\ealn
\]
for all $\alpha, \beta, \gamma, \delta$ and $i, j, k, \ell.$  
We denote the two sides of the above equation by $\text{LHS}$ and $\text{RHS}$ respectively. Straightforward calculations yield
\[
\baln
\ast(\text{LHS})
&= \chi(\alpha, i)\chi(\beta, j) \chi(\gamma, k)\chi(\delta, \ell)  \\
&\times (E(\delta, \gamma)_{\ell k} E(\beta, \alpha)_{j i} - \omega(\delta-\gamma, \beta-\alpha)  E(\beta, \alpha)_{j i}E(\delta, \gamma)_{\ell  k})\\
&= \chi(\alpha, i)\chi(\beta, j) \chi(\gamma, k)\chi(\delta, \ell) [E(\delta, \gamma)_{\ell k},  E(\beta, \alpha)_{j i}].
\ealn
\]
By using the commutation relations of $\gl(V)$, we can express the last line as 
\[
\baln
&\chi(\alpha, i)\chi(\beta, j) \chi(\gamma, k)\chi(\delta, \ell)\\
&\times \left(\delta_{\gamma\beta}\delta_{k j}E(\delta, \alpha)_{\ell i} - \omega(\delta-\gamma, \beta-\alpha)\delta_{\alpha \delta}\delta_{i\ell}E(\beta, \gamma)_{j k}\right)\\
&= \chi(\alpha, i)\chi(\delta, \ell) \delta_{\gamma\beta}\delta_{k j}E(\delta, \alpha)_{\ell i} \\
&-  \chi(\beta, j) \chi(\gamma, k) \omega(\alpha-\beta, \gamma-\delta)^{-1}\delta_{\alpha \delta}\delta_{i\ell}E(\beta, \gamma)_{j k}\\
&=\ast(\text{RHS}).
\ealn
\]
This proves the lemma.
\end{proof}

If $\alpha_c$ is the minimal element of $\Gamma_R$ and $p=1$, then $\ast, \ast': \U\lra \U$ are given by 
\beq
&&\ast(E(\alpha, \beta)_{i j}) = E(\beta, \alpha)_{j i},\quad \forall i, j, \alpha, \beta,  \label{eq:cpct}\\ 
&&\ast'(E(\alpha, \beta)_{i j}) = \omega(\alpha, \alpha) \omega(\beta, \beta) E(\beta, \alpha)_{j i},\quad \forall i, j, \alpha, \beta. \label{eq:cpct'}
\eeq
Call them the compact $\ast$-structures of type I and type II respectively. 

The reason for the terminology is that 
for the usual general linear Lie algebra, the $\ast$-structures (which coincide) 
correspond to the compact real form.  

\subsubsection{Unitarisable $\gl(V)$-modules for the compact $\ast$-structures}\label{sect:unit-mods}
We consider unitarisable modules with respect to the compact $\ast$-structures.

Let $L_\lambda$ be a simple $\gl(V)$-module with highest weight $\lambda$. 
Let $v^+_\lambda$ be a $\gl(V)$-highest weight vector which generates $L_\lambda$.  
[Highest weight vectors are non-zero.] 
We can define a sesquilinear form on $L_\lambda$
\beq\label{eq:hw-form}
\langle \ , \ \rangle: L_\lambda\times L_\lambda\lra \C, 
\eeq
by setting $\langle v^+_\lambda, v^+_\lambda \rangle=1$ and requiring 
it be contravariant with respective to the compact $\ast$-structure of type I. 
This  defines the form uniquely. Unitarisability of $L_\lambda$ with respective to the compact $\ast$-structure of type I 
means positive definiteness of  the form.

Assume that $L_\lambda$ is unitarisable. 
As $\ast(E(\alpha, \alpha)_{i i})= E(\alpha, \alpha)_{i i}$ for all $\alpha$ and $i$, 
it is evident that the highest weight $\lambda$ must be real, i.e, 
$\lambda$ belongs to $\fh^*_\R= \sum_{i, \alpha}\R \varepsilon(\alpha)_i$. 

Consider any of the subalgebras $\gl(V_{\alpha, i; \beta,j})\simeq \gl_2(\C)$ described in Lemma \ref{lem:subalgs}.  
The highest weight vector $v^+_\lambda$ generates a $\gl(V_{\alpha, i; \beta,j})$-submodule 
$\U(\gl(V_{\alpha, i; \beta,j}))v^+_\lambda\subset L_\lambda$, 
which is unitarisable with respect to the compact real form of  $\gl_2(\C)$. This is true for any $\alpha, \beta\in \Gamma_R^+$ or $\Gamma_R^-$, and any $i, j$,  thus we must have $\lambda\in\Lambda_{M_+|M_-}$. 

The above discussion is also valid for the type II compact $\ast$-structure. 
Thus we have the following necessary condition for unitarisability.  

\begin{lemma} \label{lem:unit-fd}
If a simple $\gl(V)$-module $L_\lambda$ is unitarisable with respect 
to any of the compact $\ast$-structures, then $\lambda\in\Lambda_{M_+|M_-}\cap\fh^*_\R$, 
and hence $\dim L_\lambda<\infty$.  
\end{lemma}

We have the following result. 

\begin{theorem}\label{thm:unit-tensor}
Tensor powers $V^{\ot r}$ (for all $r\in\Z_+$) of the natural $\gl(V)$-module are unitarisable
for the type I compact $\ast$-structure \eqref{eq:cpct}. Also, 
tensor powers ${V^*}^{\ot r}$ (for all $r\in\Z_+$) of the dual module $V^*$ of $V$ are unitarisable
for the type II compact $\ast$-structure \eqref{eq:cpct'}. 
\end{theorem}
 \begin{proof}
The general linear  Lie $(\Gamma, \omega)$-algebra $\gl(V)$ naturally acts on $V$. 
Define a positive definite sesquilinear form $\langle \  , \rangle: V\times V \lra \C$ by
\[
\langle e(\alpha)_i, e(\beta)_j \rangle = \delta_{\alpha \beta} \delta_{i j},
\]
for all $\alpha, \beta\in\Gamma_R$, and $i=1, 2, \dots, m_\alpha$, $j=1, 2, \dots, m_\beta$.  
Then for any $E(\chi, \gamma)_{k \ell}$, 
\[
\baln
\langle E(\chi, \gamma)_{k \ell}e(\alpha)_i, e(\beta)_j \rangle 
&= \delta_{\ell i} \delta_{\gamma \alpha} \langle e(\chi)_k, e(\beta)_j \rangle 
= \delta_{\gamma \alpha} \delta_{\chi \beta} \delta_{k j} \delta_{\ell i}, \\
\langle e(\alpha)_i, E(\chi, \gamma)_{k \ell}^+ e(\beta)_j \rangle 
&= \langle e(\alpha)_i, E(\gamma, \chi)_{ \ell k} e(\beta)_j \rangle\\
&=  \delta_{\chi \beta} \delta_{k j}  \langle e(\alpha)_i, e(\gamma)_\ell\rangle 
= \delta_{\gamma \alpha} \delta_{\chi \beta} \delta_{k j} \delta_{\ell i}.
\ealn
\]
Hence the sesquilinear form $\langle \  , \ \rangle$ is contravariant with respect to the type I compact $\ast$-structure, proving that $V$ is a $\Gamma$-graded unitarisable module.

Each tensor power $V^{\ot r}$, for any $r\in\Z_+$, has a natural $\Gamma$-graded $\gl(V)$-module structure. 
By iterating Theorem \ref{thm:uni-s-s}, we easily show that the tensor powers $V^{\ot r}$ as a $\gl(V)$-modules are unitarisable with respect to the type I compact $\ast$-structure. 
The  positive definite  contravariant sesquilinear form  
$
\Langle \  , \Rangle^{(r)}: V^{\ot r}\times V^{\ot r} \lra \C
$
on $V^{\ot r}$ is given by
\beq
&&\Langle e(\alpha_1)_{i_1}\ot \dots\ot e(\alpha_r)_{i_r},  e(\beta_1)_{j_1}\ot \dots\ot e(\beta_r)_{j_r}\Rangle^{(r)}\\
&&=\prod_{s=1}^r  \langle e(\alpha_s)_{i_s},  e(\beta_s)_{j_s}\rangle =\prod_{s=1}^r \delta_{\alpha_s \beta_s}\delta_{i_s j_s}. \nonumber
\eeq

It follows  Lemma \ref{lem:unit-dual} that 
${V^*}^{\ot r}$, for all $r\in\Z_+$,  are unitarisable
with respect to the type II compact $\ast$-structure defined by \eqref{eq:cpct'}.  
\end{proof}

If $V=V_+$ or $V_-$, then the two compact $\ast$-structures coincide. The lemma implies the following result. 

\begin{corollary} \label{cor:non-super-u}
If $V=V_+$ or $V_-$, then every finite dimensional simple $\gl(V)$-module 
with real highest weight is unitarisable for the compact $\ast$-structure. 
\end{corollary}
\begin{proof}
We have seen in Section \ref{sect:pf-fd} that if $V=V_+$ or $V_-$, every finite dimensional simple $\gl(V)$-module 
is the tensor product of a $1$-dimensional module with a simple tensor module. 
Thus Theorem \ref{thm:unit-tensor} and Lemma \ref{lem:unit-fd}  imply the corollary.
\end{proof}

\begin{remark}
Unitarisability of tensor modules of the general linear Lie superalgebra was proved in \cite{GZ}. 
Theorem \ref{thm:unit-tensor} generalises this fact to general linear Lie $(\Gamma, \omega)$-superalgebras.
\end{remark}

We find it quite amazing that $V^{\ot r}$, for all $r$, 
are untitarisable even when the commutative factor is genuinely complex valued, 
as the $\gl(V)$-action on $V^{\ot r}$ involves the commutative factor in a complicated way. 
For example, if $v_i$ (with $1\le i\le r$) are homogeneous elements of $V$ of degrees $\zeta_i$ respectively,  
then for any  $X\in \gl(V)_\alpha$, 
\[
X(v_1\ot v_2\ot\dots\ot v_r) = \sum_{j=1}^r \omega(\alpha, \zeta^{(j)}) v_1\ot\dots \ot v_{j-1}\ot X(v_j)\ot v_{j+1}\ot\dots \ot v_r, 
\]
where $\zeta^{(j)}= \sum_{i=1}^{j-1}\zeta_i$ for $1\le j\le r$.

\subsubsection{Classification theorem}\label{sect:fd-unit}

\begin{theorem}\label{thm:class-unit}
A simple highest weight $\gl(V)$-module $L_\lambda$ is unitarisable for the type I compact $\ast$-structure 
if and only if 
\begin{enumerate}
\item $\lambda\in\Lambda_{M_+|M_-}\cap \fh^*_\R$; and 
\item one of the following is true  (in the notation of Section \ref{sect:typical}).
\begin{enumerate}
\item either $L_\lambda$ is typical, and 
\beq\label{eq:tp-ast}
(\lambda+\rho, \varepsilon_{M_+}-\varepsilon_{\ol{M_-}})>0; 
\eeq
\item or $L_\lambda$ is atypical, 
and there exists $r\le M_-$ such that 
\beq\label{eq:atp-ast}
(\lambda+\rho, \varepsilon_{M_+}-\varepsilon_{\ol{r}})=0, \quad (\lambda, \varepsilon_{\ol{r}} - \varepsilon_{\ol{M_-}})=0.
\eeq
\end{enumerate}
\end{enumerate}
\end{theorem}

\begin{proof}
Recall from Lemma \eqref{lem:unit-fd} that if a simple $\gl(V)$-module $L_\lambda$ is unitarisable 
for any one of the compact $\ast$-structures, then $\lambda\in\Lambda_{M_+|M_-}\cap \fh^*_\R$. 
In particular, $\dim L_\lambda<\infty$.
Observe the following fact. Since $\lambda\in\Lambda_{M_+|M_-}\cap \fh^*_\R$,  
for $1\le j< M_+$ and $1\le s<M_-$, 
\beq\label{eq:rho-order}
(\lambda+\rho, \varepsilon_{j} - \varepsilon_{\ol{s}}) >(\lambda+\rho, \varepsilon_{M_+} - \varepsilon_{\ol{s}})> (\lambda+\rho, \varepsilon_{M_+} - \varepsilon_{\ol{M_-}}).
\eeq

Assume unitarisability of $L_\lambda$ for the compact $\ast$-structure of type I. 
Then the sesquilinear form \eqref{eq:hw-form} for $L_\lambda$ is contravariant  and  positive definite. 
Consider the following vectors
$
v_s:= \BE_{ \ol{s}, M_+} \BE_{  \ol{s}-1, M_+}\dots \BE_{ \ol{1}, M_+}v^+
$
 for $s=1, 2, \dots, M_-$. 
Note that  $\BE_{ M_+, \ol{t}} v_s=0$ for all $t>s$. Hence 
\[
\baln
\langle v_s, v_s\rangle 
&= \langle v_{s-1},   \BE_{ M_+, \ol{s}} \BE_{ \ol{s}, M_+} v_{s-1}\rangle
=\langle v_{s-1}, (\BE_{M_+, M_+}+ \BE_{\ol{s}\,  \ol{s}}) v_{s-1}\rangle\\
&= ((\lambda, \varepsilon_{M_+} - \varepsilon_{\ol{s}}) -(s-1))  \langle v_{s-1},  v_{s-1}\rangle\\
\ealn
\]
Since $(\rho, \varepsilon_{M_+} - \varepsilon_{\ol{s}})=-s+1$ by \eqref{eq:rho-ir}, we obtain
\beq\label{eq:ind-vs}
\langle v_s, v_s\rangle  = (\lambda+\rho, \varepsilon_{M_+} - \varepsilon_{\ol{s}}) \langle v_{s-1}, v_{s-1}\rangle.
\eeq

If $L_\lambda$ is typical, 
\eqref{eq:ind-vs} requires $(\lambda+\rho, \varepsilon_{M_+} - \varepsilon_{\ol{M_-}})>0$. 
This is a necessary and sufficient condition for $\lambda\in\Lambda_{M_+|M_-}\cap \fh^*_\R$ to be typical because of \eqref{eq:rho-order}.

If $L_\lambda$ is atypical, there must exist an $r$ such that $(\lambda+\rho, \varepsilon_{M_+} - \varepsilon_{\ol{r}})=0$, 
as otherwise $L_\lambda$ would be typical.  Now
$
\langle v_{r-1}, v_{r-1}\rangle = \prod_{s=1}^{r-1}(\lambda+\rho, \varepsilon_{M_+} - \varepsilon_{\ol{s}})>0.
$
Thus for all $s>r$, we have  
\[
\baln
0&\le \langle \BE_{\ol{s}, M_+}v_{r-1}, \BE_{\ol{s}, M_+} v_{r-1}\rangle/\langle v_{r-1}, v_{r-1}\rangle  \\
&= (\lambda, \varepsilon_{M_+} - \varepsilon_{\ol{s}}) -r+1\\
&= (\lambda+\rho, \varepsilon_{M_+} - \varepsilon_{\ol{r}}) + (\lambda, \varepsilon_{\ol{r}}-\varepsilon_{\ol{s}})\\
&=(\lambda, \varepsilon_{\ol{r}}-\varepsilon_{\ol{s}}). 
\ealn
\]
But $(\lambda, \varepsilon_{\ol{r}}-\varepsilon_{\ol{s}})\le 0$ since $\lambda\in\Lambda_{M_+|M_-}$. Hence 
$(\lambda, \varepsilon_{\ol{r}}-\varepsilon_{\ol{s}})=0$ for all $s\le r$, which is equivalent to  
$(\lambda, \varepsilon_{\ol{r}} - \varepsilon_{\ol{M_-}})=0$ for $\lambda\in\Lambda_{M_+|M_-}$. 

This proves the necessity of the conditions. 

We now prove sufficiency. 
Denote $\SE_+=\sum_{1\le i\le M_+} \varepsilon_i$, $\SE_-=\sum_{1\le s\le M_-} \varepsilon_ {\ol{s}}$, 
and $\SE=\SE_+-\SE_-$.  For $\lambda\in \Lambda_{M_+|M_-}\cap \fh^*_\R$, let 
\beq\label{eq:trans-wt}
{\mathring\lambda}= \lambda + t\SE, \quad 
t=\left\{
\begin{array}{l l} 
\lambda_{\ol{M_-}}, &\text{if \eqref{eq:tp-ast} is satisfied}, \\
\lambda_{\ol{r}},  &\text{if \eqref{eq:atp-ast} is satisfied},
\end{array}
\right.
 \eeq
 and write ${\mathring\lambda}=({\mathring\lambda}_1,  \dots, {\mathring\lambda}_{M_+}, 
{\mathring\lambda}_{\ol{1}}, \dots, {\mathring\lambda}_{\ol{M_-}})$. 
Then $L_\lambda=L_{t\SE}\ot L_{\mathring\lambda}$. 
Since $L_{t\SE}$ is $1$-dimensional with real weight $t\SE$, it is unitarisable. 
Thus we only need to show that $L_{\mathring\lambda}$ is untarisable. 

Let us first treat case (b), which is much easier. 

\noindent\underline{Case (b)}: In this case,  \eqref{eq:atp-ast} is satisfied,  and we have
\[
\baln
&{\mathring\lambda}_i=\lambda_i+ \lambda_{\ol{r}},  \quad 1\le i\le M_+, \\ 
&{\mathring\lambda}_{\ol{s}}=\lambda_{\ol{s}}-\lambda_{\ol{r}}, \quad1\le  s\le M_-,   \ \text{with} \\
&{\mathring\lambda}_{M_+}= r-1, \quad {\mathring\lambda}_{\ol{s}}=0, \ s\ge r.
\ealn
\]
Hence there exists some partition $\mu\in P_{M_+|M_-}$ such that ${\mathring\lambda}=\mu^\sharp$. 
Thus $L_{{\mathring\lambda}}$ is a simple tensor module, which  is  unitarisable for the type I compact $\ast$-structure by Theorem \ref{thm:unit-tensor}.

\noindent\underline{Case  (a)}: In this case, \eqref{eq:tp-ast} is satisfied, and we have 
\[
\baln
{\mathring\lambda}_i&=\lambda_i +   \lambda_{\ol{M_-}}, \  i=1, 2, \dots, M_+, \\
{\mathring\lambda}_{\ol{r}}&=\lambda_{\ol{r}} -\lambda_{\ol{M_-}}, \ r=1, 2, \dots, M_-.
\ealn
\]
Clearly \eqref{eq:tp-ast} is valid for ${\mathring\lambda}$, i.e.,
\beq\label{eq:rlam-tp}
({\mathring\lambda}+\rho, \varepsilon_{M_+}-\varepsilon_{\ol{M_-}})>0. 
\eeq

Note that ${\mathring\lambda}_{\ol{r}}\in\Z_+$ for all $r$, and  ${\mathring\lambda}_{i}$, for $i\le M_+$, are either all integers or all non-integers.  There exists $a\in [0, 1)$ such that 
\[
{\mathring\lambda}_i = \lceil{\mathring\lambda}_i\rceil - a, \quad {\mathring\lambda}_{\ol{r}} = \lceil{\mathring\lambda}_{\ol{r}}\rceil, \quad \forall i, r, 
\]
where $\lceil{x}\rceil $ is the ceiling function. 
Denote by 
$\lceil{\mathring\lambda}\rceil$ the weight with components $\lceil{\mathring\lambda}_i\rceil$ 
and ${\mathring\lambda}_{\ol{r}}$ for $i=1, \dots M_+$ and $r=1, \dots, M_-$. 
Then $\mathring\lambda=\lceil{\mathring\lambda}\rceil- a \SE_+$.
Clearly  $\lceil{\mathring\lambda}_{\ol{M_-}}\rceil=0$ 
and $\lceil{\mathring\lambda}_{M_+}\rceil\ge M_-$ by  \eqref{eq:tp-ast}. Thus 
there exists $\mu\in P_{M_+|M_-}$ 
such that ${\mathring\lambda}=\mu^\sharp$. 
Hence $L_{\lceil{\mathring\lambda}\rceil}$ is unitarisable for the compact $\ast$-structure of type I. 

If $a=0$, then $L_{\mathring\lambda}=L_{\lceil{\mathring\lambda}\rceil}$, which
is already shown to be unitarisable. 
If $0< a< 1$, the unitarisability of $L_{\mathring\lambda}$ follows from Lemma \ref{lem:key-tp} below
(which is a more general result implying unitarisability of $L_{\mathring\lambda}$  for all $a\in [0, 1)$). 

This proves the theorem,  given Lemma \ref{lem:key-tp}. 
\end{proof}

\begin{remark}
The proof of the typical case of Theorem \ref{thm:class-unit} shows that 
any simple tensor module $L_{\mu^\sharp}$,  with $\mu$ belonging to the set
\beq\label{eq:flex}
P^{(flex)}_{M_+|M_-}=\{\mu\in P_{M+|M_-}\mid \mu_{M_+}\ge M_- > \mu_{M_++1}\},  
\eeq
can be deformed while maintaining unitarisability. 
\end{remark}

Part (1) of the following result can be extracted from the proof of Theorem \ref{thm:class-unit}.
It gives a more explicit description of the simple unitarisable $\gl(V)$-modules for the compact $\ast$-structures. 
\begin{scholium}\label{sch:unit}
Let $\Lambda_{M_+|M_-}^{(u)}=\Lambda_{M_+|M_-}^{(u, 0)}\cup \Lambda_{M_+|M_-}^{(u, 1)}$ with
\[
\baln
\Lambda_{M_+|M_-}^{(u, 0)}&=\{a\SE + \mu^\sharp \mid a\in\R, \mu\in P_{M+|M_-}\},  \\
\Lambda_{M_+|M_-}^{(u, 1)}&=\{a\SE   + \mu^\sharp -b \SE_+ \mid a\in\R, b\in(0, 1), \mu\in P^{(flex)}_{M_+|M_-}\},
\ealn
\]
where the set $P^{(flex)}_{M_+|M_-}$ is defined by \eqref{eq:flex}. 
A simple $\gl(V)$-module $L_\lambda$ with  highest weight $\lambda$ is unitarisable 
\begin{enumerate}
\item for the compact $\ast$-structure of type I if and only if $\lambda\in \Lambda_{M_+|M_-}^{(u)}$; 
\item for the compact $\ast$-structure of  type II if and only if $L_\lambda=L_\mu^*$ with $\mu\in \Lambda_{M_+|M_-}^{(u)}$. 
\end{enumerate}
\end{scholium}
Note that part (2) follows from part (1) and 
the general fact proved in Lemma \ref{lem:unit-dual} about duals of unitarisable modules.

\subsection{Completing the proof of the classification theorem}
We prove Lemma \ref{lem:key-tp},  which implies Theorem \ref{thm:class-unit} for typical $L_\lambda$.
All results in this section, including their proofs,  are generalisations to the present context of those on 
the general linear Lie superalgebra proved in \cite[\S III]{GZ}. 

We maintain notation of the previous section.

\subsubsection{A unitarisability criterion in terms of the Casimir operator}
Recall the basis elements $E(\alpha, \beta)_{i j}$ of $\gl(V)$ introduced in Section \ref{sect:roots}. We regard $\gl(V)$
as embedded in its universal enveloping algebra $\U(\gl(V))$, and hence $E(\alpha, \beta)_{i j}\in \U(\gl(V))$.  
\begin{lemma}
The following element, called the quadratic Casimir operator,   
\beq
\Omega = \sum_{\alpha, \beta\in\Gamma_R}\sum_{i=1}^{m_\alpha} \sum_{j=1}^{m_\beta}  \omega(\beta, \beta) 
E(\alpha, \beta)_{i j}  E(\beta, \alpha)_{j i}, 
\eeq
belongs to the centre of $\U(\gl(V))$, i.e., 
$x\Omega=\Omega x$ for all $x\in \U(\gl(V))$. 
It acts on the simple $\gl(V)$-module with highest weight $\lambda$ by multiplication by the scalar $ (\lambda+2\rho, \lambda)$. 
\end{lemma}

We can easily verify that 
$E(\alpha, \beta)_{i j} \Omega=\Omega  E(\alpha, \beta)_{i j}$ for all $\alpha, \beta,  i, j$,  by direct calculations.
If $v^+_\lambda\in L_\lambda$ is a highest weight vector, 
the usual computation in the Lie algebra case generalises to  the present case, yielding
$
\Omega v^+_\lambda = (\lambda+2\rho, \lambda) v^+_\lambda. 
$

Let 
\[
\baln
\Omega_0^\pm &=\pm \sum_{\alpha, \beta\in\Gamma_R^\pm}\sum_{i=1}^{m_\alpha} \sum_{j=1}^{m_\beta}  
E(\alpha, \beta)_{i j}  E(\beta, \alpha)_{j i}, \quad \\ 
\Omega_1^+ &=  \sum_{\alpha\in\Gamma_R^-}\sum_{\beta\in\Gamma_R^+}\sum_{i=1}^{m_\alpha} \sum_{j=1}^{m_\beta}  
E(\alpha, \beta)_{i j}  E(\beta, \alpha)_{j i}, \\
\Omega_1^- &=  \sum_{\alpha\in\Gamma_R^+} \sum_{\beta\in\Gamma_R^-}\sum_{i=1}^{m_\alpha} \sum_{j=1}^{m_\beta}  
E(\alpha, \beta)_{i j}  E(\beta, \alpha)_{j i},  
\ealn
\]
which all commute with the subalgebra $\gl(V_+)+\gl(V_-)$.  In fact, $\pm \Omega_0^\pm$ are the quadratic Casimir operators of 
$\gl(V_\pm)$ respectively. 
We have 
\[
\Omega_1^+ + \Omega_1^-=2\hat\rho_1, \quad  
\hat\rho_1= \frac{1}{2}\sum_{\alpha\in\Gamma_R^+} \sum_{\beta\in\Gamma_R^-} 
\sum_{i=1}^{m_\alpha} \sum_{j=1}^{m_\beta}(E(\alpha, \alpha)_{i i} + E(\beta, \beta)_{j j}), 
\]
where $\hat\rho_1$ satisfies $\mu(\hat\rho_1)=(\mu, \rho_1)$ for all $\mu\in\fh^*$. We have 
\[
\baln
\Omega&=\Omega_0^+ + \Omega_0^- + \Omega_1^+ - \Omega_1^-
= \Omega_0^+ + \Omega_0^- - 2\hat\rho_1 + 2 \Omega_1^+. 
\ealn
\]

%
%
%
Now we develop a criterion for unitarisability using the Casimir operator. 

A weight $\mu$ is called a $\fk$-highest weight 
of $L_\lambda$ if there exists a $\fk$-highest weight vector $v\in L_\lambda$ of weight $\mu$. 
[Highest weight vectors are non-zero.]
Denote by $\Pi^+(\lambda)$ the set of $\fk$-highest weights of $L_\lambda$.

The following result is a direct generalisation of \cite[Proposition 3]{GZ}.  
\begin{lemma}\label{lem:casimir}
A simple $\gl(V)$-module $L_\lambda$ 
is unitarisable for the type I compact $\ast$-structure
if and only if  $\lambda\in\Lambda_{M_+|M_-}\cap\fh^*_\R$ and
$
(\lambda+2\rho, \lambda)  - (\mu+2\rho, \mu)>0 
$
for all  $\mu\in\Pi^+(\lambda)\backslash\{\lambda\}$. 
\end{lemma}
\begin{proof}
We will need the fact that 
the  finite dimensional simple $\gl(V)$-module $L_\lambda$ restricts to a semi-simple module for $\fk=\gl(V_+)+\gl(V_-)$. This can be seen by noting that $L_\lambda$ as $\fk$-module is a quotient module of $V(L_\lambda^0)=\U(\ol{\fv})\ot L_\lambda^0$, where $L_\lambda^0=\{v\in L_\lambda\mid \fv\cdot v=\{0\}\}$, which is the degree $0$ subspaces of $L_\lambda$ with respect to the $\Z$-grading discussed in Section
\ref{sect:modules}.   Now we can always find a $1$-dimensional $\fk$-module such that its tenor product with $V(L_\lambda^0)$ is a submodule of some direct sum of tensor power of $V_+\oplus V_-$. 
The tensor powers are semi-simple $\fk$-modules by Corollary \ref{cor:non-super-u}, 
hence the restriction of $L_\lambda$ to a module for $\fk$ is semi-simple.

Let $v^+_\lambda$ be a $\gl(V)$-highest weight vector of  $L_\lambda$. 
We can define a sesquilinear form $\langle \ , \ \rangle: L_\lambda\times L_\lambda\lra \C$ by requiring   
it be contravariant with respective to the $\ast$-structure \eqref{eq:cpct}, 
and $\langle v^+_\lambda, v^+_\lambda \rangle=1$. 
This uniquely defines the form. Unitarisability of $L_\lambda$ means positive definiteness of  the form. 

Given a $\fk$-highest weight vector $v\in L_\lambda$ of weight $\mu\ne \lambda$, 
\[
\baln
 \Omega v &= (\Omega_0^+ + \Omega_0^- - 2\hat\rho_1 + 2 \Omega_1^+)v \\
 &=( (\mu+2\rho_0, \mu) - 2(\rho_1, \mu))v + 2 \Omega_1^+v \\
 &= (\mu+2\rho, \mu) v + 2 \Omega_1^+v. 
 \ealn
\]
On the other hand $\Omega v =(\lambda+2\rho, \lambda) v$. Hence 
\beq
 \Omega_1^+v = \frac{1}{2}( (\lambda+2\rho, \lambda) - (\mu+2\rho, \mu)) v. 
\eeq
This leads to 
\beq
&&\frac{1}{2} ( (\lambda+2\rho, \lambda) - (\mu+2\rho, \mu))  \langle v, v \rangle \label{eq:crt} \\
&&= \sum_{\alpha\in\Gamma_R^-}\sum_{\beta\in\Gamma_R^+}\sum_{i=1}^{m_\alpha} \sum_{j=1}^{m_\beta}  
 \langle E(\beta, \alpha)_{j i}v,   E(\beta, \alpha)_{j i}v \rangle. \nonumber
\eeq
Since $v\ne 0$ and $\mu\ne \lambda$, some of the vectors $E(\beta, \alpha)_{j i}v$ are non-zero. 
If $L_\lambda$ is unitarisable, then the right hand side of \eqref{eq:crt} is positive. Hence  
\beq\label{eq:crt-1}
(\lambda+2\rho, \lambda) - (\mu+2\rho, \mu)  >0, \quad \mu\in \Pi^+(\lambda)\backslash\{\lambda\}.
\eeq

If \eqref{eq:crt-1} is satisfied (and $\lambda\in\fh^*_\R$ is assumed), we can prove 
unitarisability of $L_\lambda$ by induction of the degree of the homogeneous subspaces of 
$
L_\lambda=\sum_{k=0}^r (L_\lambda)_{-k}
$ 
in the $\Z$-grading, where $r$ is the largest integer such that $(L_\lambda)_{-r}\ne 0$. 
Note that $\langle (L_\lambda)_{-k},  (L_\lambda)_{-\ell} \rangle=0$ if $k\ne \ell$. 

The defining requirement  $\langle v^+_\lambda, v^+_\lambda \rangle=1$ 
implies that $(L_\lambda)_0= L_\lambda^0$ is unitarisable as $\fk$-module by Corollary \ref{cor:non-super-u}. 
Assume that  the form $\langle\ , \  \rangle$ is positive definite on all $(L_\lambda)_{0}$ , $(L_\lambda)_{-1}$ , \dots, $(L_\lambda)_{-k+1}$, where $k\le r$. Then given any $\fk$-highest weight vector in $v\in (L_\lambda)_{-k}$, some 
$E(\beta, \alpha)_{j i}v$ with $\alpha\in\Gamma_R^+$ and $\beta\in\Gamma_R^-$, are non-zero vectors in  $(L_\lambda)_{-k+1}$, hence $\langle E(\beta, \alpha)_{j i}v,   E(\beta, \alpha)_{j i}v \rangle>0$.  Thus it follows \eqref{eq:crt} that $\langle v, v \rangle>0$.  
Since $(L_\lambda)_{-k}$ is a finite dimensional semi-simple $\fk$-module,  it follows Corollary \ref{cor:non-super-u}
that 
$\langle v, v \rangle>0$ for all $\fk$-highest weight vectors $v\in (L_\lambda)_{-k}$ implies positivity of 
$\langle\ , \  \rangle$ on $(L_\lambda)_{-k}$. This completes the induction step, 
proving the unitarisability of $L_\lambda$. 
\end{proof}

\subsubsection{Unitarisable typical modules}
The following result is a generalisation of \cite[Lemma 1 (ii)]{GZ} to the present context, 
whose proof loc. cit. remains valid up to some minor modifications. 
\begin{lemma}\label{lem:wt-crit}
Assume that $\lambda\in\Lambda_{M_+|M_-}\cap\fh^*_\R$ satisfies the condition \eqref{eq:tp-ast}, i.e., 
$(\lambda+\rho, \varepsilon_{M_+}-\varepsilon_{\ol{M_-}})>0$. Then 
for any weight $\nu$ of $L_\lambda$, 
\[
(\nu, \Upsilon)>0, \quad \forall \Upsilon\in\Phi^+_1. 
\] 
\end{lemma}
\begin{proof} 
For convenience of the reader, we give the proof adapted from \cite{GZ}. 
Let $\mu\in\Pi^+(\lambda)$,  a $\fk$-highest weight  of $L_\lambda$. There is $\Theta\subset \Phi^+_1$ such that 
$\mu=\lambda-2\rho_1(\Theta)$ where $2\rho_1(\Theta)=\sum_{\varepsilon_i - \varepsilon_{\ol{r}}\in \Theta} (\varepsilon_i - \varepsilon_{\ol{r}})$.  Note that
\[
\baln
2(\rho_1(\Theta), \varepsilon_{M_+} - \varepsilon_{\ol{M_-}}) 
&= \sum_{\varepsilon_{M_+} - \varepsilon_{\ol{r}}\in \theta}
(\varepsilon_{M_+} - \varepsilon_{\ol{r}}, \varepsilon_{M_+} - \varepsilon_{\ol{M_-}})\\
&\le M_--1 = -(\rho, \varepsilon_{M_+} - \varepsilon_{\ol{M_-}}).
\ealn
\]
Hence $(\mu, \varepsilon_{M_+} - \varepsilon_{\ol{M_-}})\ge (\lambda+\rho, \varepsilon_{M_+} - \varepsilon_{\ol{M_-}})>0$. 

Let  $\nu$ be an arbitrary weight of $L_\lambda$. There must exist a $\mu\in \Pi^+(\lambda)$ such that 
$\nu=\mu - \sum_{i=1}^{M_+-1}p_i(\varepsilon_i - \varepsilon_{i+1}) - \sum_{r=1}^{M_-1}q_r(\varepsilon_{\ol{r}}-\varepsilon_{\ol{r+1}})$ for some $p_i, q_r\in\Z_+$.  We have 
\[
\baln
(\nu, \varepsilon_{M_+} - \varepsilon_{\ol{M_-}})&= (\mu, \varepsilon_{M_+}-\varepsilon_{\ol{M_-}}) +
p_{M_+}  + q_{\ol{M_-}-1} >0. 
\ealn
\]
This proves that any weight $\nu$ of $L_\lambda$ satisfies $(\nu, \varepsilon_{M_+} - \varepsilon_{\ol{M_-}})>0$.

For any $i, r$, there is an element $\sigma$ of the Weyl group $W$ of $\gl(V(\Gamma, \omega))$ (see Definition \ref{def:Weyl}) 
such that  $\varepsilon_i - \varepsilon_{\ol{r}}=\sigma(\varepsilon_{M_+} - \varepsilon_{\ol{M_-}})$.  
Therefore, 
$
(\nu, \varepsilon_i - \varepsilon_{\ol{r}}) =(\sigma^{-1}(\nu), \varepsilon_i - \varepsilon_{\ol{r}}) >0,
$
completing the proof. 
\end{proof}

%

By applying Lemmas \ref{lem:casimir} and \ref{lem:wt-crit}, 
we can prove the sufficiency of the condition \eqref{eq:tp-ast}.

\begin{lemma}\label{lem:key-tp}
If $\lambda\in\Lambda_{M_+|M_-}\cap\fh^*_\R$ satisfies 
$(\lambda+\rho, \varepsilon_{M_+}-\varepsilon_{\ol{M_-}})>0$,  the
simple $\gl(V)$-module $L_\lambda$ 
is unitarisable for the compact $\ast$-structure of type I.
\end{lemma}
\begin{proof}
Recall the $\Z$-grading of $L_\lambda=\sum_{k=0}^{M_+M_-} (L_\lambda)_{-k}$ (see \eqref{eq:Z-grad}). 
 If $\nu\in \Pi^+(\lambda)$
is a $\fk$-highest weigh of $(L_\lambda)_{-k-1}$, we say that it is at level $k+1$. 
Then there exist a  $\fk$-highest weigh $\mu\in \Pi^+(\lambda)$ of $(L_\lambda)_{-k}$ (at level $k$),  and $\varepsilon_i- \varepsilon_{\ol{r}}\in \Phi^+_1$, such that $\nu=\mu-(\varepsilon_i- \varepsilon_{\ol{r}})$. 
We use induction on $k$ to show that $L_\lambda$ satisfies the unitarizability criterion Lemma \ref{lem:casimir}, i.e., 
for all $\nu\in \Pi^+(\lambda)\backslash\{\lambda\}$,
\beq\label{eq:criterion}
(\lambda +\nu +2\rho, \lambda-\nu) >0.
\eeq

If $k=0$, then $\mu=\lambda$.  We have 
\[
(\lambda +\nu +2\rho, \lambda-\nu) 
= (2\lambda + 2\rho- (\varepsilon_i- \varepsilon_{\ol{r}}), \varepsilon_i- \varepsilon_{\ol{r}})
=2(\lambda + \rho,  \varepsilon_i- \varepsilon_{\ol{r}}). 
\]
Note that for any $\zeta\in \Pi^+(\lambda)$, we have 
\[
\baln
(\zeta+\rho,  \varepsilon_i- \varepsilon_{\ol{r}})&= 
(\zeta+\rho,  \varepsilon_i- \varepsilon_{M_+}) +  (\zeta+\rho,  \varepsilon_{M_+}- \varepsilon_{\ol{M_-}})   - (\zeta+\rho,  \varepsilon_{\ol{r}}- \varepsilon_{\ol{M_-}})\\
&\ge (\zeta+\rho,  \varepsilon_{M_+}- \varepsilon_{\ol{M_-}}), \quad \forall i, r. 
\ealn
\]
Hence $(\lambda + \rho,  \varepsilon_i- \varepsilon_{\ol{r}})>(\lambda + \rho, \varepsilon_{M_+}- \varepsilon_{\ol{M_-}})>0$, 
and we conclude that \eqref{eq:criterion} holds at $k=0$.  

For $\nu$ at arbitrary level $k+1\ge 2$, we have 
\[
\baln
(\lambda +\nu +2\rho, \lambda-\nu) 
&= (\lambda +\mu+ 2\rho- (\varepsilon_i- \varepsilon_{\ol{r}}), \lambda-\mu+ \varepsilon_i- \varepsilon_{\ol{r}})\\
&=(\lambda +\mu+2 \rho,  \lambda-\mu) + 2(\mu, \varepsilon_i- \varepsilon_{\ol{r}}). 
\ealn
\]
Since $\mu$ is at level $k$, by induction hypothesis, $(\lambda +\mu+2 \rho,  \lambda-\mu)>0$.  Also it follows 
Lemma \ref{lem:wt-crit} that $(\mu, \varepsilon_i- \varepsilon_{\ol{r}})>0$. Hence \eqref{eq:criterion} is satisfied by 
all $\nu\in \Pi^+(\lambda)\backslash\{\lambda\}$. 

This proves the lemma. 
\end{proof}



\section{The coordinate algebra}

We return to the setting of arbitrary $(\Gamma, \omega)$-algebras, 
where no conditions will be imposed on the commutative factor. 

We investigate a graded commutative Hopf $(\Gamma, \omega)$-algebra, 
which may be regarded as the ``coordinate algebra'' 
of some ``general linear supergroup" in the $\Gamma$-graded setting. 
As we will briefly allude to, there indeed exists a group functor associated 
with this graded commutative Hopf $(\Gamma, \omega)$-algebra.

The general method of this section is similar to that adopted in studying
the ``coordinate algebras'' 
for the general linear Lie superalgebra \cite{SchZ02}  
and quantum general linear supergroup \cite{Z98}.

\subsection{The coordinate algebra as a graded commutative Hopf $(\Gamma, \omega)$-algebra}\label{sect:coord-alg}

\subsubsection{The coordinate algebra} 
We assume that $V=V(\Gamma, \omega)$ is finite dimensional. 
Maintain the notation $\gl(V)=\gl(V(\Gamma, \omega))$, 
and $\U=\U(\gl(V))$. 
 Denote by $\U^0$ the finite dual of $\U$.  
Since $\U$ is a co-commutative Hopf $(\Gamma, \omega)$-algebra, 
$\U^0$ is a graded commutative Hopf $(\Gamma, \omega)$-algebra.

\begin{remark}
We will abuse notation to also denote the co-multiplication, co-unit and antipode of $\U^0$ by $\Delta$, 
$\varepsilon$ and $S$ respectively. This should not cause confusion with the structure maps of $\U$ in general, but if confusion could potentially arise in particular situations,
we will add in explanations of notation. 
\end{remark}

We also use notation from Section \ref{sect:roots}. In particular, 
$V=\sum_{\alpha\in \Gamma}V_\alpha$, and $m_\alpha=\dim V_\alpha \ne 0$ 
if and only if $\alpha$ belongs to the subset $\Gamma_R$ of $\Gamma$. 
We fix an arbitrary total order for $\Gamma_R$, and take an ordered basis 
$B_\alpha=\{e(\alpha)_i\mid 1\le i\le m_\alpha=\dim V_\alpha\}$ for each $V_\alpha$. 
Then $B=\bigcup_{\alpha\in\Gamma_R} B_\alpha$,  ordered according to the total order chosen for $\Gamma_R$,  is an ordered basis for $V$.  
Let $\ol{B}=
\{\ol{e}(-\gamma)_i \mid 1\le i\le m_\gamma, \gamma \in\Gamma_R\}$ 
be the dual basis for the dual space $V^*$, thus 
$\ol{e}(-\gamma)_i(e(\alpha)_j)=\delta_{\gamma \alpha}\delta_{i j}$.

Since $V$ and $V^*$ are simple left $\U$-modules, they are simple right $\U^0$-co-modules, 
whose structure maps will be written as 
\[
\baln
\delta_V: V\lra V\ot \U^0, \quad & \delta_{V^*}: V^*\lra V^*\ot \U^0.
\ealn
\]
We denote by $t(\beta, \alpha)_{j i}$ and $\ol{t}(-\beta, -\alpha)_{j i}$ respectively
 the matrix coefficients of the $\U^0$-co-modules $V$ and $V^*$, so that for all $e(\alpha)_i$ and $\ol{e}(-\alpha)_i$, 
\beq
&& \delta_V(e(\alpha)_i)=\sum_{\beta, j} e(\beta)_j\ot t(\beta, \alpha)_{j i},  \\ 
&& \delta_{V^*}(\ol{e}(-\alpha)_i)=\sum_{\beta, j} \ol{e}(-\beta)_j\ot \ol{t}(-\beta, -\alpha)_{j i}. 
\eeq
They are homogeneous  elements of $\U^0$, whose degrees are given by $d(t(\beta, \alpha)_{j i})=\alpha-\beta$
and $d(\ol{t}(-\beta, -\alpha)_{j i})=\beta-\alpha$ respectively. 

\begin{definition} Denote by $\ST(V(\Gamma, \omega))$ the $(\Gamma, \omega)$-subalgebra of $\U^0$ generated by the elements 
$t(\beta, \alpha)_{j i}, \ \ol{t}(-\beta, -\alpha)_{j i}$ for all $\alpha, \beta\in \Gamma_R$, $i=1, 2, \dots, m_\alpha$, $j=1, 2, \dots, m_{\beta}$. 
\end{definition}

We write $\ST(V)=\ST(V(\Gamma, \omega))$. 

Note that $\ST(V)$ is graded commutative since $\U^0$ is.

Consider the right $\U^0$-co-module $\delta_{V\ot V^*}: V\ot V^*\lra V\ot V^*\ot\U^0$.
We have 
\beq
\delta_{V\ot V^*}(\iota^{-1}(\id_V))= \iota^{-1}(\id_V)\ot 1, \label{eq:G-4}
\eeq
since $\iota^{-1}(\id_V)$ is a $\U$-invariant, 
hence a $\U^0$-invariant. 
This  implies the following result. 
\begin{lemma} \label{rem:T-def} 
The following relation holds in $\ST(V)$ for all $\beta, \gamma, j, k$.
\beq
\sum_{\alpha\in\Gamma_R}\sum_{i=1}^{m_\alpha} \omega(\gamma, \alpha-\beta)t(\beta, \alpha)_{j i} \ol{t}(-\gamma, -\alpha)_{k i} =\delta_{\beta \gamma}\delta_{j k}. \label{eq:G-a-4} 
\eeq
\end{lemma}
\begin{proof}
To prove this,  note that
\beq\label{eq:can-inv}
C:=\iota^{-1}(\id_V)=\sum_{\alpha\in\Gamma_R}\sum_{i=1}^{m_\alpha} e(\alpha)_i\ot \ol{e}(-\alpha)_i,  
\eeq
and we can also easily show that 
\[
\delta_{V\ot V^*}(C)=\sum_{\beta, \gamma, j, k} e(\beta)_j\ot \ol{e}(-\gamma)_k    \ot  \sum_{\alpha\in\Gamma_R} \sum_{i=1}^{m_\alpha}\omega(\alpha-\beta, -\gamma) t(\beta, \alpha)_{j i} \ol{t}(-\gamma, -\alpha)_{k i}. 
\]
We immediately obtain  \eqref{eq:G-a-4} by using equation \eqref{eq:G-4}. 
\end{proof}

\begin{remark}\label{rmk:trace}
It follows \eqref{eq:can-inv} that  $\tau_{V, V^*}(\iota^{-1}(\id_V))=\sum_{i, \alpha} \omega(\alpha, \alpha)\ol{e}(-\alpha)_i\ot  e(\alpha)_i$. 
The $\omega$-trace of any $\phi\in\End_\C(V)$ may be described more conceptually as 
\[
\baln
{\rm tr}_{(\Gamma, \omega)}(\phi) 
&:= \sum_{\alpha\in\Gamma_R}\sum_{i=1}^{m_\alpha}  \omega(\alpha-d(\phi), \alpha)\langle \ol{e}(-\alpha)_i, \phi\cdot  e(\alpha)_i\rangle, 
\ealn
\]
and hence 
$
{\rm tr}_{(\Gamma, \omega)}(\phi) = \sum_{\alpha\in\Gamma_R}\sum_{i=1}^{m_\alpha}  \omega(\alpha, \alpha) \langle\ol{e}(-\alpha)_i, \phi_0\cdot  e(\alpha)_i\rangle,
$
where $\phi_0$ is the degree $0$ homogeneous component of $\phi$. 
\end{remark}

The elements $t(\beta, \alpha)_{j i}$ (resp. $\ol{t}(-\beta, -\alpha)_{j i}$) generate a $(\Gamma, \omega)$-subalgebra
of $\ST(V)$,  
which we denote by $T$ (resp. $\ol{T}$). There is a $\Z$-grading for $\ST(V)$ with $t(\beta, \alpha)_{j i}$ of degree $1$, 
and $\ol{t}(-\beta, -\alpha)_{j i}$ of degree $-1$. Elements on the left hand side of \eqref{eq:G-a-4}  are of degree $0$.

The algebra $\ST(V)$ has a natural Hopf structure. [This is an analogue of a familiar result in the context of quantum supergroups \cite{M, Z98}.]

\begin{lemma} \label{lem:Hopf-T}
$\ST(V)$ is a Hopf $(\Gamma, \omega)$-subalgebra of $\U^0$ with co-multiplication $\Delta$, co-unit $\varepsilon$  and antipode $S$  respectively given by 
\beq
&&\Delta(t(\beta, \alpha)_{j i})=\sum_{\gamma\in\Gamma_R}\sum_{k=1}^{m_\gamma} t(\beta, \gamma)_{j k}\ot  t(\gamma, \alpha)_{k i}, \label{eq:co-t}\\
&&\Delta(\ol{t}(-\beta, -\alpha)_{j i})=\sum_{\gamma\in\Gamma_R}\sum_{k=1}^{m_\gamma} \ol{t}(-\beta, -\gamma)_{j k}\ot  \ol{t}(-\gamma, -\alpha)_{k i}, \label{eq:co-tbar}\\
&&\varepsilon(t(\beta, \alpha)_{j i})= \varepsilon(\ol{t}(-\beta, -\alpha)_{j i})=\delta_{\alpha \beta}\delta_{i j}, \label{eq:co-eps}\\
&&S(t(\beta, \alpha)_{j i})=\omega(\alpha -\beta, \alpha)\ol{t}(- \alpha, -\beta)_{i j}, \label{eq:co-S-1}\\
&&S(\ol{t}(- \alpha, -\beta)_{i j})=\omega(\alpha, \alpha-\beta) t(\beta, \alpha)_{j i},  \quad \forall  \alpha, \beta, i, j.
\label{eq:co-S-2}
\eeq
In particular, $T$ and $\ol{T}$ are $(\Gamma, \omega)$-bi-subalgebras of 
$\ST(V)$, and the homogeneous subspaces $T_r$ and $\ol{T}_{-r}$ with respect to the $\Z$-grading are two-sided co-ideals.  
\end{lemma}
\begin{proof}
It follows from the right $\U^0$-co-module structure of $V$ and $V^*$ that the co-multiplication and co-unit of $\U^0$ have the claimed properties. Thus $\ST(V)$ is a bi-subalgebra of $\U^0$. It is clear from \eqref{eq:co-t}, 
\eqref{eq:co-tbar} and \eqref{eq:co-eps} that the subalgebras 
$T$ and $\ol{T}$ 
form  $(\Gamma, \omega)$-bi-subalgebras of $\ST(V)$, with 
$\Delta(T_r)\subset T_r\ot T_r$, and $\Delta(\ol{T}_{-r})\subset \ol{T}_{-r}\ot \ol{T}_{-r}$.

Now we consider the Hopf structure. 
We can re-write \eqref{eq:G-a-4} as 
\beq
\sum_{\alpha\in\Gamma_R}\sum_{i=1}^{m_\alpha} t(\beta, \alpha)_{j i} \omega(\gamma, \alpha-\gamma)\ol{t}(-\gamma, -\alpha)_{k i} =\delta_{\beta \gamma}\delta_{j k}, 
\eeq
from which we immediately obtain, for all $\gamma, \zeta, k, \ell$, 
\[
S(t(\zeta, \gamma)_{\ell k})= \sum_{\alpha, i} \sum_{\beta, j} S(t(\zeta, \beta)_{\ell j}) t(\beta, \alpha)_{j i} \omega(\gamma, \alpha-\gamma)\ol{t}(-\gamma, -\alpha)_{k i}. 
\]
By the defining property of the antipode,
$\sum_{\beta, j} S(t(\zeta, \beta)_{\ell j}) t(\beta, \alpha)_{j i}$ $=$ $\delta_{\zeta \alpha}\delta_{\ell i}$. Hence 
\[
\baln
S(t(\zeta, \gamma)_{\ell k})
&= \sum_{\alpha, i} \delta_{\zeta \alpha}\delta_{\ell i} \omega(\gamma, \alpha-\gamma)\ol{t}(-\gamma, -\alpha)_{k i}\\
&= \omega(\gamma-\zeta, \gamma) \ol{t}(-\gamma, -\zeta)_{k \ell},
\ealn
\]
 proving \eqref{eq:co-S-1}. We can similarly prove \eqref{eq:co-S-2}.
\end{proof}

Let us consider actions of $\U$ on $\ST(V)$. It follows general facts discussed in Section \ref{sect:L-R-acts} that 
$\ST(V)$ has $\U$-module algebra structures with respect to the left actions (cf. \eqref{eq:R-act} and \eqref{eq:L-act})
\[
L, R: U\ot\ST(V)\lra\ST(V). 
\]

Note that 
$V^{\ot r}$ is semi-simple as $\U^0$-co-module for any $r$, 
and $L_{\lambda^\sharp}$ is a simple sub-co-module of $V^{\ot r}$ 
if and only if  $\lambda\in \{\mu \in P_{M_+|M_-}\mid\mu\vdash r\}$.
Denote by $T^{(\lambda)}$ the space spanned by the matrix  coefficients of $L_{\lambda^\sharp}$. 
Then 
$T=\oplus_{\lambda\in P_{M_+|M_-}} T^{(\lambda)}$, which is a sub-bi-algebra of $\ST(V)$ 
by Lemma \ref{lem:Hopf-T}. Now
$T$ is a module algebra for $L_\U$ and for also $R_\U$, 
and $T^{(\lambda)}\simeq L_{\lambda^\sharp}^*\ot L_{\lambda^\sharp}$ as an $L_\U\ot R_\U$-module by Example \ref{eg:coeffs}.  We can similarly analyse $\ol{T}$.  
To summarise, 
\begin{theorem}\label{thm:P-W}
The $(\Gamma, \omega)$-sub-bi-algebras $T$ and $\ol{T}$ of $\ST(V)$ can be expressed as
\[
T=\bigoplus_{\lambda\in P_{M_+|M_-}} T^{(\lambda)}, \quad 
\ol{T}=\bigoplus_{\lambda\in P_{M_+|M_-}} \ol{T}^{(\lambda)},
\]
where $T^{(\lambda)}$ and $\ol{T}^{(\lambda)}$ are the spaces of matrix coefficients of 
 $L_{\lambda^\sharp}$ and $L_{\lambda^\sharp}^*$ respectively.  
 Both $T$ and $\ol{T}$  are $L_\U$- and $R_\U$-module algebras, 
 and have the following decompositions as  $L_\U\ot R_\U$-modules. 
\[
T\simeq \bigoplus_{\lambda\in P_{M_+|M_-}} L_{\lambda^\sharp}^*\ot L_{\lambda^\sharp},  
\quad 
\ol{T}\simeq \bigoplus_{\lambda\in P_{M_+|M_-}} L_{\lambda^\sharp}\ot L_{\lambda^\sharp}^*.
\]
\end{theorem}

This clean description exists for $T$ and $\ol{T}$ because $V^{\ot r}$ and ${V^*}^{\ot r}$ 
are semi-simple $\U^0$-co-modules for all $r$. 
It is reminiscent of the description of the algebra of functions of compact Lie groups 
 by the Peter-Weyl theorem. Thus we shall refer to Theorem \ref{thm:P-W} as the Peter-Weyl theorem for $T$ and $\ol{T}$. 

\begin{remark}\label{rmk:coord-HD}
Theorem \ref{thm:P-W} can be interpreted as a generalised Howe duality of type 
$(\gl(V(\Gamma, \omega)), \gl(V(\Gamma, \omega)))$, from which we can deduce the 
generalised Howe duality of type 
$(\gl(V(\Gamma, \omega)), \gl(V'(\Gamma, \omega)))$; 
see Section  \ref{sect:coord-HD}.
\end{remark}

\begin{remark}
Results similar to Theorem \ref{thm:P-W} were obtained for the general linear Lie superalgebra  in \cite{SchZ02}, and for the quantum general linear supergroup in \cite{Z98}. 
\end{remark}

As a side comment, we mention a family of  unital associative $(\Gamma, \omega)$-algebras
arising from this context.  
Now $T_r$ is the space of matrix coefficients of $V^{\ot r}$ for any fixed $r\ge 1$. 
\[
T_r=\bigoplus_{\substack{\lambda\in P_{M_+|M_-}\\ \lambda\vdash r}} T^{(\lambda)}.
\]
It forms a sub $(\Gamma, \omega)$-co-algebra of $\ST(V)$.  Its dual space 
$T_r^*$ is a unital associative $(\Gamma, \omega)$-algebra, with multiplication 
$\mu^{(r)}: x\ot y\mapsto xy$, for any $x, y\in {T_r}^*$,  defined by 
\beq
\quad \langle x y, f\rangle = \sum_{(f)}\langle x\ot y, \Delta(f)\rangle, \quad \forall f\in T_r;  
\eeq
and unit ${\mathbb 1}_r$ being the restriction of the co-unit of $\ST(V)$ to $T_r ^*$. 
\begin{definition}\label{def:Schur-alg}
Call ${\mathbb S}_{\Gamma, \omega}(r)=(T_r^*,  \mu^{(r)},  {\mathbb 1}_r)$ a Schur $(\Gamma, \omega)$-algebra. 
\end{definition}
Note that the definition is valid over any field, and commutative rings. 

This algebra is nothing but the image of $\U(\gl(V))$ in the representation furnished by the module $V^{\ot r}$. 
Over $\C$, Theorem \ref{thm:SW} provides a complete understanding of this algebra. However, the algebra will become more interesting to study over fields of positive characteristics.

\subsubsection{The coordinate algebra as a ring of fractions}\label{sect:fractions}
We shall identify $\ST(V)$ with a ring of fractions of $T$ as graded commutative $(\Gamma, \omega)$-algebra.

Consider  $\ST(V)\ot\End_\C(V)$ as a $(\Gamma, \omega)$-algebra. 
The following elements belong to  the degree $0$ subalgebra $(\ST(V)\ot\End_\C(V))_0$, 
\[
\baln
t&=\sum_{\alpha, \beta, i, j}\omega(\beta, \alpha-\beta) t(\beta, \alpha)_{j i} \ot e(\beta, \alpha)_{j i}, \\
\ol{t} &=\sum_{\gamma, \zeta, k, \ell} \omega(\gamma -\zeta, \gamma-\zeta)\ol{t}(- \gamma, -\zeta)_{k \ell} \ot e(\zeta, \gamma)_{\ell k}.
\ealn
\]
Note that $\ol{t}=\sum_{\gamma, \zeta, k, \ell} \omega(\zeta, \gamma-\zeta) S(t(\zeta, \gamma)_{\ell k}) \ot e(\zeta, \gamma)_{\ell k}.$
Using Lemma \ref{rem:T-def}, we can prove the following result by straightforward computations. 
\begin{corollary} The elements $t$ and $\ol{t}$ are left and right inverses of each other in $(\ST(V)\ot\End_\C(V))_0$,  that is, 
\beq
t \ol{t} = \ol{t} t =1\ot 1.
\eeq
\end{corollary}

\begin{proof}
The corollary is equivalent to \eqref{eq:G-a-4} and the following relation 
\beq\label{eq:G-a-4-1}
\sum_{\alpha, i} \omega(\alpha-\beta, \alpha) \ol{t}(- \alpha, -\beta)_{i j}  t(\alpha, \gamma)_{i k}  
= \delta_{\beta \gamma}\delta_{j k}, \quad \forall
\beta, \gamma, j, k,
\eeq
which can be obtained by applying the antipode to \eqref{eq:G-a-4}.
\end{proof}

To better understand the matrices $t$ and $\ol{t}$, we decompose $\End_\C(V)$ into blocks as
\[
\begin{pmatrix}\Hom_\C(V_+, V_+) & \Hom_\C(V_-, V+)\\
 \Hom_\C(V_+, V_-) & \Hom_\C(V_-, V_-)
\end{pmatrix}.
\]
Then the elements $t$ and $\ol{t}$ of $(\ST(V)\ot\End_\C(V))_0$ have the block decompositions
\[
\baln
&t = t_{++}+t_{+-}+t_{-+}+t_{--}, \quad \ol{t} = \ol{t}_{++}+\ol{t}_{+-}+\ol{t}_{-+}+\ol{t}_{--},
\ \text{ with}\\ 
&t_{\pm \pm} =  \sum_{\alpha, \beta\in\Gamma_R^\pm}\sum_{i, j}
	\omega(\beta, \alpha-\beta) t(\beta, \alpha)_{j i} \ot e(\beta, \alpha)_{j i}, \\
&t_{\pm \mp} =  \sum_{\alpha\in\Gamma_R^\pm, \beta\in\Gamma_R^\mp}\sum_{i, j}
	\omega(\beta, \alpha-\beta) t(\beta, \alpha)_{j i} \ot e(\beta, \alpha)_{j i}, \\
&\ol{t}_{\pm \pm} =\sum_{\gamma, \zeta\in\Gamma_R^\pm} \sum_{k, \ell}
	\omega(\gamma -\zeta, \gamma-\zeta)\ol{t}(- \gamma, -\zeta)_{k \ell} \ot e(\zeta, \gamma)_{\ell k}, \\
&\ol{t}_{\pm \mp} =\sum_{\gamma\in\Gamma_R^\pm, \zeta\in\Gamma_R^\mp} \sum_{k, \ell}
	\omega(\gamma -\zeta, \gamma-\zeta)\ol{t}(- \gamma, -\zeta)_{k \ell} \ot e(\zeta, \gamma)_{\ell k}. 
\ealn
\]
Note that $t_{\pm \pm}, \ol{t}_{\pm \pm} \in\ST(V)\ot\End_\C(V_\pm)$, and 
$t_{+-}+t_{-+}$ and $\ol{t}_{+-}+\ol{t}_{-+}$ are nilpotent.

We will introduce a ring of fractions $\BD^{-1}T$ of $T$ such that $t$ has an inverse (i.e., left and right inverse) in 
$\BD^{-1}T\ot\End_\C(V)$. This will be done in several steps.

Let $p_+: V\lra V_+\subset V$ be the projection onto $V_+$, 
and $p_+^{\ot r}: \wedge_\omega^r V \lra \wedge_\omega^r V_+\subset \wedge_\omega^r V$ for any $r$. 
Similarly, we let $\ol{p}_+: V^*\lra V^*_+\subset V^*$ be the projection onto $V^*_+$, 
and $\ol{p}_+^{\ot r}: \wedge_\omega^r V^* \lra \wedge_\omega^r V^*_+\subset \wedge_\omega^r V^*$ 
for any $r$. These maps are all $\gl(V_+)$-module homomorphisms. 

Note that $\wedge_\omega^{M_+} V_+$ and  $\wedge_\omega^{M_+} V^*_+$ are both $1$-dimensional,  
and are homogeneous of   degrees $\pm\theta$ respectively, 
where $\theta=\sum_{\alpha\in\Gamma_R^+} m_\alpha\alpha$.
We choose any basis elements $\Lambda_+$ for $\wedge_\omega^{M_+} V_+$ and 
$\Lambda^*_+$ for $\wedge_\omega^{M_+} V^*_+$. 
Since $p_+^{\ot M_+}\circ\delta_{V^{\ot M_+}}$ and $\ol{p}_+^{\ot M_+}\circ
\delta_{{V^*}^{\ot M_+}}$ are both $\gl(V_+)$-module homomorphisms, 
 there exists $D_+, \ol{D}_+\in T$ such that 
\beq
 &p^{\ot M_+}\circ\delta_{V^{\ot M_+}}(\Lambda_+) 
 	= \Lambda_+\ot   D_+,  \label{eq:D+}\\ 
 &\ol{p}_+^{\ot M_+}\circ \delta_{{V^*}^{\ot M_+}}(\Lambda^*_+ )
 	= \Lambda^*_+ \ot  \ol{D}_+.   \label{eq:Dbar+}
\eeq
 
 Observe that $D_+$ and $\ol{D}_+$ are homogeneous of degree $0$ in the $\Gamma$-grading. 
 Thus they commute with all elements of $\ST(V)$, that is  $D_+ f = f D_+$ and $ \ol{D}_+ f = f \ol{D}_+$  for all $f\in \ST(V)$. 
 
 Denote by $T_{D_+}$ the ring of fractions of $T$ obtained by inverting $D_+$. This is a
 $(\Gamma, \omega)$-algebra. 
 \begin{assertion}\label{aser:inv+}
 The element $t_{++}\in T\ot\End_\C(V_+)$ has a multiplicative (left and right)  inverse   in $T_{D_+}\ot\End_\C(V_+)$. 
\end{assertion}
\begin{proof}
Note that $\dim\left(\wedge_\omega^{M_+-1} V_+\right)=M_+$. 
Let $\{f(\theta-\alpha)_i\mid \alpha\in\Gamma_R^+, 1\le i\le  m_\alpha\}$ be a basis  of 
$\wedge_\omega^{M_+-1} V_+$ such that 
\beq\label{eq:Lam+}
\Lambda_+=\sum_{\alpha\in\Gamma_R^+}\sum_{i=1}^{m_\alpha} e(\alpha)_i\ot f(\theta-\alpha)_i.
\eeq
Then there exist elements $A(\theta-\beta, \theta-\alpha)_{j i}$ of $T$ such that 
\[
\baln
p_+^{\ot (M_+-1)}\circ\delta_{V^{\ot (M_+-1)}}(f(\theta-\alpha)_i) = \sum_{\beta\in\Gamma_R^+}\sum_ j f(\theta-\beta)_j \ot A(\theta-\beta, \theta-\alpha)_{j i}. 
\ealn
\]
Now $p^{\ot M_+}\circ\delta_{V^{\ot M_+}}(\Lambda_+)$ can be expressed as
\[
\baln
\sum_{\alpha, \beta, \gamma\in\Gamma_R^+} \sum_{i, j, k} \omega(\alpha-\beta, \theta-\gamma) e(\beta)_j\ot  f(\theta-\gamma)_k \ot t(\beta, \alpha)_{j i} A(\theta-\gamma, \theta-\alpha)_{k i}, 
\ealn
\]
and hence 
\beq\label{eq:t-inv-r}
\sum_{\alpha\in\Gamma_R^+} \sum_i 
t(\beta, \alpha)_{j i}  \omega(\alpha-\gamma, \theta-\gamma) A(\theta-\gamma, \theta-\alpha)_{k i}
=\delta_{\beta \gamma}\delta_{j k} D_+. 
\eeq
This implies that $t_{++}\in T\ot\End_\C(V_+)$ has a right inverse in $T_{D_+}\ot\End_\C(V_+)$. 

We can similarly consider $\wedge_\omega^{M_+-1} V^*_+$. Let
$\{ \ol{f}(\alpha-\theta)_i\mid \alpha\in\Gamma_R^+, 1\le i\le m_\alpha\}$ be a basis for it 
 such that 
\beq\label{eq:Lambar+}
\Lambda^*_+ = \sum_{\alpha\in\Gamma_R^+}\sum_{i=1}^{m_\alpha} \ol{f}(\alpha-\theta)_i\ot\ol{e}(-\alpha)_i.
\eeq
Then there are $\ol{A}(\beta-\theta, \alpha-\theta)_{j i}\in \ol{T}$, for $\alpha, \beta\in\Gamma_R^+$, $1\le i \le m_\alpha$ and 
$1\le j\le m_\beta$,  such that 
\[
\ol{p}_+^{\ot (M_+-1)}\circ \delta_{{V^*}^{\ot (M_+-1)}}(\ol{f}(\alpha-\theta)_i)= \sum_{\beta\in\Gamma_R^+}\sum_{j=1}^{m_\beta} \ol{f}(\beta-\theta)_j\ot \ol{A}(\beta-\theta, \alpha-\theta)_{j i}. 
\]
Using  \eqref{eq:Lambar+} and this relation in equation \eqref{eq:Dbar+}, we obtain 
\beq
\sum_{\alpha\in\Gamma_R^+} \sum_i   \omega(\alpha-\beta,   -\alpha)\ol{t}(-\gamma, -\alpha)_{k i} \ol{A}(\beta-\theta, \alpha-\theta)_{j i}=\delta_{\beta \gamma}\delta_{j k} \ol{D}_+. 
\eeq
Applying the antipode to this equation, we obtain 
\beq\label{eq:t-inv-l}
\sum_{\alpha\in\Gamma_R^+} \sum_i   
S(\ol{A}(\beta-\theta, \alpha-\theta)_{j i}) t(\alpha, \gamma)_{i k} 
=\delta_{\beta \gamma}\delta_{j k} S(\ol{D}_+), 
\eeq
which is an equation in $T$. 

If we invert  $S(\ol{D}_+)$, equations \eqref{eq:t-inv-l}  gives rise to a left inverse of $t_{++}\in T\ot\End_\C(V_+)$. Now  left and right inverses must be the same. By inspecting \eqref{eq:t-inv-l} and \eqref{eq:t-inv-r}, we conclude that $S(\ol{D}_+)$ must be a scalar multiple of $D_+$. Hence the inverse of $t_{++}$ belongs to $T_{D_+}\ot\End_\C(V_+)$. 

This completes the proof. 
\end{proof}

Now consider $S_\omega^r(V_-)$ and $S_\omega^r(V^*_-)$. 
Let $p_-: V\lra V_-\subset V$ be the projection onto $V_-$. 
Since $S_\omega^{M_-}(V_-)$ is one dimensional,   for any basis element $\Lambda_-$ of it, there is $D_-\in T$ such that 
\beq \label{eq:D-}
 p_-^{\ot M_-}\circ\delta_{V^{\ot M_-}}(\Lambda_-) = \Lambda_-\ot   D_-.
\eeq

Observe that $D_-$ is homogeneous of degree $0$. 
Thus it commutes with all elements of $\ST(V)$, that is  $D_- f = f D_-$ for all $f\in \ST(V)$. 
 
 Let $T_{D_-}$ be the ring of fractions obtained by inverting $D_-$. 

\begin{assertion}\label{aser:inv-}
 The element $t_{--}\in T\ot\End_\C(V_-)$ has a multiplicative inverse in $T_{D_-}\ot\End_\C(V_-)$. 
\end{assertion}
\begin{proof}
The proof is similar to that of the Assertion \ref{aser:inv+}. We omit details. 
\end{proof}

Introduce the multiplicative set $\BD=\{D_+^i D_-^j\mid i, j\in\Z_+\}$, and denote by $\BD^{-1}T$ the ring of fractions of $T$ with respect $\BD$.  We have the following result. 
\begin{theorem}\label{thm:fract-T}
The element $t\in \ST(V)\ot\End_\C(V)$ has a multiplicative inverse in $\BD^{-1}T\ot\End_\C(V)$. 
 Furthermore, $\ST(V)\simeq \BD^{-1}T$ as $(\Gamma, \omega)$-algebra.
\end{theorem}
 \begin{proof}
 
It follows Assertions \ref{aser:inv+} 
and \ref{aser:inv-} that $t_{++}$ is invertible in $\BD^{-1}T\ot \End_\C(V_+)$
and $t_{- -}$ is invertible in $\BD^{-1}T\ot \End_\C(V_-)$. 
We will denote their inverses by $\hat{t}_{++}^{-1}$ and $\hat{t}_{- -}^{-1}$ respectively.   
As $t_{++}t_{--}=t_{--}t_{++}=0$, we have 
\[
(t_{++} + t_{--})^{-1}= \hat{t}_{++}^{-1} + \hat{t}_{--}^{-1}. 
\]
Thus $
t= (t_{++}+ t_{--})(1 + \hat{t}_{++}^{-1} t_{+-} + \hat{t}_{--}^{-1}t_{-+}).
$
Since $\hat{t}_{++}^{-1} t_{+-} +\hat{t}_{--}^{-1} t_{-+}$ is nilpotent,  
\beq
(1 + \hat{t}_{++}^{-1} t_{+-} +\hat{t}_{--}^{-1} t_{-+})^{-1} = 1+ \sum_{n\ge 1} (-1)^n
(\hat{t}_{++}^{-1} t_{+-} +\hat{t}_{--}^{-1} t_{-+})^n,
\eeq
where the right hand side is a finite sum.  It follows that 
\beq
t^{-1}=\left(1+ \sum_{n\ge 1} (-1)^n(t_{++}^{-1} t_{+-} t_{--}^{-1}+ t_{-+})^n\right)(\hat{t}_{++}^{-1} + \hat{t}_{--}^{-1}).
\eeq
Note that $\BD^{-1}T$ is generated by the matrix entries of $t$ and $t^{-1}$. 

Recall that $\ST(V)$ is generated by the matrix entries of $t$ and $\ol{t}$, and $\ol{t}$ is the inverse of $t$ 
in $\ST(V)\ot\End_\C(V)$. Thus the algebra map
$\iota: \ST(V)\lra \BD^{-1}T$ such that $(\iota\ot\id)(t)=t$ and $(\iota\ot\id)(\ol{t})=t^{-1}$ is an isomorphism. 
 \end{proof}
 
 \begin{remark}
 Assume that  $V_-=0$, thus $V=V_+$. Then $D=D_+$ satisfies $\Delta(D)=D\ot D$ and $\varepsilon(D)=1$. 
 The $r$-th power $D^r$ spans a $1$-dimensional  $\gl(V)\times\gl(V)$-module isomorphic 
 to $L_{r{\mathbb 1}}^*\ot L_{r{\mathbb 1}} \simeq L_{-r{\mathbb 1}}\ot L_{r{\mathbb 1}}$, 
 where ${\mathbb 1}=(1, 1, \dots, 1)$.  
 This makes sense for $r<0$ as well. 
By using Theorem \ref{thm:P-W} and noting that 
$L_\lambda\simeq L_{\lambda_{M_+}{\mathbb 1}}\ot L_{\lambda-\lambda_{M_+}{\mathbb 1}}$, 
we obtain in this case
  \[
  \baln
 \ST(V)&\simeq T_{D_+} \simeq \sum_{\substack{\lambda\in P_{M_+}, \ \lambda_{M_+}=0, \\ r\in\Z}}  L_{\lambda+r{\mathbb 1}}^*\ot L_{\lambda+r{\mathbb 1}}
 = \sum_{\substack{\mu=(\mu_1, \mu_2, \dots \mu_{M_+}) \in\Z^{M_+}, \\ \mu_1\ge \mu_2\ge \dots \ge \mu_{M_+} }} L_\mu^*\ot L_\mu.
 \ealn
 \]
 \end{remark}

\subsection{Another construction of the coordinate algebra}
We give another construction of the coordinate algebra $\ST(V)$, 
which only requires information 
about the $\Gamma$-graded vector space $V=V(\Gamma, \omega)$. This is inspired by the 
construction of quantum groups in \cite{M}. 

\subsubsection{A Hopf $(\Gamma, \omega)$-algebra}\label{sect:poly}

Let $F(V)=S_\omega(V^*\ot V)$ and $\ol{F}(V)=S_\omega(V\ot V^*)$ be the symmetric $(\Gamma, \omega)$-algebras  over 
$V^*\ot V$ and $V\ot V^*$ respectively. Recall the canonical surjection $T(V^*\ot V)\lra F(V)$ 
from the tensor algebra to $F(V)$.
Let $ X(\beta, \alpha)_{j i}$ be the image of $\ol{e}(-\beta)_j\ot e(\alpha)_i$ under the map for all $\alpha, \beta, i, j$. 
Also introduce the elements $\ol{X}(-\alpha, -\beta)$ of $\ol{F}(V)$ such that  
$\omega(\beta, \beta) e(\beta)_j\ot \ol{e}(-\alpha)_i\mapsto \ol{X}(-\alpha, -\beta)$, for all $\alpha, \beta, i, j$, 
under the canonical surjection 
$T(V\ot V^*)\lra \ol{F}(V)$. 

Note the following fact.
\begin{lemma}\label{eq:F-co-act}
The linear maps 
\[
\baln
&V\lra V\ot S^1_\omega(V^*\ot V), \quad e(\alpha)_i\mapsto \sum_{\beta, j} e(\beta)_j\ot X(\beta, \alpha)_{j i}, \\
&V^*\lra V^*\ot S^1_\omega(V\ot V^*), \quad \ol{e}(-\alpha)_i\mapsto \sum_{\beta, j} e(-\beta)_j\ot \ol{X}(-\beta, -\alpha)_{j i}
\ealn
\]
are $\gl(V)$-module homomorphisms, where $\gl(V)$ acts on both $V\ot S^1_\omega(V^*\ot V)\simeq V\ot V^*\ot V$ and $V^*\ot S^1_\omega(V\ot V^*)\simeq V^*\ot V\ot V^*$ diagonally.  
\end{lemma}
It is worth observing that the factor $\omega(\beta, \beta)$ in the definition of $\ol{X}(-\alpha, -\beta)$ is crucial to make the second map a $\gl(V)$-module homomorphism. 

Let $J_V$ be the two-sided $(\Gamma, \omega)$-ideal in $F(V)\ot \ol{F}(V)$  generated by the elements 
\beq
&&\sum_{\alpha, i} \omega(\gamma, \alpha-\beta) X(\beta, \alpha)_{j i} \ol{X}(-\gamma, -\alpha)_{k i} - \delta_{\beta \gamma}\delta_{j k},  \label{eq:JV-1}\\
&&\sum_{\alpha, i} \omega(\alpha-\beta, \alpha) \ol{X}(- \alpha, -\beta)_{i j}  X(\alpha, \gamma)_{i k}  
- \delta_{\beta \gamma}\delta_{j k}, \quad \forall
\beta, \gamma, j, k, \label{eq:JV-2}
\eeq
and define the associative $(\Gamma, \omega)$-algebra
\[
\SF(V)=F(V)\ot \ol{F}(V)/J_V.
\]
We will abuse notation slightly to use the same symbols to denote the images 
of $X(\beta, \alpha)_{j i} $ and $\ol{X}(- \alpha, -\beta)_{i j}$ in the quotient algebra.

We have the following result.
\begin{lemma} The algebra $\SF(V)$ has the structure of a Hopf $(\Gamma, \omega)$-algebra with 
co-multiplication $\Delta: \SF(V)\lra \SF(V)\ot\SF(V)$, co-unit $\varepsilon: \SF(V)\lra \C$ and antipode $S: \SF(V)\lra \SF(V)$ (an anti-homomorphism) respectively defined by 
\beq
&&\Delta(X(\beta, \alpha)_{j i})=\sum_{\gamma\in\Gamma_R}\sum_{k=1}^{m_\gamma} X(\beta, \gamma)_{j k}\ot  X(\gamma, \alpha)_{k i}, \label{eq:F-Del-1}\\
&&\Delta(\ol{X}(-\beta, -\alpha)_{j i})=\sum_{\gamma\in\Gamma_R}\sum_{k=1}^{m_\gamma} \ol{X}(-\beta, -\gamma)_{j k}\ot  \ol{X}(-\gamma, -\alpha)_{k i},  \label{eq:F-Del-2}\\
&&\varepsilon(X(\beta, \alpha)_{j i})= \varepsilon(\ol{X}(-\beta, -\alpha)_{j i})=\delta_{\alpha \beta}\delta_{i j},  \label{eq:F-epsi}\\
&&S(X(\beta, \alpha)_{j i})=\omega(\alpha -\beta, \alpha)\ol{X}(- \alpha, -\beta)_{i j}, \label{eq:F-S-1}\\
&&S(\ol{X}(- \alpha, -\beta)_{i j})=\omega(\alpha, \alpha-\beta) X(\beta, \alpha)_{j i},  \quad \forall  \alpha, \beta, i, j.
\label{eq:F-S-2}
\eeq
\end{lemma}
\begin{proof}
It is quite obvious that $(\SF(V), \Delta, \varepsilon)$ is a $(\Gamma, \omega)$-bi-algebra. 
Now we show that the anti-homomorphism $S$ is an antipode.

Recall that $J_V$ is generated by the two types of elements given by \eqref{eq:JV-1} and \eqref{eq:JV-2}. This guarantees that the elements 
\[
\baln
X&=\sum_{\alpha, \beta, i, j}\omega(\beta, \alpha-\beta) X(\beta, \alpha)_{j i} \ot e(\beta, \alpha)_{j i} \quad  \text{and}\\
\ol{X} &=\sum_{\gamma, \zeta, k, \ell} \omega(\gamma -\zeta, \gamma-\zeta)\ol{X}(- \gamma, -\zeta)_{k \ell} \ot e(\zeta, \gamma)_{\ell k}
\ealn
\] 
are two-sided inverses of each other in $\SF(V)\ot\End_C(V)$, that is 
\beq \label{eq:F-inv}
X \ol{X}=\ol{X} X=1\ot 1. 
\eeq
Now with the given bi-algebra structure for $\SF(V)$, equation \eqref{eq:F-inv} is exactly the defining property of the antipode given by \eqref{eq:F-S-1} and \eqref{eq:F-S-2}.
\end{proof}

Now $V$ and $V^*$ form right $\SF(V)$ modules with the structure maps 
\[
\delta^{(\SF)}_V: V\lra V\ot \SF(V), \quad \delta^{(\SF)}_{V^*}: V^*\lra V\ot \SF(V)
\]
given by the maps introduced in Lemma \ref{eq:F-co-act}.  The constructions in Section \ref{sect:fractions} 
can be repeated for $F(V)$ and $\SF(V)$. 
We have the elements $D_+(X), D_-(X)\in F(V)$, which are analogues  of the elements  $D_+, D_-$ defined by 
\eqref{eq:D+} and \eqref{eq:D-} respectively. 
We can also obtain these elements by replacing $t(\alpha, \beta)_{i j}$  in $D_+$ and $D_-$ by $X(\alpha, \beta)_{i j}$ for any $\alpha, \beta, i, j$.  Let $\Psi=\{D_+(X)^i D_-(X)^j\mid i, j\in\Z_+\}$, and introduce the ring of fractions $\Psi^{-1}F(V)$. 
Then we have the following analogue of Theorem \ref{thm:fract-T}.
 \begin{theorem} \label{thm:fract-F}
The element $X\in \SF(V)\ot\End_\C(V)$ has a multiplicative inverse in $\Psi^{-1}F(V)\ot\End_\C(V)$. 
 Furthermore, $\SF(V)\simeq \Psi^{-1}F(V)$ as $(\Gamma, \omega)$-algebra.
\end{theorem}
\begin{proof}
The proof is omitted as it merely repeats the proof for Theorem \ref{thm:fract-T} in the present context.
\end{proof}

\subsubsection{An alternative construction of $\ST(V)$}

We can introduce a $\Z$-grading to $\SF(V)$ with 
$X(\beta, \alpha)_{j i}$ having degree $1$, and $\ol{X}(-\gamma, -\alpha)_{k i}$ degree $-1$. 
Then $J_V$ is a homogeneous ideal. 
Clearly $F(V)\cap J_V=0$  and $\ol{F}(V)\cap J_V=0$, thus $F(V)$ and $\ol{F}(V)$ 
are $\Z$-graded sub $(\Gamma, \omega)$-bi-algebras of $\SF(V)$.
Denote by $F(V)_r$ the degree $r$ subspace of $F(V)$ with respect to the $\Z$-grading. 
Then $\Delta(F(V)_r)\subset F(V)_r\ot F(V)_r$, 
that is, $F(V)_r$ is a two-sided co-ideal of $F(V)$.

Recall from Section \ref{sect:coord-alg}.1 that $T$ and $\ol{T}$ are sub $(\Gamma, \omega)$-bi-algebras of $\ST(V)$.
Also recall that $\ST(V)$ admits a $\Z$-grating with $t(\beta, \alpha)_{j i}$ having degree 
$1$ and $\ol{t}(-\beta, -\alpha)_{j i}$ having degree $-1$. The degree $r$ subspace $T_r$ of $T$ is a
two-sided co-ideal, i.e.,  $\Delta(T_r)\subset T_r\ot T_r$ for all $r$. 

\begin{lemma}\label{lem:bialg-iso}
As $\Z$-graded $(\Gamma, \omega)$-bi-algebras, $T\simeq F(V)$ and $\ol{T}\simeq \ol{F}(V)$.  
\end{lemma}
\begin{proof} Clearly there is a  surjective $\Z$-graded $(\Gamma, \omega)$-bi-algebra map $\psi:  F(V)\lra T$ such that 
$\psi(X(\beta, \alpha)_{j i})=t(\beta, \alpha)_{j i}$ for all $\alpha, \beta, i, j$. This map  
 renders the following diagram commutative for all $r$.
\[
\begin{tikzcd}
& V^{\ot r} \arrow[dl, swap, "\delta^{(\SF)}_{V^{\ot r}}"] \arrow[dr,   "\delta_{V^{\ot r}}"] &\\
V^{\ot r} \ot F(V)_r \arrow[rr, swap, "\id_{V^{\ot r}}\ot \psi"]	&&V^{\ot r}\ot T_r. 
\end{tikzcd}
\]

The action of $\Sym_r$ on $V^{\ot r}$ given by Corollary \ref{cor:sym} commutes with the co-actions of $F(V)$ and $T$. 
Under the $\Sym_r$-action and $T$-co-action, 
$V^{\ot r}=\sum_\lambda  S_\lambda\ot L_\lambda$ Theorem \ref{thm:SW}.(1), 
where the sum is over $\{\lambda \in P_{M_+|M_-}\mid\lambda\vdash r\}$. Let $T^{(\lambda)}\subset T_r$ be the subspace spanned by the matrix coefficients of $L_\lambda\subset V^{\ot r}$.
Note that under the $\Sym_r$ -action and $F(V)$-co-action, $V^{\ot r}$ has the same decomposition, 
but $L_\lambda$ should be considered as an $F(V)$-co-module.  Let $F^{(\lambda)}\subset F(V)_r$ be the subspace spanned by the matrix coefficients of $L_\lambda$.  
The commutativity of the above diagram together with the surjectivity of $\psi$ imply that $L_\lambda$ 
is a simple $F(V)$-co-module. Thus $\psi(F^{(\lambda)})=T^{(\lambda)}$ and 
$\ker(\psi)\cap F^{(\lambda)}=0$ for all $\lambda$. This shows that $\psi$ is an isomorphism. 

We can similarly prove the isomorphism between $\ol{T}$ and $\ol{F}(V)$.  
\end{proof}

We have the following result.

\begin{theorem}\label{thm:poly-constr}
There is a unique Hopf $(\Gamma, \omega)$-algebra isomorphism 
$
\SF(V) \stackrel{\simeq}\lra \ST(V)
$
obtained by extending the map $X(\beta, \alpha)_{j i}\mapsto t(\beta, \alpha)_{j i}$ and 
$\ol{X}(-\beta, -\alpha)_{j i}\mapsto \ol{t}(-\beta, -\alpha)_{j i}$ for all $\beta, \gamma, j, k.$
\end{theorem}
\begin{proof}
It is easy to see that there exists such a Hopf $(\Gamma, \omega)$-algebra homomorphism $\psi: \SF(V)\lra \ST(V)$, 
which is unique.  The homomorphism is clearly surjective, and its restrictions to 
$F(V)$ and $\ol{F}(V)$ have images $T$ and $\ol{T}$ respectively. They yield
bi-algebra isomorphisms $\psi|_F: F(V)\stackrel{\simeq}{\lra} T$ 
and $\ol\psi|_{\ol{F}}: \ol{F}(V)\stackrel{\simeq}{\lra} \ol{T}$, 
which are those described in Lemma \ref{lem:bialg-iso}.

In order to prove that $\psi$ is a Hopf $(\Gamma, \omega)$-algebra isomorphism, 
it suffices to show that it is a $\Gamma$-graded vector space isomorphism. Note that the 
bi-algebra isomorphism $\psi|_F: F(V)\stackrel{\simeq}{\lra} T$  
satisfies $\psi(D_\pm(X))=D_\pm$. Thus it extends to an isomorphism 
$\Psi^{-1}F(V)\stackrel{\simeq}\lra \BD^{-1}T$ between the rings of fractions. Now by Theorems \ref{thm:fract-F} and 
\ref{thm:fract-T}, we have a $(\Gamma, \omega)$-algebra isomorphism 
$
\SF(V)\simeq \ST(V),  
$
proving the theorem. 
\end{proof}

\subsubsection{Another proof of the generalised Howe duality of type 
$(\gl(V), \gl(V'))$}\label{sect:coord-HD}
We expound on Remark \ref{rmk:coord-HD} to provide a much 
simpler proof for 
the generalised Howe duality of type 
$(\gl(V), \gl(V'))$ (i.e., Theorem \ref{thm:HD-general}), 
by using Theorem \ref{thm:P-W}. 

It is important to note that the proof of Theorem \ref{thm:P-W}  only  
used the semi-simplicity of $V^{\ot r}$ and their duals as $\gl(V)$-modules for all $r$. 
It did not depend on Lemma \ref{lem:at-infty} or Theorem \ref{thm:HD-general}.

\begin{corollary}\label{cor:coord-pf}
Theorem \ref{thm:P-W} implies Lemma \ref{lem:at-infty}, 
and hence also  
the generalised Howe duality of type 
$(\gl(V(\Gamma, \omega)), \gl(V'(\Gamma, \omega)))$, i.e., Theorem \ref{thm:HD-general} 
\end{corollary}
\begin{proof}
By Theorem \ref{thm:poly-constr}, we have the isomorphism of $\gl(V)\times \gl(V)$-module algebras
\[
T\simeq F(V)=S_\omega(V^*\ot V).
\] 
Hence statements on $T$ in Theorem \ref{thm:P-W} imply Lemma \ref{lem:at-infty}, 
the generalised Howe duality of type 
$(\gl(V(\Gamma, \omega)), \gl(V(\Gamma, \omega)))$.
Now Theorem \ref{thm:HD-general} 
follows from Lemma \ref{lem:at-infty}
by truncation arguments, as we have already seen.  
\end{proof}

\begin{remark}
Howe dualities for quantum general linear (super)groups 
were originally obtained \cite{LZ, WZ, Z03} following this ``coordinate algebra approach''. 
The method was also applied to prove a Howe duality 
for the quantum queer superalgebra \cite{CW20}.  
\end{remark}

\subsection{Borel-Weil construction for simple tensor modules}\label{sect:BW}

We give a construction of simple tensor representations in the spirit of the Borel-Weil theorem, 
which realises representations on spaces of holomorphic sections of line bundles on flag manifolds.  
The general methods  are adapted from  \cite{GoZ, Z98, Z04} on quantum (super)groups to the present context.

\subsubsection{The $\fh$-invariant subalgebra and finitely generated projective modules }

Recall from Section \ref{sect:roots} the Cartan and Borel subalgebras 
$\fh\subset \fb=\fh+\fn$ of $\gl(V)$. 
Let 
\[
\CA:=\ST(V)^{R_\fh}=\{ f\in \ST(V)\mid R_h(f)=0, \forall h\in\fh\},
\]
which is a graded commutative $(\Gamma, \omega)$-subalgebra of $\ST(V)$. 
Given any $\fh$-module $V^0$, let 
\[
V^0_\ST=\ST(V)\ot_\C V^0.
\]
We will consider various module structures of it.
Note that $V^0_\ST$ is 
a left and right $\CA$-module with the obvious actions 
\beq\label{eq:A-act}
\baln
&\CA\ot V^0_\ST\lra V^0_\ST, \quad  g s = \sum_i g f_i\ot v_i,\\
&V^0_\ST \ot \CA\lra V^0_\ST, \quad  s g= \sum_i \sum_i \omega(d(v_i), d(g))f_i g\ot v_i,
\ealn
\eeq
for all $g\in \CA$ and $s=\sum_i f_i\ot v_i\in V^0_\ST$.
Because of the graded commutativity of $\ST(V)$, 
the left and right module structures are related to each other by 
\[
s g= \sum_i \omega(d(v_i), d(g))f_i g\ot v_i 
= \omega(d(s), d(g)) g s.
\]
Thus it suffices to consider left $\CA$-modules only. We shall do this hereafter, 
and also refer to left $\CA$-modules simply as $\CA$-modules.

Also note that $V^0_\ST$ is a left $\U(\fh)$-module, with the action 
\[
\U(\fh)\ot V^0_\ST \lra V^0_\ST
\]
defined, for all 
$s=\sum_i f_i\ot v_i\in V^0_\ST$ and $h\in\fh$, by  
\beq\label{eq:h-act}
h\cdot s  =  \sum_i (R_h(f_i)\ot v_i + f_i \ot h\cdot v_i).
\eeq
We are interested in the subspace of $\fh$-invariants with respect to the above action:
\beq
\CS(V^0)=\left(\ST(V)\ot V^0\right)^\fh. 
\eeq
It is easy to see from \eqref{eq:h-act} that an element $s\in V^0_\ST$ belongs to $\CS(V^0)$ if and only if 
the following condition is satisfied. 
\beq\label{eq:homog}
(R_h\ot\id_{V^0})  s = - (\id_{\ST(V)}\ot h) s, \quad \forall h\in\fh.
\eeq

\begin{lemma} Retain notation above. 
The following statements are true: 
\begin{enumerate}\label{lem:sects}
\item $\CA$ is a left $\gl(V)$-module algebra with respect to the action $L$; 
\item $\CS(V^0)$ is a left $\gl(V)$-module, where any $X\in\gl(V)$ acts by $L_X\ot\id_{V^0}$;  
\end{enumerate}
\end{lemma}
\begin{proof}
Both statements follow from the fact that the $\U(\gl(V))$-actions $L$ and $R$ commute by 
Lemma \ref{lem:R-L}. 
\end{proof}

Let $W$ be a right $\ST(V)$-co-module, thus is a left $\U(\gl(V))$-module, 
and we write $\pi_W: \U(\gl(V))\lra \End_\C(W)$ for the associated representation. 
Consider
$W_\ST=\ST(V)\ot W$, and let $\eta:   W_\ST\lra    W_\ST$ 
be the map defined by the following composition
\[
  W_\ST\stackrel{\id_  \ST\ot\delta}{\loongrightarrow}   W_\ST \ot  \ST \stackrel{\id_  \ST\ot\tau_{W,  \ST}}{\looongrightarrow} 
  \ST\ot  \ST\ot W \stackrel{\mu_0\ot\id_W}{\loongrightarrow}  W_\ST,
\]
where $\mu_0$ is the multiplication of $\ST(V)$, and $\id_\ST=\id_{\ST(V)}$.

\begin{lemma}\label{lem:h-homog}
Retain notation above. 
\begin{enumerate}
\item The map $\eta$ is an  $\CA$-module isomorphism with  
the inverse map $\eta^{-1}:   W_\ST\lra    W_\ST$ defined by the composition
\[
  W_\ST\stackrel{\id_  \ST\ot\delta}{\loongrightarrow}   W_\ST \ot  \ST \stackrel{\id_ \ST\ot(S^{-1}\ot\id_W)\circ \tau_{W,  \ST}}{\looongrightarrow} 
  \ST\ot  \ST\ot W \stackrel{\mu_0\ot\id_W}\lra   \ST\ot W=W_\ST.
\]
\item The maps $\eta$ and $\eta^{-1}$ have the following properties.
\beq
&&(R_x\ot\id_W) \eta = \eta (R\ot \pi_W)\Delta(x),   \label{eq:to-free} \\
&&
\eta^{-1}(R_x\ot\id_W) =(R\ot \pi_W)\Delta(x)\eta^{-1},
\quad \forall x\in \U(\gl(V)). \label{eq:from-free}
\eeq
\end{enumerate}
\end{lemma}
\begin{proof}
Note that for any $f\in\ST$ and $v\in W$, we have 
\beq
\eta(f\ot v)= \sum_{(v)} \omega(d(v_{(1)}), d(v_{(2)})) f  v_{(2)} \ot v_{(1)}, \label{eq:eta}
\eeq
where we have used Sweedler's notation $\delta(v)=\sum_{(v)} v_{(1)}\ot v_{(2)}$. 
Also, the map defined by part (1) is given by 
\beq
\eta^{-1}(f\ot v)=\sum_{(v)} \omega(d(v_{(1)}), d(v_{(2)})) f  S^{-1}(v_{(2)}) \ot v_{(1)}. \label{eq:eta-inv}
\eeq
To show that it is indeed the inverse of $\eta$,  
let us temporarily write it as $\wt\eta$. 
Then
\[
\wt\eta(\eta(f\ot v))= \sum_{(v)} \omega(d(v_{(1)}), d(v_{(2)}))\wt\eta( f  v_{(2)} \ot v_{(1)}).
\]
Denote the right side by ${\rm RS}$. Using \eqref{eq:eta-inv}, we can express  it as 
\[
\baln
{\rm RS}&=\sum_{(v)} \omega(d(v_{(1)})+d(v_{(2)}), d(v_{(3)})) \omega(d(v_{(1)}), d(v_{(2)})) f  v_{(3)} S^{-1}(v_{(2)})\ot v_{(1)} \\
&=\sum_{(v)} \omega(d(v_{(1)}), d(v_{(3)})) \omega(d(v_{(1)}), d(v_{(2)})) f  S^{-1}(v_{(2)} S(v_{(3)} ))\ot    v_{(1)}.
\ealn
\]
Note that 
$\omega(d(v_{(1)}), d(v_{(3)})) \omega(d(v_{(1)}), d(v_{(2)})) = \omega(d(v_{(1)}), d(v)-d(v_{(1)}))$, and hence 
\[
\baln
{\rm RS}&=\sum_{(v)} \omega(d(v_{(1)}), d(v)-d(v_{(1)}))  f  S^{-1}(v_{(2)} S(v_{(3)} ))\ot v_{(1)}\\
&=\sum_{(v)} \omega(d(v_{(1)}), d(v)-d(v_{(1)}))  f \ot v_{(1)} \varepsilon(v_{(2)})\\
&= f \ot \sum_{(v)} v_{(1)} \varepsilon(v_{(2)}) = f\ot v.
\ealn
\]
Thus $\wt\eta(\eta(f\ot v))= f\ot v$.

Now
$
\eta(\wt\eta(f\ot v))=\sum_{(v)} \omega(d(v_{(1)}), d(v_{(2)})) \eta(f  S^{-1}(v_{(2)}) \ot v_{(1)}), 
$
where the right hand side is equal to 
\[
\baln
&\sum_{(v)} \omega(d(v_{(1)})+d(v_{(2)}), d(v_{(3)})) \omega(d(v_{(1)}), d(v_{(2)})) f  S^{-1}(v_{(3)}) v_{(2)} \ot v_{(1)} \\
&=\sum_{(v)} \omega(d(v_{(1)}), d(v_{(2)})+d(v_{(3)})) f  S^{-1}(S(v_{(2)})v_{(3)})\ot    v_{(1)}.
\ealn
\]
The same calculations as above show that this is equal to $f\ot v$. 
Hence $\eta(\wt\eta(f\ot v))= f\ot v$. This proves $\wt\eta=\eta^{-1}$ as we have claimed.

Let us consider $W_\ST$ as an $\CA$-module. Then by inspecting equations \eqref{eq:eta} and  \eqref{eq:eta-inv}, we easily see that $\eta$ is an $\CA$-module automorphism. 

This completes the proof of part (1).

To prove part (2), note that for any $x\in \U(\gl(V))$, we have 
\[
\baln
(R_x\ot\id_W)\eta&= (R_x\ot\id_W) (\mu_0\ot \id_W)(\id_\ST\ot\tau_{W, \ST})(\id_\ST\ot \delta).
 \ealn 
\]
As $\ST$ is an $R_\U$-module algebra (see Lemma \ref{lem:mod-alg}.(1)), the right hand side can be re-written as 
\[
\sum (\mu_0\ot \id_W)(R_{x_{(1)}}\ot R_{x_{(2)}} \ot \id_W) (\id_\ST\ot\tau_{W, \ST})(\id_\ST\ot \delta).
\]
Using the property \eqref{eq:PAB} of $\tau_{W, \ST}$, we can cast this to
\beq\label{eq:express}
\sum (\mu_0\ot \id_W)(\id_\ST\ot\tau_{W, \ST}) (R_{x_{(1)}} \ot \id_W \ot R_{x_{(2)}}) (\id_\ST\ot \delta).
\eeq
It follows the definition  of $R_x$ that 
\[
\baln
&\sum (R_{x_{(1)}} \ot \id_W \ot R_{x_{(2)}}) (\id_\ST\ot \delta)\\
 &=\sum  (R_{x_{(1)}} \ot \id_W \ot \id_\ST) (\id_\ST\ot \delta)(\id_\ST\ot \pi_W(x_{x_{(2)}}))\\
  &= \sum  (\id_\ST\ot \delta)(R_{x_{(1)}}\ot \pi_W(x_{(2)})).
\ealn 
\]
Hence 
\[
\baln
\text{\eqref{eq:express}}&=\sum (\mu_0\ot \id_W)(\id_\ST\ot\tau_{W, \ST})  (\id_\ST\ot \delta)(R_{x_{(1)}}\ot \pi_W(x_{(2)}))\\
&= \eta \sum R_{x_{(1)}}\ot \pi_W(x_{(2)}), 
\ealn
\]
where the second equality follows from the definition of $\eta$. 

This shows that $(R_x\ot\id_W)\eta= \eta \sum R_{x_{(1)}}\ot \pi_W(x_{(2)})$, which is 
equation  \eqref{eq:to-free}. 

Equation  \eqref{eq:from-free} follows \eqref{eq:to-free}, completing the proof of part (2). 
\end{proof}

The lemma above enables us to prove part (2) of the theorem below. 
\begin{theorem} \label{thm:proj}
Retain notation above. 
The following statements are true. 
\begin{enumerate}
\item $\CS(V^0)$ is an $\CA$-submodule of $\ST(V)\ot V^0$ (cf. \eqref{eq:A-act}).

\item Assume that $W$ is a right $\ST(V)$-co-module, then $\CS(W)\simeq \CA\ot_\C W$ as $\CA$-module. 

\item Assume that $V^0$ is a finite dimensional weight module for $\fh$ with weights belonging to 
$\fh^*_\Z=\sum_{\alpha\in \Gamma_R}\sum_{i=1}^{m_\alpha}\Z\varepsilon(\alpha)_i$. 
Then $\CS(V^0)$ is a finitely generated projective $\CA$-module.
\end{enumerate}
\end{theorem}

\begin{proof}
To prove part (1), note that for any $g\in \CA$ and $s=\sum_i f_i\ot v_i\in \CS$, 
\[
\baln
(R_h\ot\id_{V^0})  (g s) &= \sum_i R_h( g f_i)\ot v_i = \sum_i g R_h(f_i)\ot v_i \\
&=- (\id_{\ST(V)}\ot h) (g s), \quad \forall h\in \fh. 
\ealn
\] 
Hence $gs\in \CS(V^0)$ by \eqref{eq:homog}.

We now prove part (2) by adapting methods of \cite{GoZ, Z98}. 
Given any $\zeta\in \CS(W)$, it follows \eqref{eq:to-free} that 
\[
(R_h\ot\id_W)\eta(z) = \eta(R_h\ot\id_W+ \id_{\ST(V)}\ot\pi_W(h))\zeta=0, \quad \forall h\in\fh.  
\]
This shows that $\eta(z)\in \CA\ot W$, and 
hence $\eta(\CS(W))\subset \CA\ot W$. 

Also, if $\xi\in  \CA\ot W$, then by \eqref{eq:from-free}, 
\[
(R_h \ot \id_W + \id_{\ST(V)}\ot \pi_W(h))\eta^{-1}(\xi)  =\eta^{-1}(R_h\ot\id_W) = 0, \quad \forall h\in\fh.
\]
Hence $\eta^{-1}(\xi) \in \CS(W)$, and this shows that $\eta^{-1}(\CA\ot W)\subset \CS(W)$. Hence follows part(2) of the theorem. 

Part (3) is a consequence of part (2). 
Note that all finite dimensional weight modules for $\fh$ are semi-simple, in particular, 
right $\ST(V)$-co-modules restrict to semi-simple $\fh$-modules. 
Under the given conditions of $V^0$, there always exists a  right $\ST(V)$-co-module $W$ such that 
$W=V^0\oplus V^0_\top$ as $\fh$-module.  Thus $\CS(W)=\CS(V^0)\oplus \CS(V^0_\top)$. 
It immediately follows from part (2) that
$\CS(W)\simeq \CA\ot W$. Hence $\CS(V^0)$ is a finitely generated projective $\CA$-module. 
\end{proof}

\begin{remark}\label{rem:ncg}
In the general spirit of non-commutative geometry \cite{C, La},  
one may think of $\CA$ as the coordinate algebra of some non-commutative analogue of a flag variety, 
and finitely generated projective $\CA$-modules  as the spaces of sections of non-commutative vector bundles on the non-commutative flag variety. 
\end{remark}

\subsubsection{Constructing tensor representations}
We realises simple tensor modules on sections of the non-commutative line bundles alluded to in  Remark \ref{rem:ncg} 
by mimicking the classic Borel-Weil theorem. 
This will be done by adapting the techniques developed in \cite{GoZ, Z98} 
for quantum (super)groups to the present setting.  

We shall assume that $V^0$ is a module for the Borel subalgebra $\fb$, with the Cartan subalgebra $\fh\subset \fb$ acting semi-simply. Denote by $\pi^0: \U(\fb)\lra \End_\C(V^0)$ the associated $\U(\fb)$-representation. Then $\CS(V^0)$ admits an $\fn$-action 
$
\partial: \fn\times \CS(V^0)\lra \CS(V^0)$, $(X, \zeta)\mapsto \partial_X(\zeta), 
$
defined by 
\beq
\phantom{XXX}
\partial_X(\zeta)= (R_X\ot \id_{V^0} + \id_{\ST(V)}\ot \pi^0(X))\zeta, \quad \forall X\in\fn, \zeta\in \CS(V^0). 
\eeq

Consider the following subspaces of $\CS(V^0)$
{\small 
\[
\begin{tikzcd}
\CO(V^0)   \arrow[hook,dr]\\
&\Gamma(V^0)  \arrow[hook,r] &\CS(V^0).\\
\ol{\CO}(V^0)\arrow[hook,ur]
\end{tikzcd}
\]
}
which are defined by 
\[
\baln
\Gamma(V^0)=\CS(V^0)^\fn, \quad
\CO(V^0)= (T\ot V^0)^\fn, \quad
\ol{\CO}(V^0)=(\ol{T}\ot V^0)^\fn,
\ealn
\]
with respect to the action $\partial_\fn$ of $\fn$. 
Since $\partial_\fn$ commutes with the left $\U(\gl(V))$-action $L_{\U(\gl(V))}\ot\id_{V^0}$ 
(cf. Lemma \ref{lem:sects}.(2)), each subspace is a left $\U(\gl(V))$-module. 

The $\U(\gl(V))$-module structures of $\CO(V^0) $ and $\ol{\CO}(V^0)$ have a neat description for simple $\fb$-modules.  
If $V^0\ne 0$ is a simple $\U(\fb)$-module, then $\dim V^0=1$ and $\fn\cdot V^0=0$. There is some $\mu\in\fh^*$ such that $h\cdot v=\mu(h) v$ for all $h\in \fh$ and $v\in V^0$. In this case, we denote $V^0$ by $L^0_\mu$. 

Recall from Theorem \ref{thm:P} that a simple $\gl(V)$-module $L_\nu$ is finite dimensional if and only if its highest weight $\nu$ 
belongs to $\Lambda_{M_+|M_-}\subset \fh^*$. 
In this case,  the dual module $L_\nu^*$ is again a highest weight module, 
whose highest weigh will be denoted by $\nu^*$.

\begin{theorem} \label{thm:BW}
Let $L^0_\mu$ be a simple $\fb$-module. Then as $\U(\gl(V))$-modules, 
\beq
\CO(L^0_\mu)=\left\{
\begin{array}{l l}
L_{\lambda^\sharp}^*, & \text{if $\mu=-\lambda^\sharp$ for $\lambda\in P_{M_+|M_-}$},\\
0, &\text{otherwise},
\end{array}
\right.  \label{eq:CO}\\
\ol{\CO}(L^0_\mu)=\left\{
\begin{array}{l l}
L_{\lambda^\sharp}, & \text{if $\mu=-{\lambda^\sharp}^*$ for $\lambda\in P_{M_+|M_-}$},\\
0, &\text{otherwise}.
\end{array}
\right.
\eeq
\end{theorem}
\begin{proof}
Consider $\CO(L^0_\mu)$. It follows from the statements on $T$ in Theorem \ref{thm:P-W} that
\[
\baln
\CO(L^0_\mu) &= (T\ot L^0_\mu)^\fb \simeq  \bigoplus_{\lambda\in P_{M_+|M_-}} L_{\lambda^\sharp}^*\ot (L_{\lambda^\sharp}\ot L^0_\mu)^\fb\\
&= \bigoplus_{\lambda\in P_{M_+|M_-}} L_{\lambda^\sharp}^*\ot (L_{\lambda^\sharp}^\fn\ot L^0_\mu)^\fh.
\ealn
\]
Since $L_{\lambda^\sharp}^\fn$ is the space of highest weight vectors $(L_{\lambda^\sharp})_{\lambda^\sharp}$, which is $1$-dimensional, 
$
((L_{\lambda^\sharp})_{\lambda^\sharp}\ot L^0_\mu)^\fh=\left\{
\begin{array}{l l}
\C, & \text{if $\mu=-\lambda^\sharp$},\\
0, &\text{otherwise}.
\end{array}
\right.
$
This immediately leads to \eqref{eq:CO}.  

We can similarly prove the claim for $\ol{\CO}(L^0_\mu)$ by using Theorem \ref{thm:P-W}. 
\end{proof}

Note that an analogous result  was known \cite{Z98, Z04} for the general linear quantum supergroup at generic $q$.

\subsection{Remarks on a group functor}\label{sect:group}

Denote by $\Alg$ the category of finitely generated graded commutative associative 
$(\Gamma, \omega)$-algebras with identity over $\C$. Let $\Set$ be the category of sets.
\begin{definition}
Given any graded commutative Hopf $(\Gamma, \omega)$-algebra $H$, define the functor
\beq
G=\Hom_\Alg(H,  -): \Alg\lra \text{\bf Set}.
\eeq  
\end{definition}
Recall that  $(\Gamma, \omega)$-algebra homomorphisms 
are homogeneous of degree $0$.  

Because of the  graded commutativity all objects of $\Alg$, 
the Hopf $(\Gamma, \omega)$-algebra structure
of $H$ induces the following natural transformations.
\beq
&&\mu_0:  G \times G  \lra G , \quad  \mu_0(\varphi, \psi)(f)= \sum_{(f)}\varphi(f_{(1)}) \psi(f_{(2)}), \\
&&S_0: G \lra G , \quad  S_0(\varphi)(f)= \varphi(S(f)),\quad 
\forall f\in H,  
\eeq
where Sweedler's notation $\Delta(f)=\sum_{(f)}f_{(1)} \ot f_{(2)}$ is 
used for the co-product of $f$.  
Let us prove that $\Im(\mu_0)$ and $\Im(S_0)$ are indeed contained in $G$. 
\begin{proof}
Let $\varphi,  \psi\in G$. For all $f, g\in H$, we have 
\[
\baln
S_0(\varphi)(f g)&= \varphi(S(f g))=\omega(d(f), d(g)) \varphi(S(g) S(f)) \\
&= \omega(d(f), d(g)) \varphi(S(g)) \varphi(S(f)). 
\ealn
\]
Because of the graded commutativity of objects in $\Alg$, 
the last line of the above equation can be re-written as 
\[
\baln
\varphi(S(f)) \varphi(S(g))  =S_0(\varphi)(f) S_0(\varphi)(g).
\ealn
\]
Hence $S_0(\varphi)(f g)=S_0(\varphi)(f) S_0(\varphi)(g)$, proving that $S_0(\varphi)\in G$. 

Using
$\Delta(fg)=\sum \omega( d(f_{(2)}),  d(g_{(1)})) f_{(1)} g_{(1)} \ot f_{(2)} g_{(2)}$, 
we obtain
\[
\mu_0(\varphi, \psi) (f g)=\sum \omega( d(f_{(2)}),  d(g_{(1)})) \varphi(f_{(1)})  \varphi(g_{(1)})\psi(f_{(2)}) \psi(g_{(2)}).
\]
Again by the graded commutativity of objects in $\Alg$, we can re-write this as 
\[
\baln
\sum \varphi(f_{(1)})  \psi(f_{(2)}) \varphi(g_{(1)}) \psi(g_{(2)}) 
=\mu_0(\varphi, \psi)(f) \mu_0(\varphi, \psi)(g), 
\ealn
\]
Hence $\mu_0(\varphi, \psi) (f g)=\mu_0(\varphi, \psi)(f) \mu_0(\varphi, \psi)(g)$, 
proving that $\mu_0(\varphi, \psi)\in G$.
\end{proof}

The following result generalises to the $\Gamma$-graded setting a fundamental fact in the theories of algebraic groups (see, e.g., \cite[\S4]{A}) and algebraic supergroups (see, e.g., \cite[\S 10, 11]{CCF}).
Its proof is very easy.
\begin{lemma} The set $G$ forms a group 
with multiplication $\mu_0$, inverse map $S_0$, and identity $\varepsilon$ (the co-unit of $H$). 
\end{lemma}
\begin{proof}
It immediately follows from the co-associativity of the co-multiplication of $H$ 
that the map $\mu_0$ is associative. Also for any $\varphi\in G$,  the co-unit property leads to 
$\mu_0(\varepsilon, \varphi)= \mu_0(\varphi, \varepsilon)=\varphi$, 
and the defining property of the antipode implies 
 $\mu_0(S_0(\varphi), \varphi)= \mu_0(\varphi, S_0(\varphi))=\varepsilon$.  

This completes the proof of the lemma. 
\end{proof}

In particular,  for $H=\ST(V(\Gamma, \omega))$, we have the group functor
\[
\GL(V(\Gamma, \omega))=\Hom_\Alg(\ST(V(\Gamma, \omega)),  -): \Alg\lra \text{\bf Set}.  
\] 
For any $R\in \Alg$,  we have the group
\beq\label{eq:grp}
\GL(V(\Gamma, \omega); R):= \Hom_\Alg(\ST(V(\Gamma, \omega)),  R).
\eeq

Thus the coordinate algebra $\ST(V(\Gamma, \omega))$ leads to a  
group functor,  which is some kind of  $\Gamma$-graded ``general linear group''. 
Here we merely wish to point out this fact; 
any in-depth investigation is beyond the scope of this paper. 

\end{document}